\documentclass[twoside,11pt]{article}

\usepackage{caption}
\usepackage[preprint]{jmlr2e_edited}

\usepackage{amsmath}
\usepackage{amssymb}
\usepackage{thmtools}
\usepackage{enumitem}
\setlist[enumerate]{leftmargin=.5in}
\setlist[itemize]{leftmargin=.5in}

\graphicspath{{figures/Drawings/}}
\usepackage[percent]{overpic}

\usepackage[normalem]{ulem}
\useunder{\uline}{\ul}{}

\usepackage{tikz}
\usetikzlibrary{decorations.pathreplacing,matrix,positioning,calc}
\tikzset{
mybrace/.style={
  decorate,
  decoration={brace,aspect=#1},
  line width=1pt
  }
}

\usepackage{multirow}

\usepackage[capitalize]{cleveref}

\usepackage{algorithm}
\usepackage{algpseudocode}
\algnewcommand{\LineComment}[1]{\State \(\triangleright\) #1}
\algrenewcommand\algorithmicindent{1em}%

\usepackage{appendix}

\makeatletter
\def\cleartheorem#1{%
    \expandafter\let\csname#1\endcsname\relax
    \expandafter\let\csname c@#1\endcsname\relax
}
\makeatother

\cleartheorem{example}
\cleartheorem{theorem}
\cleartheorem{lemma}
\cleartheorem{proposition}
\cleartheorem{remark}
\cleartheorem{corollary}
\cleartheorem{definition}
\cleartheorem{conjecture}
\cleartheorem{axiom}

\newtheorem{theorem}{Theorem}
\newtheorem{example}[theorem]{Example} 
\newtheorem{lemma}[theorem]{Lemma} 
\newtheorem{proposition}[theorem]{Proposition} 
\newtheorem{remark}[theorem]{Remark}
\newtheorem{corollary}[theorem]{Corollary}
\newtheorem{definition}[theorem]{Definition}

\newtheorem{notation}[theorem]{Notation}
\newtheorem{resultx}{Result}

\newcommand{\bbN}{\mathbb{N}}
\newcommand{\bbR}{\mathbb{R}}

\renewcommand{\to}{\rightarrow}

\newcommand{\iso}{\cong}

\renewcommand{\Im}{\mathrm{Im}}

\newcommand{\op}{\mathrm{op}}
\newcommand{\rk}{\mathrm{rk}}

\newcommand{\SET}{X}
\newcommand{\Met}{M}
\newcommand{\Metalt}{N}
\newcommand{\MPS}{\mathcal{M}}
\newcommand{\MPSalt}{\mathcal{N}}

\newcommand{\define}[1]{\textbf{\boldmath{#1}}}
\newcommand{\calF}{\mathcal{F}}
\newcommand{\C}{\mathsf{C}}
\newcommand{\PD}{\mathcal{B}}

\newcommand{\PC}{\mathsf{PC}}

\newcommand{\dGHP}{d_{\mathrm{GHP}}}
\newcommand{\dGH}{d_{\mathrm{GH}}}
\newcommand{\dH}{d_{\mathrm{H}}}
\newcommand{\dP}{d_{\mathrm{P}}}
\newcommand{\dI}{d_{\mathrm{I}}}
\newcommand{\extension}[1]{\bar{#1}}
\newcommand{\mergefun}[1]{\theta_{#1}}
\newcommand{\distortion}{\mathrm{dis}}

\newcommand{\dWI}{d_{\mathrm{CI}}}
\newcommand{\dB}{d_{\mathrm{B}}}

\newcommand{\DR}{\mathrm{DR}}

\newcommand{\SL}{\mathrm{SL}}
\newcommand{\Rips}{\mathrm{R}}
\newcommand{\RL}{\mathrm{RSL}}
\newcommand{\RG}{\mathrm{PI}}
\newcommand{\cov}{_{\mathrm{cov}}}
\newcommand{\con}{_{\mathrm{con}}}
\newcommand{\Uniform}{\mathrm{Uni}}
\newcommand{\mml}{\mathrm{L}}

\newcommand{\height}{\eta}

\newcommand{\upar}[1]{_{\left[#1\right]}}

\newcommand{\bbRs}{\mathbb{R}_{>0}} 

\newcommand{\conv}[3]{\left(#1\ast #2_{#3}\right)}

\newcommand{\bfC}{\mathbf{C}}
\newcommand{\bfD}{\mathbf{D}}
\newcommand{\bfE}{\mathbf{E}}
\newcommand{\bfA}{\mathbf{A}}

\newcommand{\parts}{\mathcal{P}}
\newcommand{\linkage}[1]{#1\textup{-\texttt{link}}}
\newcommand{\theoremlinkage}[1]{#1\textup{-\texttt{link}}}

\newcommand{\birth}{\mathsf{birth}}

\newcommand{\death}{\mathsf{death}}
\newcommand{\dd}{\,\mathsf{d}}
\newcommand{\life}{\mathsf{life}}

\newcommand{\leaves}{\mathsf{leaves}}
\newcommand{\PF}{\mathsf{PF}}
\newcommand{\gapindex}{n}

\newcommand{\support}{\mathcal{S}}

\newcommand{\length}{\mathsf{length}}

\newcommand{\inter}{\mathsf{TOI}}
\newcommand{\fieldk}{\mathbb{F}}
\newcommand{\kvec}{\fieldk\text{-}\mathbf{vec}}

\renewcommand{\epsilon}{\varepsilon}

\newcommand{\barc}{\mathcal{B}}
\newcommand{\barA}{B}

\newcommand{\prom}{\mathsf{Pr}}
\newcommand{\gap}{\mathsf{gap}}
\newcommand{\gapsize}{\mathsf{gapsize}}

\newcommand{\twC}{0.49} 

\usepackage{lastpage}
\jmlrheading{25}{2024}{1-\pageref{LastPage}}{10/21; Revised 7/24}{8/24}{21-1185}{Alexander Rolle and Luis Scoccola}
\ShortHeadings{Stable and Consistent Density-Based Clustering}{Rolle and Scoccola}

\firstpageno{1}

\begin{document}

\title{Stable and Consistent Density-Based Clustering via Multiparameter Persistence}

\author{\name Alexander Rolle \email alexander.rolle@tum.de \\
       \addr Department of Mathematics\\
       Technical University of Munich\\
       Boltzmannstra{\ss}e 3, 85748 Garching, Germany
       \AND
       \name Luis Scoccola \email luis.scoccola@maths.ox.ac.uk \\
       \addr Mathematical Institute\\
       University of Oxford\\
       Woodstock Road,
       Oxford
       OX2 6GG,
       United Kingdom}

\editor{Sivan Sabato}

\maketitle

\begin{abstract}
We consider the degree-Rips construction from topological data analysis, 
which provides a density-sensitive, multiparameter hierarchical clustering algorithm. 
We analyze its stability to perturbations of the input data using the correspondence-interleaving distance, 
a metric for hierarchical clusterings that we introduce. 
Taking certain one-parameter slices of degree-Rips recovers well-known methods for density-based clustering, 
but we show that these methods are unstable. 
However, we prove that degree-Rips, as a multiparameter object, is stable, 
and we propose an alternative approach for taking slices of degree-Rips, 
which yields a one-parameter hierarchical clustering algorithm with better stability properties. 
We prove that this algorithm is consistent, 
using the correspondence-interleaving distance. 
We provide an algorithm for extracting a single clustering from one-parameter hierarchical clusterings, 
which is stable with respect to the correspondence-interleaving distance. 
And, we integrate these methods into a pipeline for density-based clustering, which we call Persistable. 
Adapting tools from multiparameter persistent homology, we propose visualization tools that guide the selection of all parameters of the pipeline. 
We demonstrate Persistable on benchmark data sets, 
showing that it identifies multi-scale cluster structure in data.
\end{abstract}



\setcounter{tocdepth}{2}

\tableofcontents


\addtocontents{toc}{\protect\setcounter{tocdepth}{1}}
\setcounter{table}{1}

\section{Introduction}

Let $f : \bbR^d \to \bbR$ be a probability density function, 
and let $\support(f)$ be its support. 
There is a one-parameter hierarchical clustering $H(f)$ of $\support(f)$ where, 
for $r > 0$, $H(f)(r)$ is the set of connected components of 
$\{ x \in \support(f) : f(x) \geq r \}$. 
This is hierarchical in the sense that, 
if $r < r'$, then $H(f)(r)$ is a refinement of $H(f)(r')$. 
Following \cite{hartigan-75}, 
we call $H(f)$ the \emph{density-contour hierarchical clustering}. 
The central theoretical problem of density-based clustering is to approximate $H(f)$, 
given finite samples drawn from $f$.

A large amount of work has been done on the related problem 
of estimating the density $f$ itself, given a finite sample. 
If one constructs an estimate $\hat{f}$ from a sample $X$, 
the ``plug-in'' approach would be to estimate $H(f)(r)$ by $H(\hat{f})(r)$, 
however this is not computationally-tractable 
(see \cite{chaudhuri-dasgupta-10}). 
Instead, \cite{Cuevas-Febrero-Fraiman} propose to 
construct a graph on $X$ that encodes distance relations, 
and then estimate $H(f)(r)$ by taking the connected components 
of the induced subgraph on the vertices $\{x \in X : \hat{f}(x) \geq r \}$. 
The graph is the \emph{Rips graph} for a fixed distance scale: 
for $x,y \in X$, there is an edge between $x$ and $y$ if $||x-y|| \leq s$, for some fixed $s > 0$. 
We call this approach the \emph{plug-in} algorithm. 
See Related Work, below, for further references for this idea.

Another popular approach to density-based clustering 
is the \emph{robust single-linkage} algorithm of \cite{chaudhuri-dasgupta-10}. 
This is a density-sensitive modification of the single-linkage algorithm. 
Chaudhuri--Dasgupta prove that this method is \textit{Hartigan consistent}: 
as the size of the sample tends to infinity, 
the robust single-linkage of a sample of $f$ converges in probability to $H(f)$, 
using a criterion of Hartigan to 
compare the density-contour hierarchical clustering 
with a hierarchical clustering produced from a sample. 

\cite{mcinnes-healy} observed that the robust single-linkage algorithm 
is closely connected to the degree-Rips bifiltration 
\citep{lesnick-wright, blumberg-lesnick-2param-published} 
from topological data analysis (TDA). 
Degree-Rips should be of great interest to researchers in the field of clustering, 
as it simultaneously generalizes several important methods for density-based clustering. 
In its original formulation, degree-Rips is a two-parameter filtration of simplicial complexes, 
but in the setting of clustering, only the underlying graphs are relevant. 
In detail, let $\Met$ be a finite metric space, 
let $s > 0$, and let $k \in (0,1)$. 
Define a graph $G_{s,k}$ with vertex set 
$\{x \in M : |B(x,s)| \geq k \cdot |\Met| \}$, 
and with an edge between $x$ and $y$ if $d_{\Met}(x,y) \leq s$. 
Here, $B(x,s)$ is the open ball in $\Met$ of radius $s$ centered at $x$. 
These graphs form a two-parameter filtration, 
in the sense that there is an inclusion $G_{s,k} \subseteq G_{s',k'}$ 
for any $s' \geq s$ and any $k' \leq k$. 
We say that the \emph{degree-Rips hierarchical clustering} of $\Met$ 
is the two-parameter hierarchical clustering $\DR(M)$ 
with $\DR(M)_{s,k}$ given by the connected components of the graph $G_{s,k}$. 
See \cref{DR-intro-figure}.

\begin{figure}
    \begin{center}
    \includegraphics[width=\textwidth]{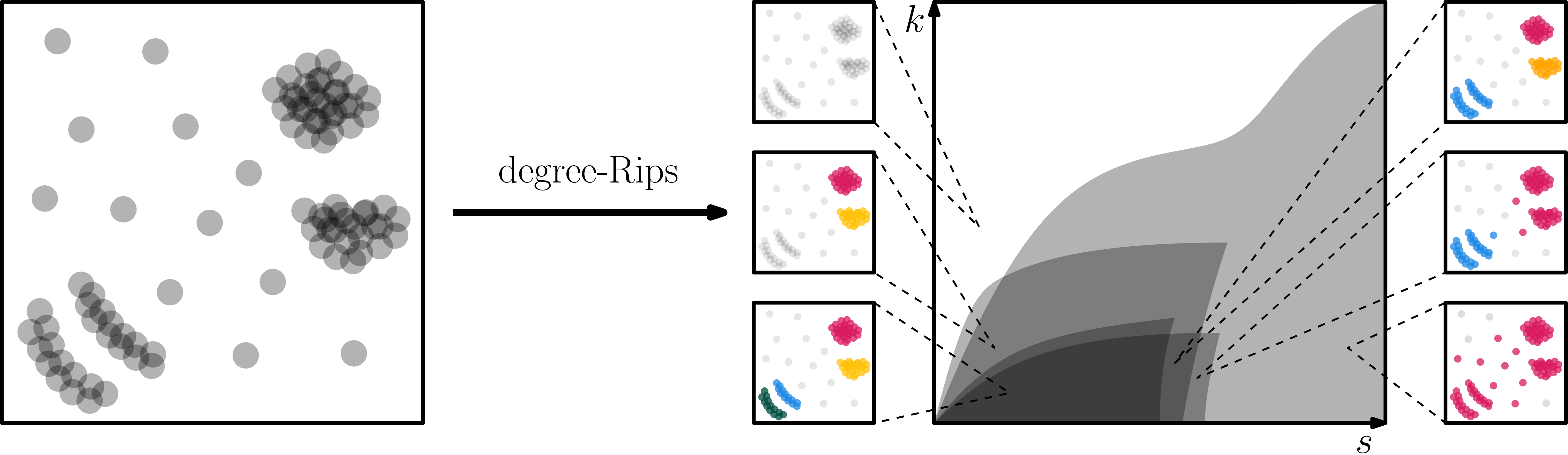}
    \end{center}
    \caption{Degree-Rips is a two-parameter hierarchical clustering. 
    The spatial parameter~$s$ controls the distance at which points are joined together: 
    at larger values of $s$, more points are joined together. 
    The density parameter $k$ controls when data points enter the hierarchical clustering: 
    at larger values of $k$, points must be in denser regions to enter. 
    The robust single-linkage algorithm fixes $k$ and lets $s$ vary 
    (taking a horizontal slice). 
    The plug-in algorithm fixes $s$ and lets $k$ vary 
    (taking a vertical slice).}
    \label{DR-intro-figure}
\end{figure}

Both the robust single-linkage algorithm and the plug-in algorithm 
can be seen as one-parameter \emph{slices} of the degree-Rips hierarchical clustering: 
if we fix $k$ and let $s$ vary, we recover the robust single-linkage of $\Met$; 
if we fix $s$ and let $k$ vary, we recover the plug-in algorithm, 
where the density estimate $\hat{f}$ is a kernel density estimate 
computed with the uniform kernel and bandwidth $s$, 
and the Rips graph is constructed also with parameter $s$. 


Furthermore, the degree-Rips hierarchical clustering recovers the popular DBSCAN clustering algorithm 
\citep{dbscan}. 
The clustering $\DR(\Met)(s,k)$ is exactly the DBSCAN* clustering of $\Met$ 
with respect to the spatial parameter $s$ and 
the number-of-neighbors parameter $\lceil k \cdot |\Met| \rceil$. 
DBSCAN* is a minor modification of the original DBSCAN algorithm, defined by \cite{campello-moulavi-sander}.

For this paper, an important observation is that both robust single-linkage and the plug-in algorithm 
are \emph{unstable}: small perturbations of the input can lead to large changes in the output. 
We make this statement precise later in the introduction. 
We therefore consider an alternative, which is very natural from the perspective of TDA. 
Rather than use slices of degree-Rips in which one parameter is fixed, 
we use slices in which both parameters vary.

We now summarize the main contributions of the paper. 
We elaborate on each point in the remainder of the introduction.

\begin{itemize}
	\item We introduce the \emph{correspondence-interleaving distance}, 
	a metric for hierarchical clusterings.
	\item We introduce \emph{kernel linkage}, 
	a density-sensitive, multiparameter hierarchical clustering method 
	that generalizes the degree-Rips hierarchical clustering described above.
	\item We prove that kernel linkage is stable with respect to the correspondence-interleaving distance 
	and the Gromov--Hausdorff--Prokhorov distance on compact metric probability spaces. 
	This implies that degree-Rips is stable, 
	and that appropriate slices of kernel linkage and degree-Rips are also stable. 
	\item We define a notion of consistency for density-based clustering 
	using the cor\-re\-spon\-dence-interleaving distance, which implies Hartigan consistency. 
	We prove that taking appropriate slices of kernel linkage is consistent in this sense.
	\item We define the \emph{persistence-based flattening algorithm}, 
	which extracts a single clustering of the underlying data from a one-parameter hierarchical clustering, 
	and prove that it is stable with respect to the correspondence-interleaving distance.
	\item Persistable is a pipeline for density-based clustering 
	that integrates the algorithms defined in this paper. 
	The Gromov--Hausdorff--Prokhorov stability theorem for kernel linkage implies 
	theoretical guarantees for the entire pipeline, 
	and it justifies a simple approximation scheme that makes it possible to apply the pipeline to large data sets. 
	We describe how the design choices of Persistable are motivated by the results of this paper, 
	we demonstrate Persistable on benchmark data sets, 
	and we show that it identifies meaningful cluster structure in data. 
	In another publication \citep{scoccola-rolle-joss}, we described the implementation of Persistable.
\end{itemize}


\subsection{The Correspondence-Interleaving Distance} 
In order to consider stability questions for hierarchical clustering methods, 
a natural approach is to use a notion of distance between hierarchical clusterings. 
For example, this is the approach taken by \cite{carlsson-memoli-hierarchical-clustering}, 
who prove a stability result for the single-linkage algorithm 
using the Gromov--Hausdorff distance from metric geometry. 
This is possible because the single-linkage of a metric space $X$ 
defines an ultrametric $\theta_X$ on $X$, 
and so one can compare the outputs of single-linkage on $X$ and $Y$ 
by comparing $(X, \theta_X)$ and $(Y, \theta_Y)$ using Gromov--Hausdorff.

However, a hierarchical clustering of $X$ does not define an ultrametric on $X$ unless it is quite special 
(in which case we call it an \emph{ultrametric hierarchical clustering}, \cref{mergefun}). 
In this paper, we formalize the notion of multiparameter hierarchical clustering 
in a way that is analogous to the multiparameter persistence modules from TDA 
\citep{carlsson-zomorodian-multidimensional}. 
We adapt the notion of interleaving from TDA \citep{chazal-etal-interleavings} to this setting, 
and use it to define the correspondence-interleaving distance ($\dWI$) 
between multiparameter hierarchical clusterings (\cref{CI-definition}), 
which generalizes the Gromov--Hausdorff distance on ultrametric hierarchical clusterings 
(\cref{CI-and-GH}).

\subsection{Stability} 
%
There are some choices baked in to the definition of degree-Rips 
that may not be optimal for some applications. 
So, we define a generalization: kernel linkage. 
Degree-Rips estimates the density of the data at a point $x$ 
by counting the number of data points in a ball centered at $x$. 
From the perspective of density estimation, this can be seen as integrating the uniform kernel 
against the uniform measure defined by the input. 
One could just as well use other kernels for estimating density, and kernel linkage allows for this. 
It is also convenient to let kernel linkage take any compact metric probability space as input; 
if the input is a finite metric space as before, one gives it the uniform probability measure.

Our stability theorem for kernel linkage (\cref{continuity-of-mml}) says that kernel linkage 
is uniformly continuous with respect to the Gromov--Hausdorff--Prokhorov distance 
on compact metric probability spaces, and the correspondence-interleaving distance on hierarchical clusterings. 
We note that one can replace the Prokhorov distance with the Wasserstein distance and get the same stability theorem for kernel linkage (\cref{GHW-stability}). 
In the special case of degree-Rips, our stability theorem is as follows:

\begin{resultx}[{\cref{DR-stability}}] \label{intro-DR-stability}
If $\Met$ and $\Metalt$ are finite metric spaces, then 
\[
	\dWI(\DR(\Met), \DR(\Metalt)) \leq 2 \cdot \dGHP(\Met, \Metalt) \, .
\]
\end{resultx}

Requiring two finite metric spaces to be close in the Gromov--Hausdorff--Prokhorov distance 
amounts to requiring that they be close in the Gromov--Hausdorff distance 
(so that their metric geometry is similar), 
and that they be close in the Gromov--Prokhorov distance (so that their uniform measures are similar). 
We use this distance for our stability theorem because 
degree-Rips fails to be continuous with respect to the Gromov--Hausdorff distance or the Gromov--Prokhorov distance (see \cref{choice-of-GHP}). 
In order to get a continuity result, one must combine these two kinds of restrictions on the input.

We regard the use of Gromov--Hausdorff--Prokhorov as a strong assumption. 
But, it leads to correspondingly strong conclusions 
(uniform continuity in the case of kernel linkage, and Lipschitz-continuity in the special case of degree-Rips). 
It is useful to know the conditions that lead to these conclusions. 
For example, a key consequence of our stability theorem is a simple subsampling approximation algorithm for degree-Rips 
(see \cref{subsampling-approximation}).  

\begin{figure}
    \begin{center}
    \includegraphics[width=\textwidth]{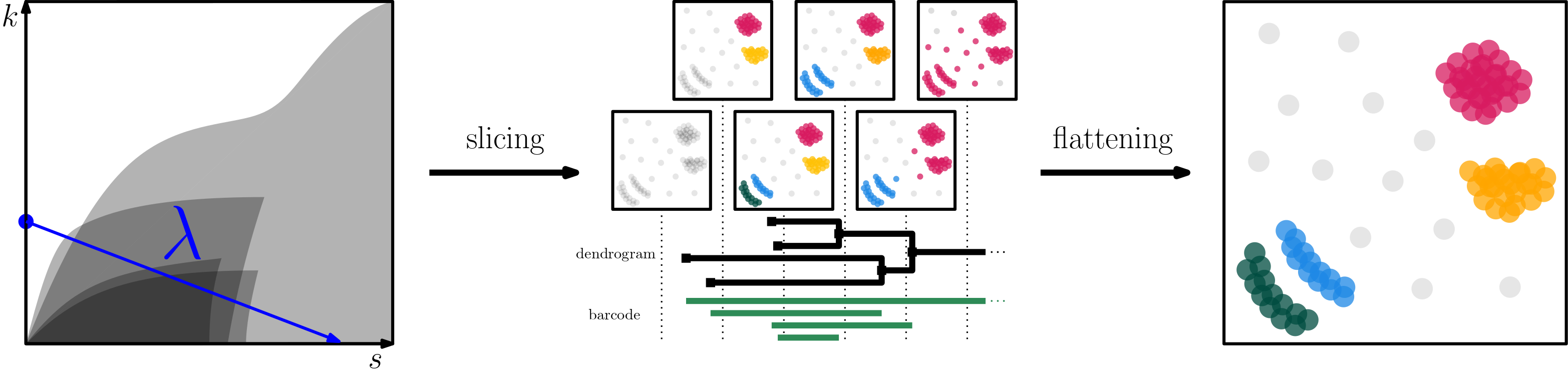}
    \end{center}
    \caption{Restricting degree-Rips (\cref{DR-intro-figure}) to a line $\lambda$ with negative slope 
    in the degree-Rips parameter space gives a one-parameter hierarchical clustering 
    we call $\lambda$-linkage. 
    The line $\lambda$ lets both of the degree-Rips parameters $s$ and $k$ vary, 
    in contrast to horizontal or vertical lines 
    (see \cref{DR-intro-figure}). 
    This allows $\lambda$-linkage to capture multi-scale cluster structure in data, 
    and it leads to better stability properties. 
    The barcode is a visualizable summary of a one-parameter hierarchical clustering. 
    The persistence-based flattening algorithm extracts a single clustering of the underlying data, 
    guided by the barcode.}
    \label{slicing-intro-figure}
\end{figure}

The Gromov--Hausdorff--Prokhorov stability of degree-Rips is in contrast to the robust single-linkage algorithm and the plug-in algorithm from above, 
which are discontinuous with respect to the Gromov--Hausdorff--Prokhorov distance, 
as we show in \cref{instability-related-methods}. 
In TDA, a standard method for extracting information from a two-parameter persistence module 
is to take one-parameter slices (see Related Work, below, for references). 
However, one usually takes slices by lines through the parameter space that do not fix either of the parameters. 
Slices in which both parameters vary have two key advantages. 
First, they are multi-scale: they capture information across a range of values of both parameters. 
Second, these slices have better stability properties, since interleavings between 
multiparameter persistence modules restrict to interleavings between these slices. 

The situation is completely analogous in the setting of hierarchical clustering. 
So, rather than use robust single-linkage or the plug-in algorithm for density-based clustering 
(which correspond to using horizontal or vertical slices of degree-Rips), 
we propose using slices of degree-Rips in which both parameters vary. 
In more detail, given a line $\lambda$ in the plane with negative slope, 
restricting $\DR(\Met)$ to $\lambda$ gives a one-parameter hierarchical clustering, 
which we call $\lambda$-linkage, denoted $\linkage{\lambda}(M)$ 
(see \cref{slicing-intro-figure}). 
In contrast to robust single-linkage and the plug-in approach, 
$\lambda$-linkage is multi-scale, and it is stable with respect to the Gromov--Hausdorff--Prokhorov distance: 
as an immediate corollary of \cref{intro-DR-stability}, 
we obtain the following stability result.

\begin{resultx}[{\cref{stability-lambda-linkage-in-X}}]
Let $\lambda$ be a line in the plane with slope $\sigma < 0$. 
If $\Met$ and $\Metalt$ are finite metric spaces, then 
\[
	\dWI(\linkage{\lambda}(\Met), \linkage{\lambda}(\Metalt)) \leq 
	\max(2 \vert \sigma \vert, 1) \cdot \dGHP(\Met, \Metalt) \, .
\]
\end{resultx}

\subsection{Consistency} 
Roughly speaking, a ``consistency result'' for density-based clustering usually says that, 
given a density function $f$ and an algorithm for computing hierarchical clusterings of finite samples drawn from $f$, 
the output of the algorithm converges in probability to the density-contour hierarchical clustering $H(f)$, 
as the sample size goes to infinity. 
To make this precise, one needs to specify what it means to converge in this context. 
There is a natural notion of consistency associated to the correspondence-interleaving distance, 
which we call CI-consistency; 
the idea is that the output of the algorithm should converge to $H(f)$ in the correspondence-interleaving distance, 
though in fact we require slightly more than this. 
CI-consistency is stronger than Hartigan consistency, 
so proving that an algorithm is CI-consistent implies that it is also Hartigan consistent. 
While this notion of consistency is novel, we remark that 
CI-consistency is similar in spirit to the notion of consistency of \cite{eldridge-belkin-wang}. 
We prove the following consistency result for $\linkage{\lambda}$. 
In the statement, the notation $\overline{\lambda}$ indicates that the slice $\linkage{\lambda}$ 
has been re-parameterized, using an explicit re-parameterization that only depends on $\lambda$; 
this can be dropped when considering Hartigan consistency, 
since Hartigan consistency is agnostic to the choice of parameterization.

\begin{resultx}[{\cref{consistency-lambda-linkage}}]
The hierarchical clustering algorithm $\theoremlinkage{\overline{\lambda}}$ 
is CI-consistent with respect to any continuous, compactly supported probability density function. 
In particular, $\theoremlinkage{\lambda}$ 
is Hartigan consistent with respect to any such density function. 
\end{resultx}

\subsection{Flattening a Hierarchical Clustering} 
For many applications, one needs a clustering of the input data, not a hierarchical clustering. 
We say that a \emph{flattening} algorithm takes a hierarchical clustering, and returns a single clustering. 
An example of such a flattening algorithm is the ToMATo clustering algorithm 
\citep{chazal-guibas-oudot-skraba}, 
which computes a flattening of the hierarchical clustering induced by a filtered graph. 
A major advantage of ToMATo is that its output can be understood in terms of the \emph{barcode} of the input hierarchical clustering. 
Barcodes are key tools in TDA 
\citep{edelsbrunner-letscher-zomorodian, carlsson-barcodes-shapes, ghrist-barcodes}; 
in this case, the barcode is a visualizable summary of 
the structure of a one-parameter hierarchical clustering (see \cref{slicing-intro-figure}). 
On a technical level however, a disadvantage of ToMATo is that its output depends on a choice of ordering of the vertices in the input graph, 
and in some use cases there may not be a clear way to make this choice. 

We define the \emph{persistence-based flattening algorithm} (\cref{persistence-based-flattening-def}), 
an adaptation of the ToMATo algorithm that avoids the dependence on an ordering of the input. 
And, we prove that it is stable with respect to the correspondence-interleaving distance 
(\cref{stability-tomato-flattening}).

\begin{figure}
    \begin{center}
    \includegraphics[width=\textwidth]{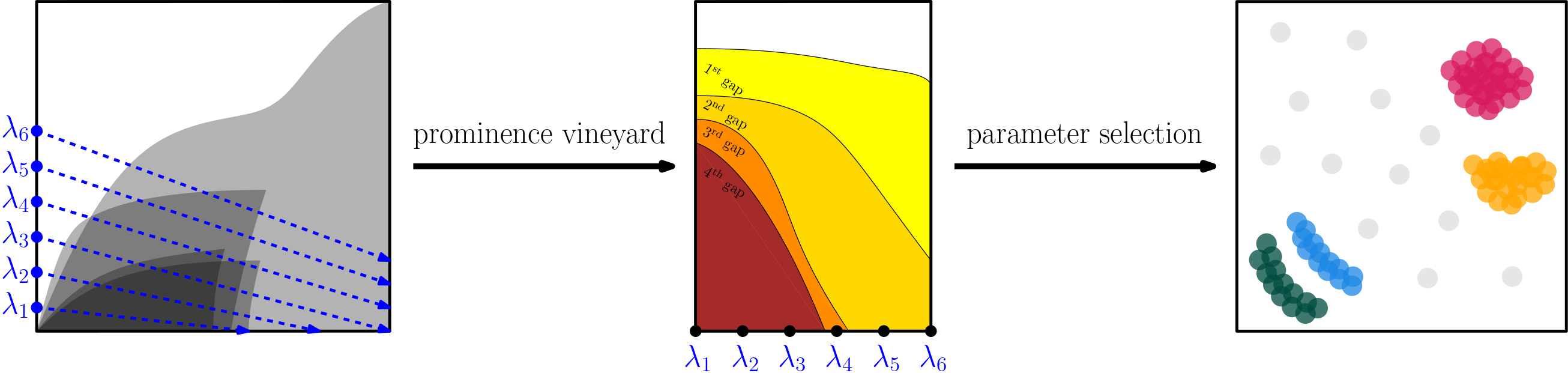}
    \end{center}
    \caption{Parameter selection in the Persistable pipeline. 
    The practitioner can choose a slice $\lambda$ using the prominence vineyard. 
    For each $\lambda$ in a chosen family, the prominence vineyard plots the length 
    of each bar in the barcode of $\lambda$-linkage. 
    As $\lambda$ varies continuously, the barcode varies continuously as well. 
    The lengths of the bars trace out continuous curves: 
    the top curve shows the length of the longest bar in the barcode of each slice, 
    the second curve shows the length of the second longest, etc. 
	Choosing a gap in the prominence vineyard leads to a clustering of the data. 
	Larger gaps lead to more stable results.}
    \label{vineyard-intro-figure}
\end{figure}

\subsection{Persistable} 
Combining the hierarchical clustering algorithm $\linkage{\lambda}$ and the persistence-based flattening algorithm, 
we obtain a pipeline for density-based clustering with good stability properties. 
We call this pipeline \emph{Persistable}.

Our stability theorems for degree-Rips and the persistence-based flattening algorithm 
imply theoretical guarantees for the entire pipeline 
(\cref{stability-lambda-and-PF} and \cref{stability-lambda-and-PF-in-lambda}). 
The stability of degree-Rips also justifies a simple approximation scheme that makes it possible to apply Persistable to large data sets 
(e.g., the rideshare data in \cref{Persistable-on-benchmark-data}). 
This approximation scheme is not valid for related methods that are not Gromov--Hausdorff--Prokhorov stable, such as HDBSCAN \citep{campello-moulavi-sander} and DBSCAN \citep{dbscan}. 

Persistable includes interactive visualization tools that practitioners can use to choose all parameters in the pipeline. 
The key task for the practitioner is to choose the slice~$\lambda$. 
Using a vineyard \citep{cohen-steiner-edelsbrunner-morozov}, 
one can see how the barcode of $\linkage{\lambda}$ changes with the choice of $\lambda$ 
(see \cref{vineyard-intro-figure}). 
Moreover, one can see from the vineyard which choices of $\lambda$ 
lead to particularly stable clusterings of the input data. 
We demonstrate Persistable, and this approach to parameter selection, on benchmark data sets, 
and we show that it provides results that capture meaningful cluster structure. 
These examples also demonstrate that Persistable can identify multi-scale cluster structure 
that is challenging for related algorithms, such as HDBSCAN.

\begin{table}
\centering
{\footnotesize
\begin{tabular}{ |l ||l |  }
\hline
HC & Abbreviation for ``hierarchical clustering'' \\

Poset of clusterings $\C(X), \preceq$ & \cref{poset-of-clusterings-def} \\

The opposite $P^{\op}$ of a poset $P$ & \cref{poset-def} \\

$P$-HC of $X$, $H : P \to \C(X)$ (with $P$ a poset) & \cref{P-HC-def} \\

$n$-parameter HC & \cref{n-param-HC} \\


Density-contour HC $H(f)$ of a density $f$ & \cref{density-contour} \\

Single-linkage $\SL(\Met)$ of a metric space $\Met$ & \cref{single-linkage-def} \\

Extension $\extension{H}$ of a HC $H$ & \cref{extension} \\

Two HCs are $\vec{\epsilon}$-interleaved & \cref{epsilon-interleaving} \\

The interleaving distance $\dI$ & \cref{interleaving-distance} \\

$\mergefun{H}$, ultrametric HC & \cref{mergefun} \\

Correspondence $R \subseteq X \times Y$, $\pi_X : R \to X$, $\pi_Y : R \to Y$ & \cref{correspondence-def} \\

Two HCs are $\vec{\epsilon}$-interleaved w.r.t. $R$ & \cref{epsilon-interleaved-wrt-R} \\

The correspondence-interleaving distance $\dWI$ & \cref{CI-definition} \\

The Hausdorff distance $\dH$ & \cref{Hausdorff-distance} \\

The Gromov--Hausdorff distance $\dGH$ & \cref{GH-distance} \\

metric probability space & \cref{mps-definition} \\

The uniform measure $\mu_{\Met}$ of $\Met$ & Right below \cref{mps-definition} \\

The uniform filtration $\Uniform$ & \cref{uniform-filtration-def} \\

The degree-Rips HC $\DR$ & \cref{def-DR} \\

A kernel $K$ & \cref{definition-of-kernel} \\

The uniform kernel & \cref{uniform-kernel} \\

The local density estimate $\conv{\mu_\MPS}{K}{s}$ & \cref{local-density-estimate} \\

The kernel filtration $\MPS\upar{s,k}$ of $\MPS$ & \cref{kernel-density-filtration-def} \\

The kernel linkage $\mml^K$ or $\mml$ & \cref{def-robust-linkage} \\

A curve $\gamma$ in a poset, a slice $H^{\gamma}$ & \cref{curve} \\

Robust single-linkage $\RL$ & \cref{robust-single-linkage} \\

The plug-in algorithm $\RG$ & \cref{vertical-slices} \\

$\lambda$, $\lambda^{x,y}$, $\lambda\con$, $\lambda\cov$, $\linkage{\lambda}$, $\lambda$-linkage & \cref{lambda-notation} \\

The Prokhorov distance $\dP$ & \cref{Prokhorov-distance} \\

The Gromov--Hausdorff--Prokhorov distance $\dGHP$ & \cref{ghp-distance} \\

The Gromov--Hausdorff--Wasserstein distance & Right above \cref{GHW-stability} \\

The closest point correspondence $R_c$ & \cref{closest-points-correspondence} \\

CI-consistency & \cref{CI-and-MI-consistent} \\

The associated cluster tree $\calF H$ & \cref{associated-cluster-tree} \\

Hartigan consistency & \cref{definition-hartigan-consistency} \\

$\overline{\lambda}$ & \cref{phi-and-rescaling-lambda} \\

Persistent cluster $\bfC$, underlying set $U(\bfC)$, $\life$, $\birth$, $\length$ & \cref{persistentclusterdef} \\

Poset of persistent clusters $\PC$ & \cref{PC-def} \\

The set of $\leaves$ & \cref{leaves-def} \\

Persistence-based pruning $H_{\geq \tau}$ & \cref{persistence-pruning-def} \\

finite, pointwise finite, essentially finite HC & \cref{finiteness} \\ 

The barcode $\barc(H)$ & \cref{barcode-of-HC-def} \\

The bottleneck distance $\dB$ & Right above \cref{bottleneck-correspondence-interleaving} \\

Prominence diagram & \cref{prominence-diagram-def} \\

Prominence diagram $\prom(H)$ of a HC $H$ & \cref{definition:prom-finite} and \cref{prom-H-def} \\

Gap $\gap_\gapindex$, gap size $\gapsize_\gapindex$ & \cref{gap-gapsize-def} \\

Gap $\gap_\gapindex(H)$, gap size $\gapsize_\gapindex(H)$ of a HC $H$ & \cref{gap-of-H-notation} \\

Persistence-based flattening $\PF$ & \cref{persistence-based-flattening-def} \\

$R_X : H(\vec{r}) \to E(\vec{r} + \vec{v} \vec{\epsilon})$ & \cref{notation-R_X} \\

$B(x,r)$, the open ball of radius $r$ centered at $x$ & \\



 

\hline
\end{tabular}
}
\caption*{Definitions and frequently used notation.}
\end{table}

\subsection{Related Work} 
Distances between (one- and two-parameter) hierarchical clusterings
have been studied by \citet{carlsson-memoli-hierarchical-clustering,carlsson-memoli-multid}. 
The cor\-re\-spon\-dence-interleaving distance is a generalization of this work; 
see \cref{distances-subsection} for a discussion.
The formigram distance, introduced by \cite{kim-memoli}, 
can also be seen as a particular instance of the correspondence-interleaving distance. 
\cite{eldridge-belkin-wang} introduce the merge distortion metric for one-parameter hierarchical clusterings, 
which is closely related to the correspondence-interleaving distance.

Work of \cite{stability-density-based-clustering} and \cite{chazal-guibas-oudot-skraba}
addresses the stability of consistent hierarchical clustering methods. 
In their frameworks, stability is guaranteed
when their assumptions on the underlying distribution are satisfied. 
In contrast, our stability results hold without distributional assumptions.

Combining density estimates and graphs that encode distance relations 
to estimate the density-contour hierarchical
clustering has a long history, and several methods based on this idea have been proposed. 
Along with the work of \cite{Cuevas-Febrero-Fraiman} already mentioned, 
see, e.g., \cite{biau-cadre-pelletier, rinaldo-wasserman, stuetzle-nugent, 
chazal-guibas-oudot-skraba, bobrowski-mukherjee-taylor}. 
For another perspective, see 
\cite{Aragam-etal}; 
the authors also combine density estimation with the single-linkage algorithm, 
but approach the clustering problem using the idea of Bayes optimal partitions from parametric model-based clustering.

The consistency of robust single-linkage was first established 
by \cite{chaudhuri-dasgupta-10}, 
and then generalized to density functions supported on manifolds by \cite{cluster-trees-on-manifolds}. 
\cite{eldridge-belkin-wang} introduced a notion of consistency that is closely related to CI-consistency, 
and, building on results of Chaudhuri--Dasgupta, they
show that robust single-linkage is consistent in this sense.

Multiparameter hierarchical clustering is a topic of increasing interest, 
as multiparameter hierarchical clusterings have the potential to capture very rich cluster structure in data. 
See, e.g., \cite{carlsson-memoli-multid, 
buchin-etal-trajectory, 
kim-memoli, 
jardine-stable-components, 
bauer-etal-hierarchical-clustering, 
cai-kim-memoli-wang}. 
We expect that the correspondence-interleaving distance will be useful for analyzing the properties 
of multiparameter hierarchical clustering methods in settings beyond this paper.

As mentioned earlier, our approach to taking slices of the degree-Rips hierarchical clustering 
is motivated by the standard practice in TDA 
of studying multiparameter persistence modules via one-parameter slices. 
See, for example, 
\cite{cerri-difabio-ferri-frosini-landi, 
cagliari-difabio-ferri, 
lesnick-wright, 
landi-interleavings, 
cfklw-kernel, 
vipond-landscape, 
carriere-blumberg}.


When the input data is a finite subset of Euclidean space, and $\gamma$ is a line with constant $s$-component, the slice of kernel linkage by $\gamma$ recovers the connected components of the weighted \v{C}ech filtration introduced by \cite{anai-chazal-glisse-ike-inakoshi-tinarrage}, when their parameter $p$ is set to~$\infty$.
In particular, their stability result applies to this slice of kernel linkage.

As we have already mentioned, the persistence-based flattening we introduce 
is a modification of the ToMATo clustering algorithm \citep{chazal-guibas-oudot-skraba}. 
The persistence-based flattening is defined using a pruning procedure 
we call the persistence-based pruning, which resembles the pruning of 
\cite{statistical-inference-cluster-trees}.

\cite{blumberg-lesnick-2param-published} prove a stability result for the simplicial degree-Rips bifiltration, 
which we discuss in \cref{choice-of-GHP}. 
\cite{jardine-persistent-homotopy} has also proved results about the stability of degree-Rips, 
using a hypothesis involving configuration spaces, rather than a distance on metric probability spaces. 
\citet[Section~6.5]{scoccola} shows that results in this paper can be lifted to the stability of the kernel filtration 
(\cref{kernel-density-filtration-def}), 
which in particular implies that other topological invariants of this multi-filtration 
are Gromov--Hausdorff--Prokhorov stable.

\addtocontents{toc}{\protect\setcounter{tocdepth}{2}}

\section{Hierarchical Clustering} \label{Hierarchical-clustering}

The notion of a hierarchical clustering (HC) 
has been formalized in a variety of ways in the clustering literature; 
see \citet{carlsson-memoli-hierarchical-clustering} and references therein. 
In this section we introduce a new formalization of this notion, 
which, in particular, allows for HCs 
with multiple parameters. 
We introduce the 
\emph{correspondence-interleaving} distance between HCs, 
which generalizes the distance on dendrograms introduced by 
\cite{carlsson-memoli-hierarchical-clustering}, 
and we develop its basic properties. 
In later sections of the paper, 
we will use the correspondence-interleaving distance to formulate 
stability and consistency results for 
hierarchical clustering algorithms. 



We also define the degree-Rips and kernel linkage hierarchical clusterings, 
as well as one-parameter slices of these constructions. 
These are the basis for all the clustering methods we consider in the rest of the paper.

\subsection{The Definition of a Hierarchical Clustering}

In order to define the notion of hierarchical clustering, 
we first define the notion of clustering. 
See \cref{hdbscan-data} for an example.

\begin{definition} \label{clustering-def}
Let $\SET$ be a set. 
A \define{clustering} of $\SET$ is a set of non-empty, disjoint subsets of $\SET$. 
The elements of a clustering are called \define{clusters}.
\end{definition}

We will formalize hierarchical clusterings using the notion 
of a partially ordered set. 
There are many good references for this notion, 
for example \cite[Ch. 2.2.2]{chiossi}.

\begin{definition} \label{poset-def}
A \define{partially ordered set (poset)} is a set $P$ 
together with a binary relation~$\preceq$ such that 
$(1)$ for all $p \in P$, $p \preceq p$; 
$(2)$ for all $p,q \in P$, if $p \preceq q$ and $q \preceq p$ then $p = q$; 
$(3)$ for all $p,q,r \in P$, if $p \preceq q$ and $q \preceq r$ then $p \preceq r$. 
If $P,Q$ are posets, and $f : P \to Q$ is a function, 
then $f$ is \define{order-preserving} 
if for all $p,p' \in P$ with $p \preceq p'$, 
$f(p) \preceq f(p')$ in $Q$. 
If~$P$ is a poset, the \define{opposite poset} $P^{\op}$ 
is the poset with the same underlying set, 
and with $p \preceq p'$ in $P^{\op}$ 
if and only if $p \succeq p'$ in $P$.
\end{definition}

\begin{definition} \label{poset-of-clusterings-def}
Let $\SET$ be a set. The \define{poset of clusterings} of $\SET$,
denoted $\C(\SET)$, is the poset whose elements are the clusterings of $\SET$,
and where $S \preceq T \in \C(\SET)$ if, for each cluster $A \in S$, 
there is a (necessarily unique) cluster $B \in T$ such that $A \subseteq B$.
\end{definition}


\begin{definition} \label{P-HC-def}
Let $P$ be a poset, and let $\SET$ be a set. 
A \define{$P$-hierarchical clustering} of $\SET$ 
is an order-preserving function 
$H : P \to \C(\SET)$.
\end{definition}

The notion of a $P$-hierarchical clustering generalizes the dendrograms of 
\citet[Section 3.1]{carlsson-memoli-hierarchical-clustering}, 
where the indexing poset was taken to be $[0, \infty)$.

\begin{definition} \label{n-param-HC}
Let $\SET$ be a set, and let $n \geq 1$. 
An \define{$n$-parameter hierarchical clustering} of $\SET$ 
is a $P$-hierarchical clustering $H : P \to C(\SET)$, 
where $P = I_1 \times \dots \times I_n$ 
with $I_j$ an interval of $\bbR$ or $\bbR^{\op}$ 
for all $1 \leq j \leq n$.
\end{definition}

Note that one-parameter HCs come in two flavors, 
depending on whether clusters merge 
as the real parameter increases or decreases; 
borrowing terminology from category theory, 
if $I \subseteq \bbR$ is an interval, 
we call an $I$-hierarchical clustering 
\emph{covariant}, 
and if $I \subseteq \bbR^{\op}$, 
we call an $I$-hierarchical clustering 
\emph{contravariant}. 
One-parameter HCs can be visualized 
by \emph{dendrograms}: 
see \cref{HC-interleaving}. 
We now give two key examples of one-parameter HCs.


\begin{example} \label{density-contour}
Let $f : \bbR^d \to \bbR$ be a probability density function, 
and let $\support(f)$ be its support. 
Following \cite{hartigan-75}, 
the \define{density-contour} hierarchical clustering $H(f)$ 
is the contravariant, $(0, \infty)^{\op}$-hierarchical clustering of $\support(f)$, 
where, for $r > 0$, $H(f)(r)$ is the set of connected components of 
$\{ x \in \support(f) \, : \, f(x) \geq r \}$. 
\end{example}

\begin{example} \label{single-linkage-def}
Let $\Met$ be a metric space. 
The \define{single-linkage} hierarchical clustering $\SL(M)$ \citep{sibson} 
is the covariant, $(0, \infty)$-hierarchical clustering of $\Met$, 
where, for $r > 0$, $\SL(\Met)(r)$ is the partition of $\Met$ 
defined by the smallest equivalence relation $\sim_{r}$ on $\Met$ with
$x \sim_{r} y$ if $d_\Met (x,y) \leq r$. 
Single-linkage can also be defined in terms of the \define{Rips graph}~$\Rips(\Met)$. 
For $r > 0$, let $\Rips(\Met)_r$ be the graph with vertex set $\Met$ 
and with an edge between~$x$ and $y$ if $d_\Met (x,y) \leq r$. 
Then $\SL(\Met)(r)$ is the partition of $\Met$ by the vertex sets of the connected components of $\Rips(\Met)_r$.
\end{example}

We now describe a way to \emph{extend} 
any $n$-parameter hierarchical clustering 
$H : P \to \C(\SET)$ to an $\bbR^n$-hierarchical clustering 
$\extension{H} : \bbR^n \to \C(\SET)$. 
This will be useful when we consider distances between HCs, 
since we can compare any two $n$-parameter HCs, 
with possibly different indexing posets, 
by first extending them to $\bbR^n$-HCs, 
and then comparing the extensions. 
The idea is to first make $H$ covariant in each parameter, 
by replacing any interval of $\bbR^{\op}$ in $P$ with its negative, 
and then to extend $H$ to all of $\bbR^n$ 
using the empty clustering $\emptyset$ (the minimum in $\C(\SET)$) 
and the clustering $\{\SET\}$ (the maximum in $\C(\SET)$).

Say $(I, \preceq) \subseteq \bbR^{\op}$ is an interval: 
as a set $I$ is a real interval, 
and $a \preceq b$ in $I$ if and only if 
$a \geq b$ as real numbers. 
Let $-I = \{-a : a \in I\}$. 
There is an isomorphism of posets 
$\rho_I : (-I, \preceq)^{\op} \to (I, \preceq)$ 
with $\rho_I(a) = -a$, 
and $(-I, \preceq)^{\op} = (-I, \leq)$ 
is an interval of $\bbR$. 

\begin{definition} \label{extension}
Say $P = I_1 \times \dots \times I_n$ 
with each $I_j$ an interval of $\bbR$ or $\bbR^{\op}$. 
Let $P'$ be the poset obtained from $P$ 
by replacing each interval $I_j \subseteq \bbR^{\op}$ 
with the interval $-I_j \subseteq \bbR$. 
Then we have an isomorphism of posets 
$\rho_P : P' \to P$. 
If $\SET$ is a set, and $H : P \to \C(\SET)$ 
is an $n$-parameter hierarchical clustering of $\SET$, 
let $H' : P' \to \C(\SET)$ be $H \circ \rho_P$. 
The \define{extension} of $H$ 
is the $\bbR^n$-hierarchical clustering 
$\extension{H} : \bbR^n \to \C(\SET)$ with 
\[
    \extension{H}(r) =
    \begin{cases}
    	H'(r) &\text{ if $r \in P'$}\\
    	\{\SET\} &\text{ if $r \in \bbR^n \setminus P'$ and there is $p \in P'$ with $p < r$}\\
    	\emptyset &\text{ else}.
    \end{cases}
\]
\end{definition}

\subsection{The Correspondence-Interleaving Distance} \label{distances-subsection}

The distances for hierarchical clusterings we consider 
are based on the notion of \emph{interleaving}, 
which we have adapted from persistent homology 
\citep{chazal-etal-interleavings}. 
In the HC setting, 
interleavings have a simple definition, 
which we now give.

\begin{notation}
We write $\vec{\epsilon} = (\epsilon_1, \dots, \epsilon_n) \geq \vec{0}$
if $\epsilon_i \geq 0$ for $1 \leq i \leq n$.
\end{notation}

\begin{definition} \label{epsilon-interleaving}
Let $H$ and $E$ be $n$-parameter hierarchical clusterings of a set $\SET$, 
and let $\vec{\epsilon} \geq \vec{0}$. 
We say that $H$ and $E$ are $\vec{\epsilon}$-\define{interleaved} if, 
for all $\vec{r} \in \bbR^n$, 
we have $\extension{H}(\vec{r}) \preceq \extension{E}(\vec{r}+\vec{\epsilon})$ and 
$\extension{E}(\vec{r}) \preceq \extension{H}(\vec{r}+\vec{\epsilon})$ in $\C(\SET)$.
\end{definition}

\begin{definition} \label{interleaving-distance}
Let $H$ and $E$ be $n$-parameter hierarchical clusterings of a set $\SET$. 
Define the \define{interleaving distance}
\[
	\dI (H, E) = \inf \{ \epsilon \geq 0 : 
	H, E \text{ are } (\epsilon, \dots, \epsilon)\text{-interleaved} \}.
\]
\end{definition}

In the special case of one-parameter HCs, 
the interleaving distance has a very concrete, alternative formulation. 
We give this now, in order to provide intuition 
for interleavings.

\begin{definition} \label{mergefun}
Let $H : I \to \C(\SET)$ be a one-parameter hierarchical clustering of a set $\SET$. 
Define $\mergefun{H} \colon \SET \times \SET \to [-\infty, \infty]$ by
$\mergefun{H}(x,y) = \inf \{ r \in \bbR : 
\exists C \in \extension{H}(r), \; x,y \in C \}$. 
We say $H$ is an \define{ultrametric hierarchical clustering} 
if $I = [0, \infty)$; 
for all $x \in \SET$, 
there is $r_x > 0$ such that 
for any $r$ in the interval $[0, r_x)$, 
the clustering $H(r)$ contains 
the singleton cluster $\{x\}$; 
and there is $r \in [0, \infty)$ such that $H(r) = \{\SET\}$. 
\end{definition}

For example, the single-linkage of a finite metric space 
is an ultrametric hierarchical clustering. 
If $H$ is an ultrametric hierarchical clustering of $\SET$, 
then the function $\mergefun{H}$ 
defines an ultrametric on $\SET$. 
See \cite{carlsson-memoli-hierarchical-clustering} 
for a detailed discussion of this perspective. 
For $H, E$ one-parameter HCs of $\SET$, we write 
$d_{\infty}(\mergefun{H}, \mergefun{E}) = 
\sup \{ |\mergefun{H}(x,y) - \mergefun{E}(x,y)| : x,y \in \SET \}$.

\begin{proposition} \label{dI-and-infty-distance}
If $H$ and $E$ are one-parameter hierarchical clusterings of a set $\SET$, 
then $\dI(H, E) = d_{\infty}(\mergefun{H}, \mergefun{E})$.
\end{proposition}

The proof is elementary; see \cref{appendix-HC}. 
This formulation of the interleaving distance 
shows that, if $H$ and $E$ are $\epsilon$-interleaved, 
then the parameter values where clusters 
are born and merge are perturbed by at most $\epsilon$. 
See \cref{HC-interleaving} for an example. 
We now give a simple example of a stability result that can be formulated using interleavings. 

\begin{figure}[tbp]
    \centering
    \includegraphics[width=0.6\textwidth]{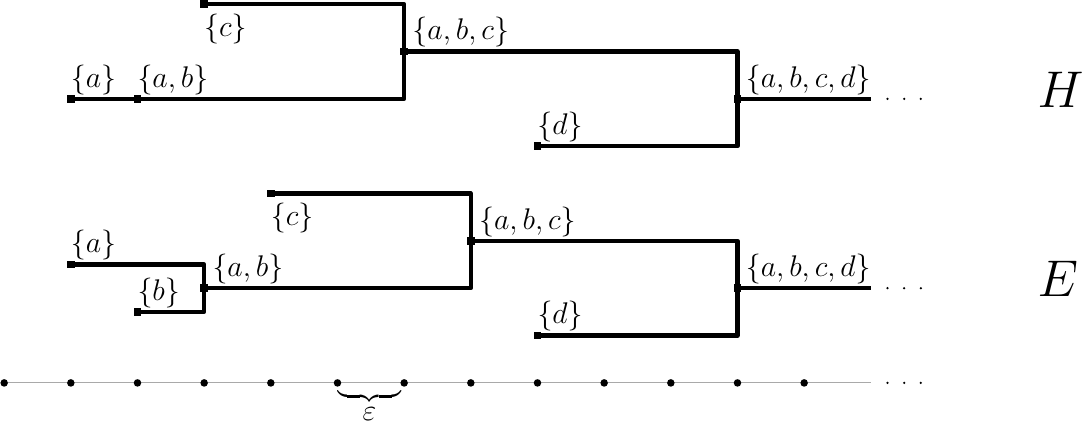}
    \caption{Two dendrograms representing one-parameter, covariant HCs $H$ and $E$
    of the set $\{a,b,c,d\}$. 
    The parameter values where points enter the HC, and where clusters merge, are perturbed by at most $\epsilon$,
    so the HCs are $\epsilon$-interleaved.}
    \label{HC-interleaving}
\end{figure}


\begin{proposition} \label{Hf-stability}
Let $f, g : \bbR^d \to \bbR_{\geq 0}$ 
be probability density functions with the same support. 
Then $\dI(H(f), H(g)) = || f - g ||_{\infty}$.
\end{proposition}

We give an elementary proof in \cref{appendix-HC}. 
This kind of stability result for real-valued functions is standard in topological data analysis. 
See, for example, \citet[Example 4.3]{chazal-silva-glisse-oudot}. 
We now extend the interleaving distance 
to HCs of different sets, 
using correspondences.

\begin{definition} \label{correspondence-def}
A \define{correspondence} $R$ between sets $X$ and $Y$ is given by a set
$R \subseteq X \times Y$ such that the projections $\pi_X : R \to X$ and
$\pi_Y : R \to Y$ are surjective.
\end{definition}

If $\psi : Y \to X$ is a function between sets, 
and $S = \{C_i\}$ is a clustering of $X$, 
then $\psi^{*}(S) = \{\psi^{-1}(C_i)\}$ is a clustering of $Y$. 
This defines an order-preserving map $\psi^{*} : \C(X) \to \C(Y)$. 
If $P$ is a poset and $H$ is a $P$-hierarchical clustering of $X$, 
then $\psi^{*}(H) = \psi^{*} \circ H$ is a $P$-hierarchical clustering of $Y$.

\begin{definition} \label{epsilon-interleaved-wrt-R}
Let $H$ and $E$ be $n$-parameter hierarchical clusterings 
of sets $X$ and $Y$ respectively,
let $R \subseteq X \times Y$ be a correspondence, 
and let $\vec{\epsilon} \geq 0$.
We say that $H$ and $E$ are $\vec{\epsilon}$-\define{interleaved 
with respect to} $R$ if $\pi_X^{*}(H)$ and $\pi_Y^{*}(E)$ 
are $\vec{\epsilon}$-interleaved as 
$n$-parameter hierarchical clusterings of $R$.
\end{definition}

\begin{definition} \label{CI-definition}
Let $H$ and $E$ be $n$-parameter hierarchical clusterings 
of sets $X$ and $Y$ respectively. 
Define the \define{correspondence-interleaving distance}
\[
	\dWI (H, E) = \inf_{R} \inf \{ \epsilon \geq 0 : 
	H, E \text{ are } (\epsilon, \dots, \epsilon)\text{-interleaved w.r.t. } R \}.
\]
where the infimum is over all correspondences $R$ between $X$ and $Y$.
\end{definition}

Aside from set-theoretic concerns, $\dWI$ defines an extended-pseudo-metric 
on $n$-pa\-ram\-e\-ter hierarchical clusterings 
(see \cref{appendix-HC} for the elementary proof):

\begin{proposition} \label{CI-distance-properties}
The distance $\dWI$ satisfies the following properties, 
for all $n$-parameter hierarchical clusterings:
$(1)$ for any $H$, $\dWI(H,H) = 0$; 
$(2)$ for any $H, E$, $\dWI(H, E) = \dWI(E, H)$; 
$(3)$ for any $H, E, F$, $\dWI(H, F) \leq \dWI(H, E) + \dWI(E, F)$.
\end{proposition}

Using correspondences to extend 
the interleaving distance to HCs 
of different sets is inspired by 
the Gromov--Hausdorff distance from metric geometry \citep[Chapter 7.3]{burago-etal}. 
In fact, there is a close connection between 
the correspondence-interleaving distance 
and the Gromov--Hausdorff distance. 
In their work on hierarchical clustering, 
\cite{carlsson-memoli-hierarchical-clustering} 
use the Gromov--Hausdorff distance between 
the ultrametrics induced by HCs such as 
the single-linkage HC of a finite metric space. 
We now recall the definition of the Gromov--Hausdorff distance, 
and show that the correspondence-interleaving distance 
recovers this distance, 
in the special case of ultrametric hierarchical clusterings.

\begin{definition} \label{Hausdorff-distance}
Let $A, B$ be compact subsets of a metric space $\Met$. 
The \define{Hausdorff distance} between $A$ and $B$ is 
$\dH^\Met(A,B) = 
\inf \{ \epsilon > 0 : A \subseteq B^{\epsilon} \text{ and } 
B \subseteq A^{\epsilon}\}$, 
where, for any $W \subseteq \Met$, $W^{\epsilon} = 
\{x \in \Met : \exists w \in W, d_\Met(x,w) < \epsilon\}$.
\end{definition}

\begin{definition} \label{GH-distance}
Let $\Met,\Metalt$ be compact metric spaces. 
The \define{Gromov--Hausdorff} distance is
\[
	\dGH(\Met,\Metalt) = \inf_{i, j} \; \dH^Z(i(\Met), j(\Metalt)) \; ,
\]
where the infimum is taken over all isometric embeddings $i : \Met \to Z$ and 
$j : \Metalt \to Z$ into a common metric space $Z$.
\end{definition}

\begin{proposition} \label{CI-and-GH}
Let $H$ and $E$ be ultrametric hierarchical clusterings 
of sets $X$ and $Y$ respectively. 
Then $\dWI(H, E) = 2 \cdot 
\dGH\left( (X, \mergefun{H}), (Y, \mergefun{E}) \right)$.
\end{proposition}

\begin{proof}
Let $R$ be a correspondence between $X$ and $Y$. 
One says that the \emph{distortion} of $R$ is 
$\distortion(R) = \sup \{|\mergefun{H}(x,x') - \mergefun{E}(y,y')| : 
(x,y),(x',y') \in R\}$ 
\cite[Definition 7.3.21]{burago-etal}. 
Then, one has 
$\dGH\left( (X, \mergefun{H}), (Y, \mergefun{E}) \right) 
= \frac{1}{2} \inf_{R} \distortion(R)$, 
where the infimum is taken over all correspondences between $X$ and $Y$ 
\cite[Theorem 7.3.25]{burago-etal}. 
Now, the proposition follows from the fact that, 
for any correspondence $R$, 
$\distortion(R) = 
\inf \{\epsilon \geq 0 : 
H,E \text{ are } \epsilon\text{-interleaved w.r.t. } R\}$, 
which is \cref{dis-interleaving-lemma}.
\end{proof}

\subsection{Degree-Rips and Kernel Linkage} \label{robust-linkage-section}

We now introduce \emph{degree-Rips} and \emph{kernel linkage}, 
the multiparameter hierarchical clustering methods 
that are the basis for all the clustering algorithms we consider in this paper. 
In the introduction, we described degree-Rips in the case that the input is a finite metric space. 
However, it is convenient to consider a natural generalization of this construction. 
Metric measure spaces \citep{gromov, villani} 
are metric spaces together with a Borel measure \citep{dudley}. 
Since the measures we consider will always be probability measures, 
we use the notion of metric probability space:


\begin{definition} \label{mps-definition}
A \define{metric probability space} consists of a metric space $\MPS$
together with a Borel probability measure $\mu_\MPS$ on $\MPS$.
\end{definition}

The degree-Rips hierarchical clustering we define in this section takes a metric probability space as input. 
If $\Met$ is a finite metric space, and one equips $\Met$ with the \emph{uniform measure} $\mu_{\Met}$, 
such that $\mu_{\Met}(A) = |A| \, / \, |\Met|$ for any $A \subseteq \Met$, 
then the degree-Rips hierarchical clustering of $(\Met, \mu_{\Met})$ 
recovers the version of degree-Rips we described in the introduction. 
Unless otherwise stated, we equip finite metric spaces with the uniform measure.

Working in the generality of metric probability spaces has two main advantages. 
First, if $f$ is a density function on Euclidean space, we can consider the degree-Rips hierarchical clustering 
of the metric probability space $(\support(f), \mu_f)$, 
where $\support(f)$ is the support of $f$, and $\mu_f$ is the probability measure defined by $f$. 
This construction plays a key role in the proof of our consistency theorem. 
Second, finite metric spaces with non-uniform measures are useful for computational purposes. 
In \cref{subsampling-approximation}, we describe an approximation scheme for degree-Rips, 
in which a large input $\Met$ (a finite metric space) is approximated by a small subset $\Metalt \subset \Met$, 
where $\Metalt$ has a non-uniform measure that approximates the uniform measure of $\Met$.

\begin{definition} \label{uniform-filtration-def}
Let $\MPS$ be a metric probability space, and let $s,k > 0$.
Let $\Uniform(\MPS)\upar{s,k} = \left\{x \in \MPS \; : \; \mu_{\MPS}(B(x, s)) \geq k\right\}$. 
Here and throughout the paper, $B(x,s)$ is the open ball in $\MPS$ of radius $s$ centered at $x$. 
We have $\Uniform(\MPS)\upar{s,k} \subseteq \Uniform(\MPS)\upar{s',k'}$ 
whenever $s' \geq s$ and $k' \leq k$. This forms a $2$-parameter filtration of $\MPS$, 
which we call the \define{uniform filtration} of $\MPS$.
\end{definition}

\cite{blumberg-lesnick-2param-published} call this the ``measure bifiltration''. 
We combine the uniform filtration with single-linkage (\cref{single-linkage-def}) to define degree-Rips.

\begin{definition} \label{def-DR}
Let $\MPS$ be a metric probability space. 
Define the \define{degree-Rips} hierarchical clustering of $\MPS$ as
the $2$-parameter hierarchical clustering:
\begin{align*}
	\DR(\MPS) : \bbRs \times \bbRs^{\op} &\to \C(\MPS) \\
	(s, k) &\mapsto \SL\left(\Uniform(\MPS)\upar{s,k}\right)(s) \; .
\end{align*}
\end{definition}

See \cref{DR-intro-figure} for an illustration of degree-Rips. 
As described in the introduction, 
we are motivated to consider the degree-Rips hierarchical clustering because of its close connection 
to well-established methods for data analysis, such as the DBSCAN clustering algorithm 
and the degree-Rips bifiltration. 
However, there are some choices baked in to the definition that may not be optimal for some applications. 
So, we will define a generalization of degree-Rips, which we call \emph{kernel linkage}. 

As motivation, notice that degree-Rips estimates the density near a point $x$ 
by taking the measure of the ball $B(x,s)$. 
Equivalently, with respect to the measure $\mu_{\MPS}$, 
one integrates the \emph{uniform kernel}, which is equal to one on this ball and vanishes elsewhere. 
One could just as well use another kernel when estimating density, 
as with kernel density estimators \citep{silverman86}. 
Second, notice that the definition of degree-Rips uses the $s$ parameter twice: 
as the radius of the ball $B(x,s)$, and as the spatial parameter for single-linkage. 
It is not necessary for these two values to be equal, 
and in fact, the robust single-linkage algorithm (\cref{robust-single-linkage}) 
allows these two values to differ by a constant factor. 
These two observations motivate the definition of kernel linkage.

In the setting of non-parametric density estimation, 
a \emph{kernel} \citep[Ch. 4.2]{silverman86} 
quantifies local-ness; 
given a point $x$, a kernel quantifies the extent to which any other point~$x'$ is close to $x$. 
Since we are working with metric spaces, 
we will apply kernel functions to the distance between $x$ and $x'$. 

\begin{definition} \label{definition-of-kernel}
A \define{kernel} is a non-increasing function $K : \bbR_{\geq 0} \to \bbR_{\geq 0}$
that is continuous from the right and such that $0 < \int_0^\infty K(r) \dd r < \infty$.
\end{definition}

Note that, in particular, $K(0) > 0$ and $\lim_{r \to \infty} K(r) = 0$.

\begin{example} \label{uniform-kernel}
Many kernels used for density estimation are kernels in the above sense 
(see \cref{reparametrize-for-consistency}). 
We will be particularly interested in $K = \mathbf{1}_{\{r < 1\}} : \bbR_{\geq 0} \to \bbR_{\geq 0}$, 
with $K(x) = 1$ if $x < 1$ and $K(x) = 0$ otherwise. 
We refer to this as the \define{uniform kernel}.
\end{example}

\begin{definition} \label{local-density-estimate}
Let $K$ be a kernel, and let $\MPS$ be a metric probability space.
Define the \define{local density estimate} of a point $x \in \MPS$ at scale $s > 0$ as
\[
    \conv{\mu_\MPS}{K}{s}(x) := 
    \int_{x' \in \MPS} K\left(\frac{d_\MPS(x,x')}{s}\right) \dd\mu_\MPS.
\]
\end{definition}

\begin{remark} \label{reparametrize-for-consistency}
Let $\MPS$ be the metric probability space given by 
Euclidean space $\bbR^d$ 
equipped with the empirical measure defined by a finite set of points $Z \subset \bbR^d$. 
The formula for the local density estimate is
\[
	\conv{\mu_\MPS}{K}{s}(x) = 
	\frac{1}{|Z|} \, \sum_{z \in Z} K \left( \frac{|| x - z ||}{s} \right) \, .
\]
Based on the usual formula for kernel density estimates 
(\citealp[Section 4.2.1]{silverman86}), 
one might expect a factor of $1 / s^d$ here. 
However, we need our local density estimate to be monotonic in $s$, 
in order to define the kernel filtration, below. 
In effect, one can re-introduce the factor $1 / s^d$ 
after taking one-parameter slices, 
and this is what we do to prove our consistency result 
(see \cref{phi-and-rescaling-lambda}).
\end{remark}

\begin{definition} \label{kernel-density-filtration-def}
Let $K$ be a kernel, let $\MPS$ be a metric probability space, and let $s,k > 0$.
Let $\MPS\upar{s,k} = \left\{x \in \MPS \; : \; \conv{\mu_\MPS}{K}{s}(x) \geq k\right\}$.
Note that, since $K$ is non-increasing, we have $\MPS\upar{s,k} \subseteq \MPS\upar{s',k'}$ 
whenever $s' \geq s$ and $k' \leq k$. This forms a $2$-parameter filtration of $\MPS$, 
which we call the \define{kernel filtration} of $\MPS$.
\end{definition}

In analogy to the definition of degree-Rips, 
we combine the kernel filtration with single-linkage to define kernel linkage:

\begin{definition} \label{def-robust-linkage}
Let $K$ be a kernel, and let $\MPS$ be a metric probability space. 
Define the \define{kernel linkage} of $\MPS$ as
the $3$-parameter hierarchical clustering of $\MPS$:
\begin{align*}
	\mml^K(\MPS) : \bbRs \times \bbRs \times \bbRs^{\op} &\to \C(\MPS) \\
	(s, t, k) &\mapsto \SL\left(\MPS\upar{s,k}\right)(t) \; .
\end{align*}
If there is no risk of confusion, we suppress $K$ from the notation, and write $\mml(\MPS)$.
\end{definition}


To build intuition about kernel linkage, 
it is helpful to first think about degree-Rips, which is easier to visualize. 
We provide examples and visualizations in \cref{Persistable}, where we describe Persistable. 
The interested reader may wish to look at these visualizations before reading 
the theoretical material in \cref{Stability}.

\subsection{Slices of Kernel Linkage and $\lambda$-linkage} \label{gamma-linkage-section}

We now formally define the notion of a one-parameter slice of a hierarchical clustering. 
This is analogous to taking a one-parameter slice of a multiparameter persistence module; 
see the Related Work section of the introduction for references. 
Taking one-parameter slices of kernel linkage, 
one recovers well-known methods for density-based clustering.

\begin{definition} \label{curve}
Let $P$ be a poset. 
A \define{curve} in $P$ 
is given by an interval $I_{\gamma}$ 
of $\bbR$ or $\bbR^{\op}$, 
and an order-preserving function $\gamma : I_{\gamma} \to P$. 
If $H : P \to \C(X)$ is a $P$-hierarchical clustering of a set $X$, 
and $\gamma : I_{\gamma} \to P$ is a curve in $P$, 
then the \define{slice} of $H$ by $\gamma$ 
is the one-parameter hierarchical clustering 
$H^{\gamma} : I_{\gamma} \to \C(X)$ given by $H \circ \gamma$.
\end{definition}


As discussed in the introduction, 
some well-known methods for density-based clustering can be recovered by taking slices of kernel linkage.

\begin{example} \label{robust-single-linkage}
The robust single-linkage algorithm 
of \cite{chaudhuri-dasgupta-10} 
can be recovered by taking slices of kernel linkage. 
Let $\Met$ be a finite metric space with $n = |\Met|$. 
Let $\kappa \in \bbN$ be the density threshold parameter 
of robust single-linkage, 
and let $\alpha > 0$ be its scale parameter. 
The robust single-linkage of $\Met$ is 
$\RL_{\kappa, \alpha}(\Met) = \mml^{K}(\Met)^{\gamma}$, 
where we take $K$ to be the uniform kernel, 
and $\gamma$ is the covariant curve 
$\gamma \colon (0, \infty) \to \bbRs^{\times 3}$ with 
$\gamma(r) = (r, \alpha r, \kappa / n)$. 
This is a line through 
the kernel linkage parameter space, 
which fixes the density threshold parameter 
\textup{k}, and allows the spatial parameters 
\textup{s} and \textup{t} to vary.
\end{example}

\begin{example} \label{vertical-slices}
If we fix the spatial parameters \textup{s} and \textup{t}, 
and allow the density threshold parameter \textup{k} to vary, 
we recover the plug-in algorithm for density-based clustering, 
described in the introduction. 
See, for example, 
\cite{Cuevas-Febrero-Fraiman, 
chazal-guibas-oudot-skraba}. 
In detail, let $\Met$ be a finite metric space. 
For any $s,t > 0$, and for any kernel $K$, 
the plug-in hierarchical clustering of $\Met$ is 
$\RG_{s,t}^K(\Met) = \mml^{K}(\Met)^{\gamma}$ 
for the contravariant curve 
$\gamma \colon (0, \infty) \to \bbRs^{\times 3}$ with $\gamma(r) = (s, t, r)$. 
\end{example}


Slices in which one parameter is fixed, 
like in the previous two examples, 
lead to stability problems, 
as we show in \cref{instability-related-methods}. 
Moreover, such slices can struggle to capture multi-scale cluster structure in data 
(see the rideshare data in \cref{Persistable-on-benchmark-data}). 
So, for Persistable, 
we use lines in the kernel linkage parameter space 
in which all parameters vary.

%

\begin{example} \label{lambda-notation}
For Persistable, we take slices of kernel linkage 
by a family of curves $\lambda$ that we specify now. 
See \cref{slicing-intro-figure}. 
Each $\lambda$ parameterizes a line 
in the $(s, k)$-space $\bbRs \times \bbRs^{\op}$, 
and we extend this to a curve in the 
$(s, t, k)$-space $\bbRs \times \bbRs \times \bbRs^{\op}$ by setting $s = t$. 
We specify a line $\lambda$ by choosing 
an $s$-intercept $x > 0$ and a $k$-intercept $y > 0$. 
We write $\lambda^{x,y}$ if we need to specify the intercepts. 
Let $\sigma = -y / x$ be the slope of $\lambda$. 
If we parameterize $\lambda$ with the $k$ coordinate, 
we get the curve 
$\lambda\con : (0, y)^{\op} \to \bbRs \times \bbRs \times \bbRs^{\op}$ 
defined by $\lambda\con(r) = \left( (r / \sigma) + x, \; (r / \sigma) + x, \; r \right)$. 
If we parameterize with the $s$ coordinate, we get the curve 
$\lambda\cov : (0, x) \to \bbRs \times \bbRs \times \bbRs^{\op}$ 
defined by $\lambda\cov(r) = \left( r, \; r, \; \sigma r + y \right)$. 

We say that the $\lambda$-linkage of a metric probability space $\MPS$ 
is the hierarchical clustering  
\[
	\linkage{\lambda}(\MPS) := \mml^{K}(\MPS)^{\lambda}
\]
where $K$ is the uniform kernel. 
Since we use the uniform kernel, 
the slices $\linkage{\lambda}$ are slices of the degree-Rips hierarchical clustering.
\end{example}

When the input is a finite metric space, 
computing $\lambda$-linkage is similar to computing robust single-linkage. 
So, one can adapt the algorithms of \cite{mcinnes-healy} to compute $\lambda$-linkage. 
This is what we do for our implementation of Persistable \citep{scoccola-rolle-joss}.

\section{Stability} \label{Stability}

In the introduction to this paper, we stated \cref{intro-DR-stability}, 
which says that the degree-Rips hierarchical clustering method is 2-Lipschitz, 
with respect to the Gromov--Hausdorff--Prokhorov distance on finite metric spaces, 
and the correspondence-interleaving distance on hierarchical clusterings. 
In \cref{robust-linkage-section} we defined the degree-Rips hierarchical clustering 
not just of a finite metric space, but in the generality of metric probability spaces. 
In this section, we prove that degree-Rips is 2-Lipschitz for compact metric probability spaces 
(this includes \cref{intro-DR-stability} as a special case). 
Furthermore, we consider the kernel linkage construction, 
also defined in \cref{robust-linkage-section}, 
and show that it is uniformly continuous with respect to the 
Gromov--Hausdorff--Prokhorov and correspondence-interleaving distances.

\subsection{Stability of Kernel Linkage} \label{stability-robust-linkage}

We begin by recalling the definition of the Gromov--Hausdorff--Prokhorov distance. 
See \citet[p. 762]{villani} or \citet{miermont} 
(though note that \citealt{villani} takes a sum instead of the maximum 
of $\dH$ and $\dP$ in the definition). 
We discussed the Hausdorff distance in \cref{distances-subsection}. 
The second ingredient we need is the Prokhorov distance \citep[Chapter 11.3]{dudley}.

\begin{definition} \label{Prokhorov-distance}
Let $\mu, \nu$ be Borel probability measures on a metric space $\Met$. 
The \define{Prokhorov distance} between $\mu$ and $\nu$ is 
\[
	\dP(\mu, \nu) = \inf \{ \epsilon > 0 \, : \, \mu(A) \leq \nu(A^{\epsilon}) + \epsilon 
			\; \text{and} \; \nu(A) \leq \mu(A^{\epsilon}) + \epsilon 
			\text{ for all Borel sets} \; A \subseteq \Met\} \, .
\]
\end{definition}

Now, the Gromov--Hausdorff--Prokhorov distance is a metric on the set of 
isometry-equivalence classes of compact metric probability spaces 
(see, e.g., \citealp{miermont}).

\begin{definition}
    \label{ghp-distance}
Let $(\MPS, \mu_\MPS), (\MPSalt, \mu_\MPSalt)$ be compact metric probability spaces. 
The \define{Gromov--Hausdorff--Prokhorov} distance between $(\MPS,\mu_\MPS)$ and $(\MPSalt,\mu_\MPSalt)$ is
\[
	\dGHP(\MPS,\MPSalt) = \inf_{i, j} \; \left\{ \max(\dH^Z(i(\MPS), j(\MPSalt)), 
	\dP(i_* \mu_\MPS, j_* \mu_\MPSalt)) \right\} \; ,
\]
where the infimum is taken over all isometric embeddings $i : \MPS \to Z$ and 
$j : \MPSalt \to Z$ into a common metric space $Z$.
\end{definition}

Before proving the stability of kernel linkage, we define a canonical correspondence between
two compact metric spaces embedded in a common metric space.

\begin{definition} \label{closest-points-correspondence}
Let $\Met$ and $\Metalt$ be compact metric spaces, let $Z$ be any metric space, 
and let $i : \Met \to Z$ and $j : \Metalt \to Z$ be isometric embeddings.
Define the \define{closest point correspondence} $R_c \subseteq \Met \times \Metalt$, where
$(x,y) \in R_c$ if and only if $d_Z(i(x),j(y)) = \min_{y' \in \Metalt} d_Z(i(x),j(y'))$
or $d_Z(i(x),j(y)) = \min_{x' \in \Met} d_Z(i(x'),j(y))$.
\end{definition}

\begin{theorem} \label{continuity-of-mml}
Kernel linkage is uniformly continuous with respect to
the Gromov--Haus\-dorff--Prokhorov distance on compact metric probability spaces, 
and the correspondence-interleaving distance.
If kernel linkage is defined using the uniform kernel, then it is $2$-Lipschitz.
\end{theorem}

\begin{proof}
Let $K$ be a kernel. 
We prove the following: 
for every $\epsilon > 0$, there exists $\delta > 0$ 
such that if $\MPS$ and $\MPSalt$ are compact metric probability spaces and $i : \MPS \to Z$ and $j : \MPSalt \to Z$ 
are isometric embeddings into a metric space $Z$ with 
$\dH(i(\MPS), j(\MPSalt)),\dP(i_* \mu_\MPS, j_* \mu_\MPSalt) < \delta$, 
then $\mml^{K}(\MPS)$ and $\mml^{K}(\MPSalt)$ are 
$(\epsilon, \epsilon, \epsilon)$-interleaved with respect to 
the closest point correspondence $R_c \subseteq \MPS \times \MPSalt$. 

Let $r' \in (0,K(0))$ and $\delta > 0$,
and define $\delta_s = \frac{2\delta}{K^{-1}(r')}$ and
$\delta_k = K(0)^2/r' - K(0) + K(0) \delta$.
We now prove that if $\MPS$ and $\MPSalt$ are compact metric probability spaces and $i : \MPS \to Z$ and $j : \MPSalt \to Z$ 
are isometric embeddings with $\dH(i(\MPS), j(\MPSalt)),\dP(i_* \mu_\MPS, j_* \mu_\MPSalt) < \delta$, 
then $\mml^{K}(\MPS)$ and $\mml^{K}(\MPSalt)$ are
$(\delta_s,2\delta,\delta_k)$-interleaved with respect to $R_c$.
This implies the statement of the previous paragraph, by taking
$r' \in (0,K(0))$ such that $K(0)^2/r' - K(0) < \epsilon/2$,
and $\delta > 0$ such that $2\delta/K^{-1}(r') < \epsilon$, $2\delta< \epsilon$, 
and $K(0) \delta < \epsilon / 2$.

It suffices to show that, for any $s,t > 0$ and $k > \delta_k$, we have relations in $\C(R_c)$:
\begin{align*}
	\pi_\MPS^{*}\left(\mml^{K}(\MPS)\right)(s,t,k) &\preceq
    \pi_\MPSalt^{*}\left(\mml^{K}(\MPSalt)\right)(s+\delta_s, t+2\delta, k-\delta_k) \\
	\pi_\MPSalt^{*}\left(\mml^{K}(\MPSalt)\right)(s,t,k) &\preceq
	\pi_\MPS^{*}\left(\mml^{K}(\MPS)\right)(s+\delta_s, t+2\delta, k-\delta_k) \, .
\end{align*}
We show that the first relation holds, 
and the second relation follows from a symmetric argument. 
Let $(x,y) \in R_c$. If $(x,y)$ belongs to a cluster of
$\pi_\MPS^{*}(\mml^{K}(\MPS))(s,t,k)$, then it belongs to a cluster of
$\pi_\MPSalt^{*}(\mml^{K}(\MPSalt))(s+\delta_s, t+2\delta, k-\delta_k)$,
by \cref{localdensitybounds}.
Now, assume that $(x,y)$ and $(x',y') \in R_c$ belong to the same cluster in
$\pi_\MPS^{*}(\mml^{K}(\MPS))(s,t,k)$.
This means that $x \sim_t x'$ in $\MPS\upar{s,k}$.
Since $|d_\MPS (x_1,x_2) - d_\MPSalt (y_1,y_2)| < 2\delta$ for every $(x_1,y_1),(x_2,y_2) \in R_c$,
we have that $y \sim_{t+2\delta} y'$ in $\MPSalt\upar{s+\delta_s, k-\delta_k}$ as required. 

It remains to consider the case where $K$ is the uniform kernel. 
Then $K(0) = 1$, and, for every $r' \in (0,1)$
we have $K^{-1}(r') = 1$, since $K^{-1} = K$.
Letting $r' \to 1$, the interleaving we constructed above
approaches a $(2\delta, 2\delta, \delta)$-interleaving, as needed.
\end{proof}

\begin{corollary} \label{DR-stability}
If $\MPS$ and $\MPSalt$ are compact metric probability spaces, then 
\[
	\dWI(\DR(\MPS), \DR(\MPSalt)) \leq 2 \cdot \dGHP(\MPS, \MPSalt) \, .
\]
\end{corollary}

\begin{proof}
Since degree-Rips is defined using the uniform kernel, 
it is $2$-Lipschitz by \cref{continuity-of-mml}.
\end{proof}

\cref{continuity-of-mml} implies a similar result for the 
Gromov--Hausdorff--Wasserstein distance, 
which is defined just as in \cref{ghp-distance}, 
except one replaces the Prokhorov distance with the Wasserstein distance 
\citep[p.~424]{gibbs-su}.

\begin{corollary} \label{GHW-stability}
Kernel linkage is uniformly continuous with respect to
the Gromov--Haus\-dorff--Wasserstein distance on compact metric probability spaces, 
and the correspondence-interleaving distance.
\end{corollary}

\begin{proof}
By \citet[Theorem~2]{gibbs-su}, if $\mu$ and $\nu$ are probability measures on a compact metric space, 
then $\dP(\mu, \nu)^2 \leq d_{\mathrm{W}}(\mu, \nu)$, 
where $d_{\mathrm{W}}$ denotes the Wasserstein distance. 
Now the corollary follows immediately from \cref{continuity-of-mml}.
\end{proof}

\begin{remark} \label{choice-of-GHP}
We now discuss why we use the Gromov--Hausdorff--Prokhorov distance 
for analyzing the stability of the degree-Rips and kernel linkage hierarchical clusterings. 
Because these constructions are density-sensitive, 
they are not continuous with respect to the Gromov--Hausdorff distance, 
unlike single-linkage \citep{carlsson-memoli-hierarchical-clustering}. 
They are also not continuous with respect to the Gromov--Prokhorov distance. 
This was observed for the simplicial degree-Rips bifiltration 
by \citet[Remark 3.8]{blumberg-lesnick-2param-published}, 
using the \emph{homotopy interleaving distance} on simplicial bifiltrations. 
The same example shows that the degree-Rips hierarchical clustering 
is not continuous with respect to the Gromov--Prokhorov distance 
on finite metric spaces (equipped with the uniform measure) 
and the correspondence-interleaving distance. 
However, as we have shown, if one uses the Gromov--Hausdorff--Prokhorov distance, 
degree-Rips is continuous, and even Lipschitz.

We note that \citet[Theorem 1.7]{blumberg-lesnick-2param-published} 
prove a Gromov--Prokhorov stability result for the simplicial degree-Rips bifiltration using homotopy interleavings. 
Necessarily, the conclusion is weaker than continuity. 
This stability result is complementary to our results. 
By working with the Gromov--Prokhorov distance, they make weaker assumptions on the input, 
and get correspondingly weaker conclusions.
\end{remark}

\subsection{Stability of Slices of Kernel Linkage} \label{stability-ell-linkage}

Interleavings between multiparameter hierarchical clusterings restrict to interleavings between slices, 
provided the slice does not fix any parameters. 
This is analogous to the behavior of interleavings and slices of multiparameter persistence modules; 
see the Related Work section of the Introduction for references.

Because the curves $\lambda$ that we use for Persistable (\cref{lambda-notation}) 
allow all parameters of kernel linkage to vary, 
we get Gromov--Hausdorff--Prokhorov stability for $\linkage{\lambda}$ 
as an immediate corollary of \cref{continuity-of-mml}.



\begin{corollary} \label{stability-lambda-linkage-in-X}
Let $\lambda = \lambda^{x,y}$ for $x,y > 0$, and let $\sigma$ be the slope of $\lambda$. 
Then, with respect to the Gromov--Hausdorff--Prokhorov distance on compact metric probability spaces 
and the correspondence-interleaving distance:
\begin{enumerate}
	\item $\theoremlinkage{\lambda\con}$ is $\max(2 \vert \sigma \vert, 1)$-Lipschitz,
	\item $\theoremlinkage{\lambda\cov}$ is $\max(\vert 1 / \sigma \vert, 2)$-Lipschitz.
\end{enumerate}
\end{corollary}

\begin{proof}
If $\MPS$ and $\MPSalt$ are compact metric probability spaces 
and $\delta > \dGHP(\MPS,\MPSalt)$, then the proof of \cref{continuity-of-mml} 
shows that $\mml(\MPS)$ and $\mml(\MPSalt)$ are $(2\delta, 2\delta, \delta)$-interleaved 
with respect to the closest-point correspondence. 
Restricting this interleaving to the line $\lambda$, 
as in e.g. \citet[Lemma 1]{landi-interleavings}, 
we get the required interleavings.
\end{proof}

Based on this result, 
we say that $\linkage{\lambda\con}$ and $\linkage{\lambda\cov}$ 
are \emph{stable with respect to the 
Gromov--Hausdorff--Prokhorov distance}. 
The slices $\linkage{\lambda\con}$ and $\linkage{\lambda\cov}$ 
are also stable in the choice of $\lambda$:



\begin{proposition} \label{stability-lambda-linkage-in-lambda}
Let $\MPS$ be a metric probability space. 
Let $\lambda = \lambda^{x,y}$ with slope $\sigma = -y / x$ be defined by intercepts $x,y > 0$, 
and let $\lambda' = \lambda^{x',y'}$ with slope $\sigma' = -y' / x'$ be defined by intercepts $x',y' > 0$.
\vspace{5pt}
\begin{enumerate}
	\item 
	$\dWI\left( \theoremlinkage{\lambda\con}(\MPS), 
	\theoremlinkage{\lambda'\con}(\MPS) \right) \leq 
	\max \left( |y - y'|, \; |x - x'| \cdot \min(|\sigma|, |\sigma'|) \right)$.
	\vspace{5pt}
	\item 
	$\dWI\left( \theoremlinkage{\lambda\cov}(\MPS), 
	\theoremlinkage{\lambda'\cov}(\MPS) \right) \leq 
	\max \left( |x - x'|, \; |y - y'| \cdot \min(|1/\sigma|, |1/\sigma'|)  \right)$.
\end{enumerate}
\end{proposition}

\begin{proof}
One can construct the required interleavings as in e.g. \citet[Lemma 2]{landi-interleavings}.
\end{proof}

\subsection{Instability of Related Methods} \label{instability-related-methods}

In the introduction, we discussed two well-known methods for density-based clustering, 
which can be recovered by taking slices of kernel linkage; 
these are robust single-linkage (\cref{robust-single-linkage}) 
and the plug-in algorithm (\cref{vertical-slices}). 
In contrast to the hierarchical clusterings $\linkage{\lambda}$ we use for Persistable, 
we now show that these methods are discontinuous with respect to the 
Gromov--Hausdorff--Prokhorov distance.

We begin with robust single-linkage. 
If one fixes the robust single-linkage parameters~$\kappa \in \bbN$ and $\alpha > 0$, 
then one can think of robust single-linkage $\RL_{\kappa, \alpha}$ 
as a function that takes a finite metric space as input 
and produces a one-parameter hierarchical clustering as output, 
and this function is discontinuous:

\begin{proposition} \label{RSL-is-GHP-discontinuous}
Let $\kappa \geq 2$ and $\alpha > 0$. 
With respect to the Gromov--Hausdorff--Prokhorov distance 
and the correspondence-interleaving distance, 
$\RL_{\kappa, \alpha}$ is discontinuous.
\end{proposition}

We prove this by giving a simple example in \cref{appendix-stability}. 
One could also formalize robust single-linkage differently, 
taking the density threshold parameter to be a ratio $k \in (0,1)$, 
and then letting $\RL_{k, \alpha}(\Met) = \mml(\Met)^{\gamma}$ 
for the covariant curve 
$\gamma \colon (0, \infty) \to \bbRs^{\times 3}$ with 
$\gamma(r) = (r, \alpha r, k)$. 
We show in \cref{appendix-stability} 
that this variant is also discontinuous 
with respect to the Gromov--Hausdorff--Prokhorov distance.

In contrast to the stability of $\linkage{\lambda}$ in $\lambda$ (\cref{stability-lambda-linkage-in-lambda}), 
changing the density threshold parameter $\kappa$ 
of robust single-linkage can lead to arbitrarily large changes 
in the output:

\begin{proposition} \label{RSL-is-not-Lipschitz-in-kappa}
Let $\kappa, \kappa' \in \bbN$ with $\kappa \neq \kappa'$, 
and let $\alpha > 0$. 
For any $D > 0$, there is a finite metric space $\Met$ such that 
$\dWI(\RL_{\kappa, \alpha}(\Met), \RL_{\kappa', \alpha}(\Met)) > D$. 
\end{proposition}

There is an analogous result for the variant of robust single-linkage 
that takes a density threshold $k \in (0,1)$ instead of $\kappa$. 
See \cref{appendix-stability}. 

We now consider the plug-in algorithm. 
As before, if one fixes the parameters $s,t > 0$, 
then $\RG_{s,t}$ is a function that takes a finite metric space as input 
and produces a one-parameter hierarchical clustering as output, 
and we have the following:

\begin{proposition} \label{RG-is-GHP-discontinuous}
Let $s,t > 0$, and let $\RG$ be defined using any kernel. 
With respect to the Gromov--Hausdorff--Prokhorov distance 
and the correspondence-interleaving distance, 
$\RG_{s,t}$ is discontinuous.
\end{proposition}

We prove this in \cref{appendix-stability} by giving a simple example. 
Finally, we consider the instability of the plug-in algorithm in its parameters. 
For a fixed metric probability space $\MPS$, \cref{stability-lambda-linkage-in-lambda} 
implies that (in both the covariant and contravariant versions) 
$\linkage{\lambda}(\MPS)$ is continuous as a function from 
its parameter space $\{\lambda^{x,y}\}_{x,y>0}$ 
to the space of one-parameter hierarchical clusterings endowed with the correspondence-interleaving distance. 
Similarly, if we fix a finite metric space $\Met$, then the plug-in algorithm 
can be seen as a function $\RG_{-,-}(\Met)$ that takes input $s,t > 0$ 
and produces a one-parameter hierarchical clustering as output. 
However, this is not continuous (see \cref{appendix-stability} for the proof):

\begin{proposition} \label{RG-is-discontinuous-in-t}
Let $\RG$ be defined using any kernel, and let $\Met$ be any finite metric space with $|\Met| \geq 2$. 
Then $\RG_{-,-}(\Met)$ is discontinuous, 
with respect to the Euclidean distance on $\bbRs^{2}$ 
and the correspondence-interleaving distance.
\end{proposition}

\subsection{Approximation of $\lambda$-linkage by Subsampling} \label{subsampling-approximation}

Because degree-Rips and $\lambda$-linkage are Gromov--Hausdorff--Prokhorov stable, 
they admit a very simple approximation algorithm. 
For example, say $\lambda = \lambda\con$, 
and we want to compute $\linkage{\lambda}(\Met)$, 
where $\Met$ is a finite metric space, equipped with the uniform measure. 
Say $\Metalt \subset \Met$ is a subsample, 
with $\epsilon = \dH(\Met,\Metalt)$. 
Then, by \cref{subsample-GHP}, one can compute a probability measure on $\Metalt$ 
such that $\dGHP(\Met,\Metalt) \leq \epsilon$, 
and therefore, by \cref{stability-lambda-linkage-in-X}, we have 
$\dWI(\linkage{\lambda}(\Met), \linkage{\lambda}(\Metalt)) \leq \max(2 \vert \sigma \vert, 1) \cdot \epsilon$. 

Therefore, if we can find a small subsample of $\Met$ that is close in the Hausdorff distance, 
we need only compute $\linkage{\lambda}$ of the subsample in order to approximate $\linkage{\lambda}(\Met)$. 
Persistable implements several subsampling methods, which can be used to get fast results on large data sets. 
We present an example in \cref{Persistable}.

\begin{proposition} \label{subsample-GHP}
    Let $(\MPS,\mu_\MPS)$ be a finite metric probability space.
    Let $\MPSalt \subseteq \MPS$ be a subset and let $i : \MPSalt \to \MPS$ denote the inclusion.
    Choose any function $f : \MPS \to \MPSalt$ with the property that, for every $x \in \MPS$, the point $f(x) \in \MPSalt$ is a closest point of $\MPSalt$ to $x$.
    Define a probability measure on $\MPSalt$ by
    $\mu_\MPSalt = f_*(\mu_\MPS)$.
    Then $\dP(\mu_\MPS, i_* \mu_\MPSalt) \leq \dH(\MPS,\MPSalt)$ and, in particular, $\dGHP(\MPS,\MPSalt) \leq \dH(\MPS,\MPSalt)$.
\end{proposition}
\begin{proof}
    Let $\epsilon > 0$ be such that $\MPS \subseteq \MPSalt^\epsilon$; it is enough to show that $\dP(\mu_\MPS, i_* \mu_\MPSalt) \leq \epsilon$.
    We prove that, for every $A \subseteq \MPS$ we have $\mu_\MPS(A) \leq \mu_\MPSalt(A^\epsilon)$ and $\mu_\MPSalt(A) \leq \mu_\MPS(A^\epsilon)$.

    Note that $d_\MPS(x,f(x)) \leq \epsilon$ for all $x \in \MPS$.
    It follows that $f^{-1}(A \cap \MPSalt) \subseteq A^\epsilon$ and $f(A) \subseteq A^\epsilon \cap \MPSalt$ for all $A \subseteq \MPS$.
    Note also that $i_*\mu_\MPSalt(B) = \mu_\MPS(f^{-1}(B \cap \MPSalt))$ for every $B \subseteq \MPS$, by definition of $i_*$ and $f_*$. 
    Let $A \subseteq \MPS$.
    Using the above, we get on one hand $i_*\mu_\MPSalt(A) = \mu_\MPS(f^{-1}(A\cap \MPSalt)) \leq \mu_\MPS(A^\epsilon)$.
    On the other hand, $A \subseteq f^{-1}(f(A)) \subseteq f^{-1}(A^\epsilon \cap \MPSalt)$, and thus $\mu_\MPS(A) \leq \mu_\MPS(f^{-1}(A^\epsilon \cap \MPSalt)) = i_*\mu_\MPSalt(A^\epsilon)$.
\end{proof}


\section{Consistency} \label{Consistency}

There is a natural notion of consistency for hierarchical clustering algorithms 
associated to the correspondence-interleaving distance. 
In this section, we define this, show that it implies Hartigan consistency, 
and show that $\lambda$-linkage is consistent with respect to 
the correspondence-interleaving distance.

In this section, unless otherwise stated, 
a hierarchical clustering will be a 
one-parameter hierarchical clustering (\cref{n-param-HC}).

\subsection{Notions of Consistency of Hierarchical Clustering Algorithms}

\begin{definition} \label{definition-hierarchical-clustering-algorithm}
A \define{hierarchical clustering algorithm} $\mathbb{A}$ 
with parameter space $\Theta$ is a mapping 
that assigns to each finite metric space $\Met$ 
and each parameter $\theta \in \Theta$ 
a hierarchical clustering $\mathbb{A}^\theta(\Met)$ of $\Met$.
\end{definition}

We now define the notion of consistency associated to the correspondence-interleaving distance, 
using the density-contour hierarchical clustering $H(f)$ (\cref{density-contour}) 
and the closest point correspondence $R_c$ (\cref{closest-points-correspondence}). 

\begin{definition} \label{CI-and-MI-consistent}
    Let $f : \bbR^d \to \bbR$ be a probability density function with support $\support(f)$. 
    A hierarchical clustering algorithm $\mathbb{A}$ with parameter space $\Theta$ is
    \define{CI-consistent} with respect to $f$ if for every $n \in \bbN$
    there exists a parameter $\theta_n \in \Theta$ such that, for every $\epsilon > 0$ and
    $X_n$ an i.i.d.~$n$-sample of $\support(f)$ with distribution $f$, the probability that
    $\mathbb{A}^{\theta_n}(X_n)$ and $H(f)$ are
    $\epsilon$-interleaved 
    with respect to $R_c$  
    goes to $1$ as $n$ goes to $\infty$.
\end{definition}

\begin{remark}
    In practice, one may want an explicit rule for choosing the parameters $\theta_n$ of \cref{CI-and-MI-consistent} as a function of $n$.
    Moreover, one may also want rates of convergence for the algorithm.
    Although we do not specifically address this in this paper, we mention that such results can be extracted from the proof of the consistency result \cref{consistency-gamma-linkage} together with rates of convergence of samples in the Hausdorff distance \citep{cuevas-casal} and in the Prokhorov distance \citep{dudley-69}.
\end{remark}

We now define Hartigan consistency, following \cite{hartigan-81}.

\begin{definition} \label{definition-cluster-tree}
    Let $X$ be a set. A \define{cluster tree} of $X$ is given by
    a family $\mathcal{T}$ of subsets of $X$ with the property that
    whenever $A$ and $B$ are distinct elements of $\mathcal{T}$, 
    then one of the following is true: $A \cap B = \emptyset$,
    $A \subseteq B$, or $B \subseteq A$.
    The elements of $\mathcal{T}$ are called \define{clusters}.
\end{definition}

\begin{example} \label{associated-cluster-tree}
Let $H : I \to \C(X)$ be a hierarchical clustering of a set $X$.
We can define an associated cluster tree $\calF H = \{C \in H(r) : r \in I\}$.
\end{example}

\begin{definition}
    A \define{cluster tree algorithm} $\mathbb{A}$ with parameter space $\Theta$ is a mapping that assigns to
    each finite metric space $\Met$ and each parameter $\theta \in \Theta$ a cluster tree $\mathbb{A}^\theta(\Met)$
    of $\Met$.
\end{definition}

\begin{definition}[cf.~{\citealp{hartigan-81}}]
    \label{definition-hartigan-consistency}
    Let $f : \bbR^d \to \bbR$ be a probability density function with support $\support(f)$. 
    A cluster tree algorithm $\mathbb{A}$ with parameter space $\Theta$ is \define{Hartigan consistent}
    with respect to $f$ if for every $n \in \bbN$
    there exists a parameter $\theta_n \in \Theta$ such that,
    given $A$ and $A'$ distinct elements of $H(f)(r)$ for some $r > 0$,
    and $X_n$ an i.i.d.~$n$-sample of $\support(f)$ with distribution $f$ we have
    \[  
       P(A_n \cap A'_n = \emptyset) \xrightarrow{n \to \infty} 1,
    \]
    where $A_n$ is the smallest cluster in $\mathbb{A}^{\theta_n}(X_n)$ that contains
    $A \cap X_n$, and $A'_n$ is the smallest cluster in $\mathbb{A}^{\theta_n}(X_n)$ that
    contains $A' \cap X_n$.
\end{definition}

The proof of the following result is in \cref{appendix}.

\begin{proposition} \label{CI-implies-Hartigan}
    Let $f : \bbR^d \to \bbR$ be a continuous and compactly supported probability density function.
    If a hierarchical clustering algorithm $\mathbb{A}$ is CI-consistent with respect to $f$,
    then the associated cluster tree algorithm $\calF\mathbb{A}$ is Hartigan consistent with respect to $f$.
\end{proposition}

\subsection{Consistency of $\lambda$-linkage} \label{DR-and-f}

Let $f : \bbR^d \to \bbR$ be a continuous and compactly supported 
probability density function with support $\support(f)$, 
and let $\mu_f$ be the probability measure defined by $f$. 
We now prove that the hierarchical clustering algorithm 
$\linkage{\lambda}$ is CI-consistent with respect to $f$. 
The strategy is to construct an interleaving 
between $H(f)$ and the $\linkage{\lambda}$ of 
the metric probability space $(\support(f), \mu_f)$. 
Then, the stability of $\linkage{\lambda}$ implies that, 
for a sufficiently good sample $X_n$ of $f$, 
the $\linkage{\lambda}$ of $X_n$ is a good approximation of 
the $\linkage{\lambda}$ of $(\support(f), \mu_f)$. 

However, in order to interleave 
$\linkage{\lambda}$ and $H(f)$, 
we must first reparameterize $\linkage{\lambda}$, 
as discussed in \cref{reparametrize-for-consistency}.


\begin{definition} \label{phi-and-rescaling-lambda}
Let $\lambda = \lambda\con^{x,y}$ for $x,y > 0$ (see \cref{lambda-notation}). 
For $s > 0$, we write $v_s$ for the volume of a ball in $\bbR^d$ of radius $s$. 
Define an order-preserving function $\varphi : (0,y)^{\op} \to \bbRs^{\op}$ 
by $\varphi(r) = \frac{r}{v_{\lambda_s(r)}}$. 
Note that $\varphi$ is a bijection; 
we write $\overline{\lambda} = \lambda \circ \varphi^{-1}$. 
For any metric probability space $\MPS$, 
we write $\linkage{\overline{\lambda}}(\MPS) = \mml^{K}(\MPS)^{\overline{\lambda}}$, 
with $K$ the uniform kernel.
\end{definition}



\begin{theorem} \label{consistency-lambda-linkage}
The hierarchical clustering algorithm $\theoremlinkage{\overline{\lambda}}$ 
with parameter space $\{\lambda\con^{x,y}\}_{x,y > 0}$ is CI-consistent with respect to 
any continuous, compactly supported probability density function $f : \bbR^d \to \bbR$.
\end{theorem}

This is a special case of \cref{consistency-gamma-linkage}, 
which is proved in \cref{appendix}.

\begin{remark}
For any $\lambda \in \{\lambda\con^{x,y}\}_{x,y > 0}$, 
$\linkage{\lambda}$ and $\linkage{\overline{\lambda}}$ 
produce the same underlying cluster tree. 
So, it follows from the preceding theorem that 
the algorithm $\linkage{\lambda}$ 
with parameter space $\{\lambda\con^{x,y}\}_{x,y > 0}$ is Hartigan consistent with respect to 
any continuous, compactly supported probability density function $f : \bbR^d \to \bbR$.
\end{remark}

\section{Structure of One-Parameter Hierarchical Clusterings} \label{Structure-one-parameter}

Barcodes are used in topological data analysis 
to summarize structural information about data 
\citep{edelsbrunner-letscher-zomorodian, carlsson-barcodes-shapes, ghrist-barcodes}. 
Since they were first introduced, 
a rich theory has been developed 
for barcodes 
(see e.g., \citealt{chazal-silva-glisse-oudot}). 
Barcodes can be defined in many different contexts, 
and are used to summarize various geometric and topological properties of different kinds of data. 
In particular, one-parameter hierarchical clusterings have barcodes, 
and these are a key ingredient in the Persistable pipeline. 
See \cref{figure:barcode} for an example of a barcode.

\begin{figure}[btp]
    \centering
    \includegraphics[width=0.6\textwidth]{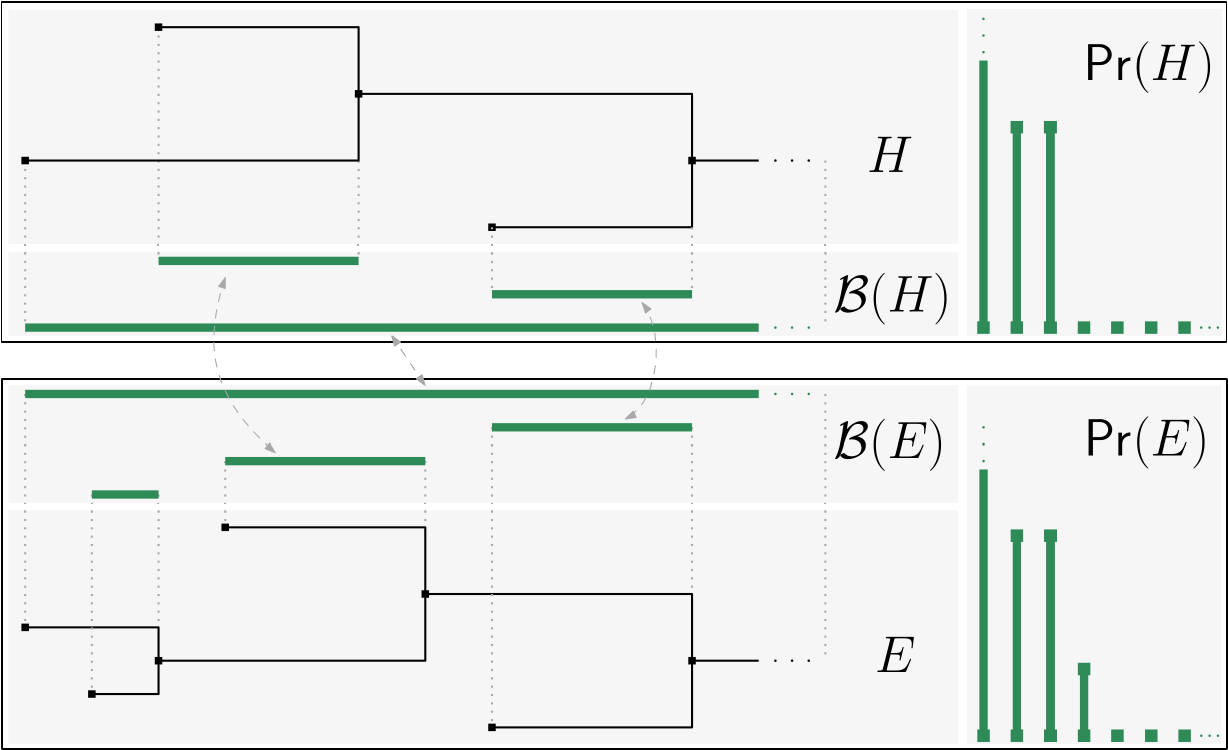}
    \caption{The \textbf{barcode} $\barc(H)$ of a hierarchical clustering $H$ 
    is a collection of real intervals, called \emph{bars} (displayed in green). 
    Informally, the barcode is constructed using the following two rules: 
    (1) If a new cluster enters $H$ at parameter $r$, start a new bar with left endpoint~$r$. 
    (2) If two clusters merge at $r$, take the cluster that entered the hierarchy later 
    (i.e. at a larger parameter value), and end its bar at~$r$. 
    The second rule is called the \emph{elder rule}, since the elder bar survives. 
    In the case of HCs induced by filtered graphs, 
    we give pseudocode for this procedure 
    (\cref{barcode-algorithm}). 
    A \textbf{matching} is shown between the barcodes of $H$ and $E$. 
    The \textbf{prominence diagram} $\prom(H)$ is simply 
    the data of the lengths of the bars in $\barc(H)$ 
    (\cref{subsection:prominence-diagram}).}
    \label{figure:barcode}
\end{figure}

Barcodes of hierarchical clusterings and related structures are a standard topic in topological data analysis 
(see e.g.~\citealt{curry}; 
\citealt{cai-kim-memoli-wang}). 
An important point in practice is that the so-called ``elder rule'' can be used to efficiently compute the barcode 
\citep[Ch. VII.1]{edelsbrunner-harer}. 
In the setting of one-parameter hierarchical clusterings, 
it is possible to define the barcode and describe an algorithm for computing it 
without using any topological or algebraic machinery. 
So, for the benefit of readers who are not already familiar with topological data analysis, 
in this section we provide a definition of the barcode and describe some of its basic properties. 
Some readers may wish to skim this section on a first reading of the paper, 
and refer to it as needed when encountering barcodes.

\subsection{The Poset of Persistent Clusters} 
\label{poset-of-persistent-clusters-def}


We now describe a fundamental object associated to a hierarchical clustering, which we call the poset of persistent clusters.
Picturing a hierarchical clustering as a dendrogram, 
the basic idea is to identify the edges in the dendrogram
and define a partial order on them (see \cref{fig:basic-defs-pic} for an illustration). 
To the best of our knowledge, the poset of persistent clusters was first defined by 
\citet[Appendix A]{statistical-inference-cluster-trees}, 
in the setting of the density-contour hierarchical clustering 
(though they did not use the terminology ``persistent cluster''). 
This construction was also considered by \citet[Section 2.3]{mcinnes-healy} 
in the setting of robust single-linkage (\cref{robust-single-linkage}), 
although phrased in the language of sheaf theory. 
\citet[Section 1]{jardine-stable-components} 
defines an extension of this construction to 2-parameter hierarchical clusterings.
We first define the notion of \emph{persistent cluster}.

\begin{figure}[btp]
    \centering
    \includegraphics[width=0.6\textwidth]{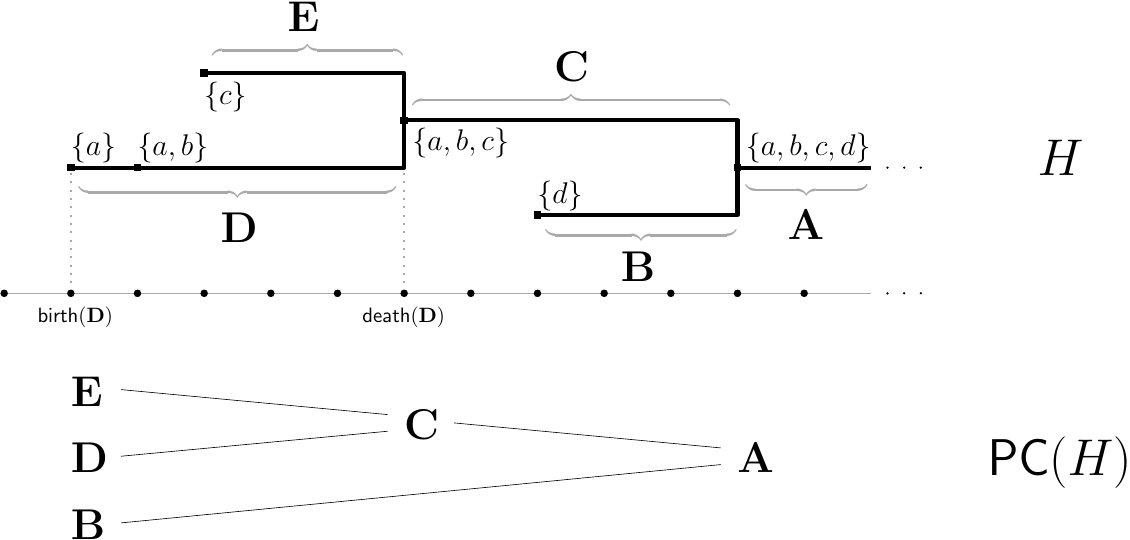}
    \caption{For the hierarchical clustering $H$, the poset of persistent clusters $\PC(H)$ 
    is the poset with elements $\mathbf{A}, \mathbf{B}, \mathbf{C}, \mathbf{D}, \mathbf{E}$, 
    where $\mathbf{E}, \mathbf{D} < \mathbf{C}$, etc. 
    The leaves of $H$ are $\mathbf{E}, \mathbf{D}$, and $\mathbf{B}$. 
    The parameter values $\birth(\mathbf{D})$ and $\death(\mathbf{D})$ are marked; 
    the underlying set of $\mathbf{D}$ is $U(\mathbf{D}) = \{a,b\}$.}
    \label{fig:basic-defs-pic}
\end{figure}

\begin{definition}\label{persistentclusterdef}
Let $X$ be a set.
A \define{persistent cluster} $\bfC$ of $X$ consists of an interval
$\life(\bfC) \subseteq \bbR$
together with an order-preserving function $\bfC : \life(\bfC) \to \parts(X)$, 
where $\parts(X)$ is the power set of $X$ ordered by inclusion. 
The \define{underlying set} of the persistent cluster $\bfC$ is 
$U(\bfC) = \cup_{r \in \life(\bfC)} \bfC(r)$, 
and two persistent clusters are \define{disjoint} if their underlying sets are disjoint. 
Let $\birth(\bfC) = \inf\life(\bfC)$, $\death(\bfC) = \sup\life(\bfC)$, 
and $\length(\bfC) = \death(\bfC) - \birth(\bfC)$.
\end{definition}

We remark that persistent clusters are, in particular, one-parameter hierarchical clusterings, so one can consider interleavings between persistent clusters.

\begin{definition} \label{PC-def}
Let $I \subseteq \bbR$ be an interval and let $H$ be an $I$-hierarchical clustering. 
The \define{poset of persistent clusters} of $H$, denoted $\PC(H)$, 
is the poset whose underlying set is the quotient set $\left( \coprod_{r \in I} H(r) \right) / \sim$ where:
\begin{itemize}
    \item The set $\coprod_{r \in I} H(r)$ denotes the disjoint union of all clusterings as $r$ varies in $I$.
    \item The relation $\sim$ is the symmetric closure of the following relation. 
        For $r_1 \leq r_2$, $C_1 \in H(r_1)$, and $C_2 \in H(r_2)$, we have that $C_1$ and $C_2$ are related if and only if $C_1 \subseteq C_2$ and, for every $r_3 \in [r_1,r_2]$, there is exactly one cluster $C_3 \in H(r_3)$ such that $C_3 \subseteq C_2$. 
\end{itemize}
Let $\bfC \in \PC(H)$.
The equivalence class $\bfC$ is naturally a persistent cluster in the sense of \cref{persistentclusterdef}, with $\life(\bfC) = \{ r \in I : \exists C \in H(r) \; \text{with} \; \bfC = [C] \}$ and such that, for $r \in \life(\bfC)$, we let $\bfC(r) = C$, with $C \in H(r)$ the only cluster in $H(r)$ such that $[C] = \bfC$. 
With this in mind, we define the partial order on $\PC(H)$ by letting $\bfC \leq \bfD$ if $U(\bfC) \subseteq U(\bfD)$.
\end{definition}

The second poset axiom (\cref{poset-def}) for $\PC(H)$ is established in 
\cref{lemma:underlying-clusters-are-disjoint}. 
The other poset axioms follow immediately from the definition.


\begin{definition} \label{leaves-def}
Let $H$ be a one-parameter hierarchical clustering.
The set of \define{leaves} of $H$, denoted $\leaves(H)$, is the set of minimal elements of $\PC(H)$.
\end{definition}

See \cref{fig:basic-defs-pic} for an illustration of the poset of persistent clusters and of the leaves of a hierarchical clustering.

\subsection{Tameness Conditions}
We now define several tameness conditions that one can impose on hierarchical clusterings 
in order to get a notion of a barcode. 
The barcode is most naturally defined for \emph{pointwise finite} HCs. 
However, some HCs of interest may not be pointwise finite 
(see \cref{examples-essentially-finite}). 
So, we introduce a notion of \emph{essentially finite} HCs. 
While essentially finite HCs may not have barcodes, 
they at least have \emph{prominence diagrams}, 
a closely related notion (see \cref{subsection:prominence-diagram}). 

We begin by introducing the \emph{persistence-based pruning} of an HC; 
see \cref{fig:pers-pruning-pic}. 
This pruning procedure is similar in spirit to the pruning of 
\citet[Section 4.2]{statistical-inference-cluster-trees}:
the persistence-based pruning shortens all branches by a chosen amount, making some of them disappear, while the pruning of \citet{statistical-inference-cluster-trees} removes all branches shorter than the chosen amount, and leaves the rest of the branches intact.
In particular, the persistence-based pruning is stable with respect to interleavings (\cref{stability-persistence-pruning}), while the pruning of \citet{statistical-inference-cluster-trees} is not. 
Let $I \subseteq \bbR$ be an interval and let $H : I \to \C(\SET)$ be a one-parameter hierarchical clustering of a set $\SET$.
For $r \leq r' \in I$ we write $H(r \leq r') : H(r) \to H(r')$ for the function 
that takes $C \in H(r)$ to the unique $D \in H(r')$ such that $C \subseteq D$. 

\begin{definition} \label{persistence-pruning-def}
    Let $H$ be an $\bbR$-hierarchical clustering of a set $\SET$.
    Let $\tau \geq 0$.
    The \define{per\-sis\-tence-based pruning} of $H$ with respect to the threshold $\tau$ is the $\bbR$-hierarchical clustering $H_{\geq \tau}$ of $\SET$ such that, for all $r \in I$, we let
\[
	H_{\geq \tau}(r)
    \; := \; \Im\, H(r-\tau \leq r) 
    \; = \; \{ C \in H(r) : \exists D \in H(r-\tau) \text{ with } D \subseteq C \}.
\]
\end{definition}

\begin{figure}
    \centering
    \includegraphics[width=0.6\textwidth]{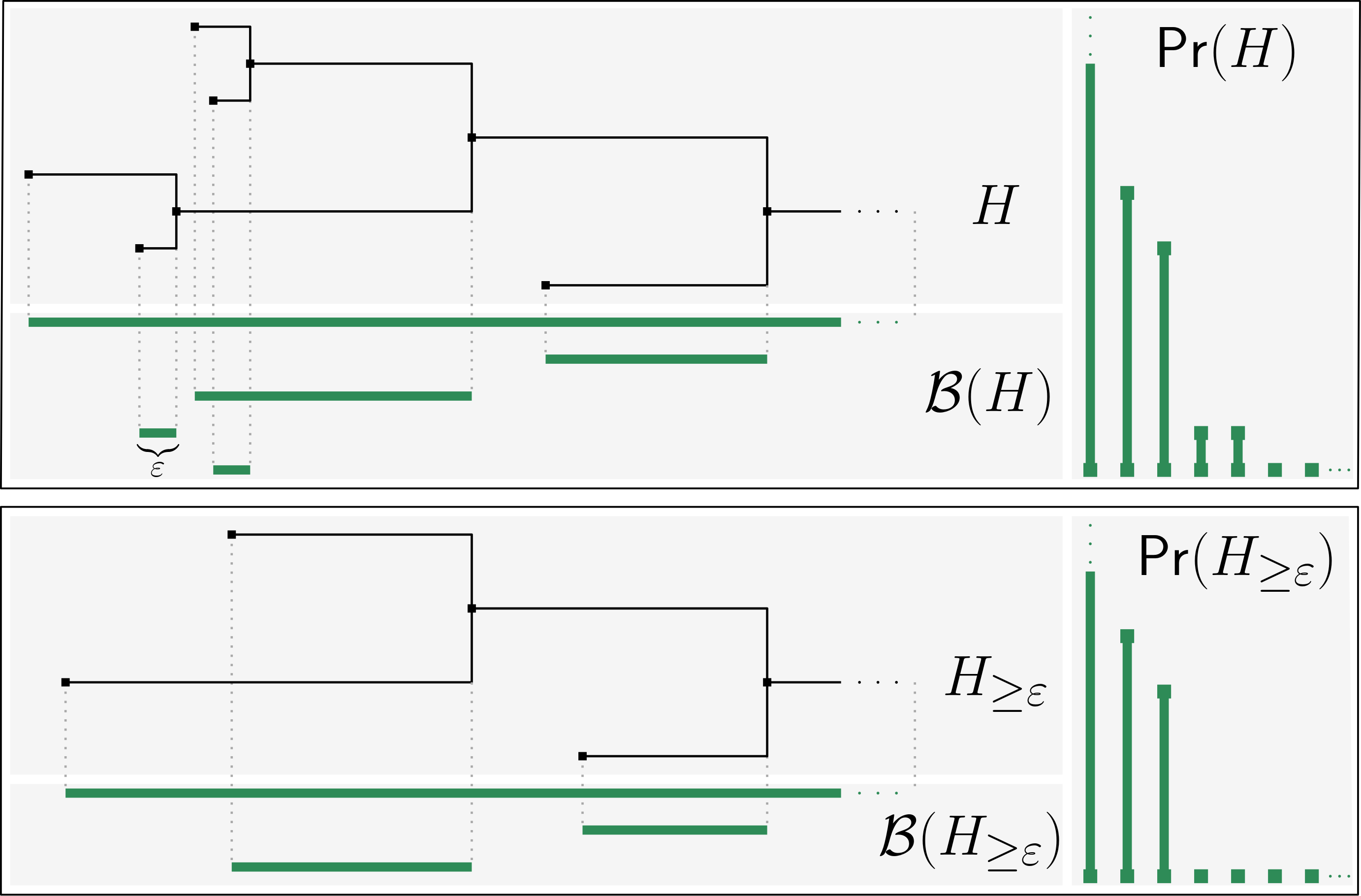}
    \caption{A hierarchical clustering $H$ and the persistence-based pruning $H_{\geq \epsilon}$. 
    The barcode of $H$ contains two short bars, reflecting short leaves of $H$. 
    As the third-longest bar of $H$ is much longer than the fourth-longest, $\gapsize_3(H)$ is large. 
    Pruning $H$ by an appropriate $\epsilon$ removes the short leaves.}
    \label{fig:pers-pruning-pic}
\end{figure}


\begin{definition} \label{finiteness}
An $\bbR$-hierarchical clustering $H$ is 
\define{finite} if $\PC(H)$ is finite;
\define{pointwise finite} if, for all $r \in \bbR$, the cardinality of $H(r)$ is finite; 
and \define{essentially finite} if $H_{\geq \tau}$ is finite for every $\tau > 0 \in \bbR$. 
A one-parameter hierarchical clustering is finite (respectively pointwise finite, essentially finite) 
if its extension (\cref{extension}) is finite (respectively pointwise finite, essentially finite).
\end{definition}

The above notion of finite hierarchical clustering was introduced by 
\citet{statistical-inference-cluster-trees}. 
For readers familiar with the theory of persistence modules, 
we now briefly explain the connection between the other two tameness conditions 
and well-known tameness conditions for persistence modules. 
Given an $\mathbb{R}$-hierarchical clustering $H$ and a choice of field $\mathbb{F}$, 
there is a persistence module $\mathbb{F}H$ generated by $H$ 
(see \cref{appendix-structure-one-parameter} for details). 
Now, an HC $H$ is pointwise finite if and only if $\mathbb{F}H$ is pointwise finite-dimensional 
\citep[Section 3.8]{chazal-silva-glisse-oudot}. 
Say that $H$ is \emph{bounded} if there is $s \leq t \in \mathbb{R}$ 
such that $H$ is constant on $(-\infty, s)$ and $(t, \infty)$. 
As we show in the proof of \cref{lemma-examples-tame}, 
assuming $H$ is bounded, $H$ is essentially finite if and only if $\mathbb{F}H$ is q-tame 
\citep[Section 3.8]{chazal-silva-glisse-oudot}.

\begin{example}
    \label{example:hierarchical-clustering-finite-set}
    Any one-parameter hierarchical clustering $H$ of a finite set $\SET$ is finite, since $|PC(H)|$ is bounded above by the number of subsets of $\SET$, by \cref{lemma:underlying-clusters-are-disjoint}.
\end{example}

We prove the claims in the following example as \cref{lemma-examples-tame}.

\begin{example} \label{examples-essentially-finite}
    For any $\lambda \in \{\lambda\con^{x,y}\}_{x,y > 0}$ and any compact metric probability space $\MPS$, the hierarchical clustering $\linkage{\lambda}(\MPS)$ is essentially finite.
    If $f : \bbR^d \to \bbR$ is continuous and compactly supported, then $H(f)$ is essentially finite.
    Note that, even if $f : \bbR^d \to \bbR$ is continuous and compactly supported, the hierarchical clustering $H(f)$ need not be pointwise finite, as the following simple example with $d=1$ shows.

    Let $h(x) = x \cdot \sin(1/x) + 1$ if $x \neq 0$ and $h(0) = 1$.
    Let $g : \bbR \to \bbR_{\geq 0}$ be continuous, such that $g|_{(-1/2,1/2)} = c > 0$, $g|_{(-\infty,-1) \cup (1,\infty)} = 0$, and such that $f(x) := h(x)g(x)$ integrates to $1$ on $\bbR$; such a function $g$ can be taken to be piecewise linear, or even smooth.
    Then $f$ is a pdf and $H(f)(c)$ has at least as many clusters as there are connected components in $\{x \in \bbR : x \cdot \sin(1/x) \geq 0, x \in (0,1/2)\}=\{x \in \bbR : \sin(1/x) \geq 0, x \in (0,1/2)\} = \{1/y : \sin(y) \geq 0, y > 2\}$, which clearly has countably many connected components.
\end{example}

\subsection{The Barcode}
\label{barcode-HC-finite-set}

We now define the barcode of a pointwise finite one-parameter HC. 
In order to give a definition that does not require any homology theory, 
we follow \citet{carlsson-zomorodian-multidimensional}, 
and define the barcode in terms of the rank invariant, 
also known as persistent Betti numbers \citep{edelsbrunner-letscher-zomorodian}. 
For zero-dimensional homology, which is the relevant context for clustering,
the rank invariant had already been introduced under the name of ``size function'' by \cite{frosini-2, frosini-1, frosini-4};
the multiparameter version of size functions is due to \citet{biasotti-cerri-frosini-giorgi-landi}.

Let $I \subseteq \bbR$ be an interval and let $H : I \to \C(\SET)$ be a one-parameter hierarchical clustering of a set $\SET$.
Given $r \leq r' \in I$, define the \define{rank invariant} of $H$ at $r \leq r'$ as
\[
    \rk(H)(r \leq r') \; := \; | \Im\, H(r\leq r')| 
    \; = \; | \{ C \in H(r') : \exists D \in H(r) \text{ with } D \subseteq C \}|.
\]
We interpret the quantity $\rk(H)(r \leq r')$ as 
the number of clusters in $H(r')$ that have lived for at least $r' - r$ time. 
Equivalently, it is 
the maximum cardinality of a set of clusters in~$H(r)$ that survive, as distinct clusters, until~$H(r')$.

The rank invariant of a pointwise finite $I$-HC is, a priori, a function mapping comparable elements of $I$ to natural numbers, and, as such, can be hard to visualize.
Nevertheless, the function can be encoded as a multiset of subintervals of $I$ in a unique way, as the following theorem asserts.
In the generality of pointwise finite HCs, this theorem follows directly from a theorem of Crawley-Boevey; see \cref{appendix-structure-one-parameter} for a proof of \cref{theorem:barcode-existence-uniqueness}, as well as for the precise definition of multiset.

\begin{theorem}
    \label{theorem:barcode-existence-uniqueness}
    Let $I \subseteq \bbR$ be an interval and let $H$ be a pointwise finite $I$-hierarchical clustering.
    There exists a unique multiset of non-empty intervals $\barc(H) = \{\barA_j \subseteq I\}_{j \in J}$ with the property that, for all $r \leq r' \in \bbR$, we have
    $\rk(H)(r \leq r') = |\{ j \in J : r,r' \in \barA_j \}|$.
\end{theorem}

\begin{definition} \label{barcode-of-HC-def}
    Let $H$ be a pointwise finite one-parameter HC.
    The \define{barcode} of $H$ is the multiset of intervals $\barc(H)$ of \cref{theorem:barcode-existence-uniqueness}.
\end{definition}

\subsubsection{Bottleneck Stability of Barcodes}

A standard way to compare barcodes is the \emph{bottleneck distance}. 
For readers unfamiliar with this notion, we give an intuitive explanation, 
and refer to \citet[Section~3.2]{bauer-lesnick-matchings} for details. 
As in \cref{barcode-of-HC-def}, a \emph{barcode} is a multiset of intervals of the real line. 
For barcodes~$\mathcal{B}$ and $\mathcal{C}$, 
a \emph{matching} between $\mathcal{B}$ and $\mathcal{C}$ 
is a bijection between a sub-multiset $\mathcal{B'}$ of $\mathcal{B}$ 
and a sub-multiset $\mathcal{C'}$ of $\mathcal{C}$. 
For $\delta \geq 0$, a $\delta$-matching is a matching such that every interval of 
$\mathcal{B}$ and $\mathcal{C}$ with length greater than $2\delta$ is matched 
(i.e., is in $\mathcal{B'}$ or $\mathcal{C'}$); 
and such that, when an interval $B \in \mathcal{B'}$ is matched with an interval $C \in \mathcal{C'}$, 
the endpoints of $B$ and $C$ differ by at most $\delta$. 
Then, the bottleneck distance between $\mathcal{B}$ and $\mathcal{C}$ is
\[
	\dB(\mathcal{B}, \mathcal{C}) = 
	\inf \{ \delta \geq 0 \; | \; 
	\exists \text{ a $\delta$-matching between $\mathcal{B}$ and $\mathcal{C}$} \}.
\]

%


We now give a proposition that relates the correspondence-interleaving distance 
between hierarchical clusterings with the bottleneck distance between their barcodes. 
This is just a translation of the well-known bottleneck stability theorem for persistence barcodes 
into the setting of hierarchical clusterings. 
The original versions of this stability result are due to
\citet{amico-frosini-landi} in the case of zero-dimensional homology,
and to \citet{cohen-steiner-edelsbrunner-harer} and \citet{chazal-etal-interleavings} in the case of homology in arbitrary dimension.
We use the explicit formulation of 
\citet[Theorem~6.4]{bauer-lesnick-matchings} 
to easily derive the following proposition 
(see \cref{appendix-structure-one-parameter} for the proof).

\begin{proposition} \label{bottleneck-correspondence-interleaving}
Let $H$ and $E$ be pointwise finite $\bbR$-hierarchical clusterings of sets $X$ and $Y$. 
Let $\epsilon \geq 0$.
If $H$ and $E$ are $\epsilon$-interleaved with respect to a correspondence between $X$ and $Y$, 
then there exists an $\epsilon$-matching between $\barc(H)$ and $\barc(E)$.
In particular, 
$\dB(\PD(H), \PD(E)) \leq \dWI(H,E)$.
\end{proposition}

The stability results for $\lambda$-linkage in \cref{stability-ell-linkage} 
give upper bounds on the cor\-re\-spon\-dence-interleaving distance between certain HCs. 
Combining these with \cref{bottleneck-correspondence-interleaving}, 
one gets upper bounds on the bottleneck distance between the barcodes of these HCs.

\subsubsection{The Barcode of a Finite Hierarchical Clustering} 

We now describe the barcode of a finite hierarchical clustering in terms of its leaves. 
Fix an interval $I \subseteq \bbR$. 
Let $H$ be a finite $I$-hierarchical clustering of a set $\SET$ and let $\bfC \in \leaves(H)$. 
Define $H \setminus \bfC$ to be the $I$-hierarchical clustering of $\SET$ with
\[
    (H\setminus \bfC)(r) = \{ D \in H(r) : [D] \neq \bfC \in \PC(H) \}.
\]


\begin{definition}
Let $H$ be a finite $I$-hierarchical clustering.
Let $\bfC, \bfD \in \leaves(H)$.
We say that $\bfD$ \define{is born earlier} than $\bfC$ if, for every $r \in \life(\bfC)$, there exists $r' \in \life(\bfD)$ such that $r' \leq r$.
A \define{minimal leaf} of $H$ is a leaf $\bfC$ such that $\length(\bfC)$ is minimal among leaves of $H$, 
and such that if $\bfD$ is another leaf of minimal length, 
then $\bfD$ is born earlier than $\bfC$.
\end{definition}

If $H$ is a finite $I$-hierarchical clustering that is not constantly empty, then it admits some minimal leaf.

\begin{proposition}
    \label{proposition:elder-rule}
    Let $H$ be a finite $I$-hierarchical clustering.
    We define a sequence of $I$-hierarchical clusterings $H_0, \dots, H_k$ ending at $k = |\leaves(H)|$.
    Define $H_0 := H$.
    Given $H_i$, let $\bfC_{i+1}$ be any minimal leaf of $H_i$ and define $H_{i+1} := H_i \setminus \bfC_{i+1}$.
    Then, this sequence of hierarchical clusterings is well-defined and $\barc(H) = \{\life(\bfC_i)\}_{1 \leq i \leq k}$.
\end{proposition}


\subsubsection{Computation of Barcodes Using the Elder Rule} 

We now describe an algorithm that uses the elder rule to compute the barcode of a hierarchical clustering induced by a finite filtered graph. 
The problem of computing the barcode of a filtered graph 
has been discussed extensively in the persistent homology literature. 
The textbook of \citet[Ch. VII.2]{edelsbrunner-harer} 
explains how the general persistence algorithm can be optimized in this case. 
\citet{curry} describes the elder rule for so-called Morse sets, 
which are an abstraction of the path components of a Morse function. 
\citet{cai-kim-memoli-wang} describe the elder rule for so-called treegrams, 
which are hierarchical clusterings with a constructibility condition. 
The textbook of \citet[Section~3.5.3]{dey-wang} also describes how the general persistence algorithm 
can be adapted to the case of a filtered graph. 

Despite the wealth of references for this topic, 
we give a description of the elder rule in our setting, 
for the convenience of readers who are not familiar with notions such as simplicial homology.

\begin{definition} \label{filtered-graph-def}
A \define{finite filtered graph} is a pair $(G,f)$, 
where $G$ is a graph on a finite set $\SET$, instantiated as a set of vertices and edges, 
i.e., $G$ is a set of subsets of $\SET$, with $\{x\} \in G$ for all $x \in \SET$, 
and there is an edge between $x$ and $y$ if and only if $\{x,y\} \in G$; 
and $f : G \to \bbR$ is a function such that if 
$\sigma_1 \subseteq \sigma_2$ in $G$, then $f(\sigma_1) \leq f(\sigma_2)$.
A finite filtered graph $(G,f)$ induces a covariant hierarchical clustering 
$H(G,f) : \bbR \to \C(\SET)$, with $H(G,f)(r)$ the set of connected components of the subgraph 
$f^{-1}((-\infty,r]) \subseteq G$.
\end{definition}

We are motivated to consider this case because the hierarchical clustering $\linkage{\lambda}(\Met)$ 
is induced by a finite filtered graph, for any finite metric space $\Met$. 
In more detail, let $\lambda = \lambda\cov^{x,y}$, and let $\sigma = -y/x$, 
as in \cref{lambda-notation}. 
For $a \in M$, let $f(a) = \inf \{ r > 0 : |B(a,r)| \geq (\sigma r + y) \cdot |M| \}$. 
For $a,b \in M$, let $f(\{a,b\}) = \min \left(x, \max \left(f(a), f(b), d_\Met(a,b)\right) \right)$. 
Let $G$ be a minimum spanning tree of the complete graph on $M$, weighted by $f$. 
Then $\linkage{\lambda}(\Met)$ is induced by $(G,f)$.



\begin{algorithm}[H]
\caption{Compute the barcode of the HC induced by a finite filtered graph}
\label{barcode-algorithm}
    \footnotesize
    \begin{algorithmic}[1]
    \Procedure{Barcode}{$G,f$}
    \State Order elements of $G$ as $[\sigma_1, \dots, \sigma_p]$ with $f(\sigma_i) \leq f(\sigma_{i+1})$ and $\sigma_i \subseteq \sigma_j \Rightarrow i \leq j$
    \State Let $\mathtt{conn\_comp} \gets \{\}$ and $\mathtt{barcode} \gets \{\}$
    \For{$1 \leq i \leq p$}
        \If{$\sigma_i = \{x\}$} \Comment{A vertex appears and a connected component is born}
            \State $\mathtt{conn\_comp} \gets \mathtt{conn\_comp} \cup \big\{(\{x\},f(\sigma_i))\big\}$
            \State $\mathtt{barcode} \gets \mathtt{barcode} \cup \left\{[f(\sigma_i), \infty)\right\}$
        \ElsIf{$\sigma_i = \{x,y\}$} \Comment{An edge appears}
            \State Let $(c,u), (d,v) \in \mathtt{conn\_comp}$ be such that $x \in c$ and $y \in d$
            \If{$c \neq d$} \Comment{Two distinct connected components are being merged}
                \State $\mathtt{conn\_comp} \gets \left(\mathtt{conn\_comp} \setminus \big\{(c,u), (d,v)\big\}\right) \cup \big\{(c\cup d,\min(u,v))\big\}$ \Comment{Merge components}
                \State $\mathtt{barcode} \gets \left(\mathtt{barcode} \setminus \big\{[u,\infty), [v,\infty)\big\}\right) \cup \big\{[\min(u,v),\infty), [\max(u,v),f(\sigma_i))\big\}$ \Comment{Elder rule}
            \EndIf
        \EndIf
    \EndFor
    \State Remove from $\mathtt{barcode}$ all intervals of the form $[t,t)$
    \State \textbf{return} $\mathtt{barcode}$
    \EndProcedure
    \end{algorithmic}
\end{algorithm}

As is well-known \citep[Ch. VII.2]{edelsbrunner-harer}, 
\cref{barcode-algorithm} can be implemented to have time complexity in $O(p \, \log p)$, where $p$ is the size of the input graph, that is, the number of vertices plus the number of edges.
To see this, note that the operations between Line 4 and Line 15 of the algorithm, and specifically the check of Line 10, can be implemented with a union-find data structure, 
also known as a disjoint-set data structure \citep{tarjan}, 
which keeps track of the connected components as the graph is filtered by~$f$.
Thus, the time complexity is dominated by that of Line 2, which sorts vertices and edges according to their $f$-value, and which has time complexity in $O(p \log p)$.


While \cref{barcode-algorithm} requires an ordering of the elements of $G$, 
the output of the algorithm does not depend on this ordering:

\begin{lemma} \label{barcode-algorithm-independent-ordering}
The output of \cref{barcode-algorithm} is independent of the ordering of the elements of $G$ 
chosen in Line 2.
\end{lemma}


\begin{proposition} \label{barcode-algorithm-correctness}
    Let $(G,f)$ be a finite filtered graph. 
    When given $(G,f)$ as input, \cref{barcode-algorithm} returns the barcode of $H(G,f)$.
\end{proposition}

The proofs of \cref{barcode-algorithm-independent-ordering} and \cref{barcode-algorithm-correctness} 
are in \cref{appendix-structure-one-parameter}.

\subsection{The Prominence Diagram} \label{subsection:prominence-diagram}
In the theory of barcodes, the lengths of the bars play an important role. 
The length of a bar is called the ``persistence'' 
(\citealt{cohen-steiner-etal-lipschitz}, \citealt[Ch. VII.1]{edelsbrunner-harer}) 
or the ``prominence'' \citep{chazal-guibas-oudot-skraba} of the bar. 
Following \citet{chazal-guibas-oudot-skraba} we adopt the term prominence, 
which avoids confusion with the notion of persistence diagram \citep[Ch. VII.1]{edelsbrunner-harer}. 
We will sort all the prominences of a barcode in descending order, 
and call the result the \emph{prominence diagram}. 
This construction was considered by \citet{bauer-munk-sieling-wardetzky} 
in the setting of mode hunting, 
where the sorted list of prominences (divided by two) was called the ``persistence signature''. 
Our main motivation for considering the prominence diagram is 
the persistence-based flattening algorithm we introduce in \cref{Pruning-and-flattening}. 
As we explain there, parameter selection for this algorithm 
involves choosing a cut-off between long and short bars in a barcode, 
and Persistable provides visualizations of prominence diagrams to guide this choice. 

The proofs of all results in this sub-section are in \cref{appendix-structure-one-parameter}.

\begin{definition} \label{prominence-diagram-def}
    A \define{prominence diagram} consists of a non-increasing function $P : \bbN \to [0,\infty]$ with $P(j) \to 0$ as $j \to \infty$.
    Define a distance $d_\infty$ between prominence diagrams by letting
    \[
        d_\infty(P,Q) = \sup_{i \in \bbN} |P(i) - Q(i)|,
    \]
    for all $P,Q : \bbN \to [0,\infty]$, with the convention that $|\infty - x| = |x - \infty|$ is equal to $\infty$ if $x \in [0,\infty)$ and to $0$ if $x = \infty$.
\end{definition}

\begin{definition} \label{gap-gapsize-def}
    Let $\gapindex \in \bbN_{\geq 1}$.
    The \define{$\gapindex^{\mathrm{th}}$ gap} of a prominence diagram $P : \bbN \to [0,\infty]$ is the (possibly empty) interval $\gap_\gapindex(P) = (P(\gapindex), P(\gapindex-1)) \subseteq [0,\infty]$.
    The \define{$\gapindex^{\mathrm{th}}$ gap size} is the length of the gap, 
    $\gapsize_\gapindex(P) = P(\gapindex-1) - P(\gapindex)$.
\end{definition}

Let $H$ be a finite $\bbR$-hierarchical clustering.
It follows from \cref{proposition:elder-rule} that the barcode $\barc(H) = \{B_j\}_{j \in J}$ of $H$ contains finitely many intervals.
Thus, $\{\length(B_j) \in [0, \infty]\}_{j \in J}$ is a finite multiset of elements of $[0,\infty]$.

\begin{definition} \label{definition:prom-finite}
    Let $H$ be a finite $\bbR$-hierarchical clustering and let $\{\ell_0 , \dots, \ell_k\} \subseteq [0,\infty]$ denote the lengths of the intervals in $\barc(H)$, with repetitions and ordered from largest to smallest.
    The \define{prominence diagram} of a finite $\bbR$-hierarchical clustering $H$ is the decreasing sequence $\prom(H) : \bbN \to [0,\infty]$ such that $\prom(H)(i) = \ell_i$ if $0 \leq i \leq k$ and $\prom(H)(i) = 0$ otherwise.
\end{definition}

It is a consequence of the stability of barcodes (\cref{bottleneck-correspondence-interleaving}) 
that the prominence diagram is stable with respect to the correspondence-interleaving distance:

\begin{lemma}
    \label{lemma:stability-prominence}
    Let $H$ and $E$ be finite $\bbR$-hierarchical clusterings.
    Then 
    \[
    		d_\infty(\prom(H), \prom(E)) \leq 2 \, \dWI(H,E).
    	\]
\end{lemma}

\begin{definition} \label{prom-H-def}
    Let $H$ be an essentially finite one-parameter hierarchical clustering of a set $\SET$ and let $\bar{H} : \bbR \to \C(\SET)$ be its extension as in \cref{extension}.
    By \cref{lemma:stability-prominence} and \cref{stability-persistence-pruning}, the prominence diagrams $\prom(\bar{H}_{\geq \tau})$ converge uniformly as $\tau \to 0$ to a prominence diagram which we denote by $\prom(H)$ and refer to as the \define{prominence diagram} of $H$.
\end{definition}

\begin{notation} \label{gap-of-H-notation}
Let $H$ be an essentially finite one-parameter hierarchical clustering. 
The $\gapindex^{\mathrm{th}}$ prominence gap of $H$ is 
$\gap_\gapindex(H) = \gap_\gapindex(\prom(H))$ 
and the $\gapindex^{\mathrm{th}}$ gap size of $H$ is 
$\gapsize_\gapindex(H) = \gapsize_\gapindex(\prom(H))$.
\end{notation}

We note that \cref{lemma:stability-prominence} 
is true also for essentially finite hierarchical clusterings:

\begin{lemma} \label{lemma:stability-prominence-essentially-finite}
Let $H$ and $E$ be essentially finite $\bbR$-hierarchical clusterings. 
Then 
\[
	d_\infty(\prom(H), \prom(E)) \leq 2 \, \dWI(H,E).
\]
\end{lemma}

\section{Persistence-Based Flattening of One-Parameter Hierarchical Clusterings} \label{Pruning-and-flattening}

For many applications, one needs a clustering of the input data 
(in the sense of \cref{clustering-def}), not a hierarchical clustering. 
We say that a \emph{flattening} algorithm takes a hierarchical clustering of a set $\SET$, 
and returns a clustering of $X$. 
Persistable clusters data by first constructing a hierarchical clustering of the data 
(using the $\linkage{\lambda}$ algorithm from \cref{gamma-linkage-section}), 
and then applying the \emph{persistence-based flattening algorithm}, 
which we introduce in this section.

The most obvious flattening algorithm takes a hierarchical clustering $H$, 
and returns $H(r)$ for some index $r$. 
However, it can happen that $H$ encodes multi-scale clustering structure in the data 
that is not reflected in $H(r)$ for any single choice of $r$. 
We want a flattening algorithm that can extract clusters at multiple scales.

An example of such an algorithm is the ToMATo clustering algorithm \citep{chazal-guibas-oudot-skraba}, 
which computes a flattening of the hierarchical clustering induced by a filtered graph. 
A major advantage of ToMATo is its innovative parameter selection process: 
the user determines how fine the output clustering will be by choosing a merging parameter $\tau$, 
and this choice is guided by the barcode of the hierarchical clustering (\cref{figure:barcode}). 
On a technical level however, one disadvantage of this algorithm is that its output depends on a choice of ordering of the vertices in the input graph, 
and in some use cases there may not be a clear way to make this choice. 
The persistence-based flattening algorithm ($\PF$) is an adaptation of the ToMATo algorithm 
that avoids the dependence on an ordering of the input. 

As input, $\PF$ takes a one-parameter hierarchical clustering $H$. 
We prove a stability theorem for this algorithm that is stated in terms of interleavings; 
so, this result is compatible with our stability and consistency results for $\linkage{\lambda}$. 
Parameter selection is very similar to that of the ToMATo algorithm, however, for $\PF$, 
the user determines how fine the output clustering will be by choosing the number of clusters, 
guided by the barcode of the input. 

In many TDA applications, barcodes are used to distinguish significant features in data from noise. 
A cut-off is chosen between ``long'' and ``short'' bars; 
the long bars correspond to significant features, and the short bars to noise 
\citep{ghrist-barcodes, fasy-lecci-rinaldo-wasserman-balakrishnan-singh}. 
In order to choose the number of clusters for $\PF(H)$, the practitioner chooses how many bars in the barcode of $H$ to regard as significant features. 
If $\gapindex$ bars are chosen, the output of $\PF$ will consist of $\gapindex$ clusters. 
We call the difference between the length of the 
$\gapindex^{\mathrm{th}}$ longest and $(\gapindex+1)^{\mathrm{th}}$ longest bars 
the $\gapindex^{\mathrm{th}}$ \emph{gap size} of $H$. 
This quantity plays the key role in our stability theorem for $\PF$. 
The larger the gap size, the more stable the output will be. 
So, choosing the number of clusters boils down to looking at the barcode of $H$, 
and finding choices of $\gapindex$ such that the $\gapindex^{\mathrm{th}}$ gap size is large.

In this section, we restrict attention to $\bbR$-hierarchical clusterings. 
One can apply the constructions and results 
of this section to any one-parameter hierarchical clustering $H$ 
by first taking the $\extension{H}$ construction from \cref{extension}. 

We now define $\PF$. 
The basic idea is that one can extract a clustering from a one-parameter hierarchical clustering 
by taking the leaves (\cref{leaves-def}, \cref{fig:basic-defs-pic}). 
However, noise in the underlying data can lead to spurious, short leaves. 
So, we first prune the hierarchical clustering $H$ 
by taking the persistence-based pruning $H_{\geq \tau}$ 
(\cref{persistence-pruning-def}, \cref{fig:pers-pruning-pic}). 

The construction uses the $\gapindex^{\mathrm{th}}$ prominence gap of $H$ (\cref{gap-of-H-notation}),
the notion of persistent cluster (\cref{persistentclusterdef}), 
and the prominence diagram $\prom(H)$ (\cref{prom-H-def}).

\begin{definition}
    \label{persistence-based-flattening-def}
Let $H$ be an essentially finite $\bbR$-hierarchical clustering of a set $X$.
Assume that the $\gapindex^{\mathrm{th}}$ prominence gap of $H$ is non-empty.
The \define{persistence-based flattening} of $H$ with respect to the $\gapindex^{\mathrm{th}}$ prominence gap of $H$ is the set of $\gapindex$ pairwise-disjoint persistent clusters of $X$ given by $\PF(H, \gapindex) = \leaves(H_{\geq \tau})$, where $\tau = (\prom(H)(\gapindex-1) + \prom(H)(\gapindex))/2$.
\end{definition}

The output of $\PF$ is a set of pairwise-disjoint persistent clusters. 
This is important for our stability theorem. 
However, if we want a clustering of $X$ in the sense of \cref{clustering-def}, 
we take the underlying set (\cref{persistentclusterdef}) of each persistent cluster in $\PF(H, \gapindex)$.

When the input of $\PF$ is a hierarchical clustering induced by 
a finite filtered graph (\cref{filtered-graph-def}), 
$\PF$ can be computed by adapting the ToMATo algorithm \citep{chazal-guibas-oudot-skraba}. 
This is what we do for our implementation of Persistable \citep{scoccola-rolle-joss}.

In \cref{persistence-based-flattening-def}, 
we take $\tau$ to be the average of 
$\prom(H)(\gapindex-1)$ and $\prom(H)(\gapindex)$ for convenience. 
If one takes a different $\tau$ in the $\gapindex^{\mathrm{th}}$ prominence gap, 
one gets the same clustering of the underlying data by the following proposition, 
which is proved in \cref{appendix-simplification}.

\begin{proposition} \label{stability-PF-in-tau}
Let $H$ be an essentially finite $\bbR$-hierarchical clustering, 
and say $\gapindex \geq 1$ and $\tau, \tau' \in \gap_\gapindex(H)$. 
There is a bijection $m : \leaves(H_{\geq \tau}) \to \leaves(H_{\geq \tau'})$ 
such that for all $\bfC \in \leaves(H_{\geq \tau})$, 
the underlying sets of $\bfC$ and $m(\bfC)$ are equal.
\end{proposition}

There are many ways to measure the similarity between two clusterings of a data set 
(see, e.g., \cite{meila} and references therein), 
so there are many ways one could try to formulate a stability result for a flattening procedure. 
Our approach is based on the fact that $\PF$ produces a set of persistent clusters. 
The following stability theorem guarantees that if $H$ and $E$ are hierarchical clusterings 
that are sufficiently close in the correspondence-interleaving distance, 
then the persistent clusters in $\PF(H, \gapindex)$ are interleaved with the persistent clusters in $\PF(E, \gapindex)$. 
Here, the $\gapindex^{\mathrm{th}}$ gap size of $H$ determines what ``sufficiently close'' means. 
The proof of the theorem is in \cref{appendix-simplification}. 

\begin{theorem}\label{stability-tomato-flattening}
Let $H$ and $E$ be essentially finite 
$\bbR$-hierarchical clusterings of sets $X$ and $Y$ respectively. 
Let $\gapindex \geq 1$, and assume there is 
$\epsilon < \gapsize_\gapindex(H) / 16$ 
such that $H$ and $E$ are $\epsilon$-interleaved with respect to a correspondence $R \subseteq X \times Y$. 
Then there is a bijection $m : \PF(H, \gapindex) \to \PF(E, \gapindex)$ 
such that for all $\bfC \in \PF(H, \gapindex)$, 
$\bfC$ and $m(\bfC)$ are $3\epsilon$-interleaved with respect to $R$.
\end{theorem}

The interleavings guaranteed by this theorem imply that, 
if $\bfC \in \PF(H, \gapindex)$, and $x \in U(\bfC)$ appears early enough in the lifetime of $\bfC$, 
then every point in $Y$ that corresponds to $x$ under $R$ must belong to $U(m(\bfC))$.

Because this stability theorem for persistence-based flattening is stated in terms of interleavings, 
it can be combined with the stability and consistency results 
proved earlier in this paper. As an example, we state the following stability results 
for $\linkage{\lambda}$ (\cref{lambda-notation}). 
The combination of $\linkage{\lambda}$ and persistence-based flattening 
is the core algorithm of Persistable.

The first result concerns stability in the input data:

%

\begin{corollary} \label{stability-lambda-and-PF}
Let $\MPS$ be a compact metric probability space, let $\lambda = \lambda\cov^{x,y}$ 
with slope $\sigma$, and assume $\gap_\gapindex(\theoremlinkage{\lambda}(\MPS))$ is non-empty. 
Let $\MPSalt$ be a compact metric probability space with
\[
	\dGHP(\MPS, \MPSalt) < 
	\frac{\gapsize_\gapindex(\theoremlinkage{\lambda}(\MPS))}{16 \cdot \max(|1 / \sigma|, 2)} \, .
\]
There is a bijection 
$m: \PF(\theoremlinkage{\lambda}(\MPS), \gapindex) \to 
\PF(\theoremlinkage{\lambda}(\MPSalt), \gapindex)$ 
such that, for all $\bfC \in \PF(\theoremlinkage{\lambda}(\MPS), \gapindex)$, 
$\bfC$ and $m(\bfC)$ are $3\epsilon$-interleaved with respect to a correspondence between 
$\MPS$ and $\MPSalt$, for some 
$\epsilon < \gapsize_\gapindex(\theoremlinkage{\lambda}(\MPS)) / 16$.
\end{corollary}

\begin{proof}
By \cref{stability-lambda-linkage-in-X}(2), 
$\dWI(\linkage{\lambda}(\MPS), \linkage{\lambda}(\MPSalt)) \leq 
\dGHP(\MPS,\MPSalt) \cdot \max(|1 / \sigma|, 2)$. 
So, we can take $\epsilon$ with 
$\dWI(\linkage{\lambda}(\MPS), \linkage{\lambda}(\MPSalt)) < \epsilon < 
\gapsize_\gapindex(\theoremlinkage{\lambda}(\MPS)) / 16$. 
Then the result follows from \cref{stability-tomato-flattening}.
\end{proof}

The second result concerns stability in the choice of $\lambda$:

\begin{corollary} \label{stability-lambda-and-PF-in-lambda}
Let $\MPS$ be a compact metric probability space, 
let $\lambda = \lambda\cov^{x,y}$ with slope $\sigma$, and 
let $\lambda' = \lambda\cov^{x',y'}$ with slope $\sigma'$. Say 
\[
	\epsilon := \max \left( |x - x'|, \; |y - y'| \cdot \min(|1/\sigma|, |1/\sigma'|)  \right) 
	\, < \, 	\gapsize_\gapindex(\theoremlinkage{\lambda}(\MPS)) / 16 \, .
\]
Then there is a bijection 
$m: \PF(\theoremlinkage{\lambda}(\MPS), \gapindex) \to \PF(\theoremlinkage{\lambda'}(\MPS), \gapindex)$ 
such that, for all $\bfC \in \PF(\theoremlinkage{\lambda}(\MPS), \gapindex)$, 
$\bfC$ and $m(\bfC)$ are $3\epsilon$-interleaved.
\end{corollary}

\begin{proof}
By \cref{stability-lambda-linkage-in-lambda}, 
$\theoremlinkage{\lambda}(\MPS)$ and $\theoremlinkage{\lambda'}(\MPS)$ 
are $\epsilon$-interleaved. 
So, the result follows from \cref{stability-tomato-flattening}.
\end{proof}

The ToMATo algorithm takes as input a finite graph and a real-valued function $f$ on its vertices. 
This induces the upper-star filtration on the graph, 
where an edge $\{x,y\}$ appears at $\min\{f(x),f(y)\}$. 
We now describe the \emph{exhaustive persistence-based flattening algorithm} 
(\cref{tomato-flattening-algorithm}), 
which is essentially a generalization of ToMATo to the more general filtered graphs of \cref{filtered-graph-def}. 
This is of interest because the $\linkage{\lambda}$ of a finite metric space 
is induced by a filtered graph, but not by an upper-star filtration. 
We describe the precise relationship between {\sc ExhaustivePF} and ToMATo 
in \cref{exhaustive-PF-ToMATo} in \cref{appendix-simplification}.

We call the algorithm ``exhaustive'' because, unlike $\PF$, 
{\sc ExhaustivePF} clusters every point in its input. 
For $\PF$, points that enter $H_{\geq \tau}$ outside of a leaf do not get clustered. 
{\sc ExhaustivePF} uses the data of the input graph 
and an ordering of its simplices to assign such points to some leaf. 
For Persistable, we prefer $\PF$ to {\sc ExhaustivePF}, 
because of the good stability properties of $\PF$, 
and the fact that it does not depend on an ordering of the input. 
However, {\sc ExhaustivePF} also produces interesting results 
(see the Olive oil data in \cref{Persistable-on-benchmark-data} for an example). 

\begin{algorithm}
\caption{Exhaustive persistence-based flattening of the HC induced by a finite filtered graph}
\label{tomato-flattening-algorithm}
    \footnotesize
    \begin{algorithmic}[1]
    \Procedure{ExhaustivePF}{$G = [\sigma_1, \dots, \sigma_p],f,\tau$}
    \LineComment{Assume the simplices of $G$ are ordered such that $f(\sigma_i) \leq f(\sigma_{i+1})$ and $\sigma_i \subseteq \sigma_j \Rightarrow i \leq j$}
    \State Let $\mathtt{clusters} \gets \{\}$
    \For{$1 \leq i \leq p$}
        \If{$\sigma_i = \{x\}$} \Comment{A vertex appears and a cluster is born}
            \State $\mathtt{clusters} \gets \mathtt{clusters} \cup \big\{(\{x\},f(\sigma_i))\big\}$
        \ElsIf{$\sigma_i = \{x,y\}$} \Comment{An edge appears}
            \State Let $(c,u), (d,v) \in \mathtt{clusters}$ be such that $x \in c$ and $y \in d$
            \If{$c \neq d$} 
                \If{$f(\sigma_i) - u \leq \tau$ or $f(\sigma_i) - v \leq \tau$} \Comment{At least one cluster did not persist enough}
                    \State $\mathtt{clusters} \gets \left(\mathtt{clusters} \setminus \big\{(c,u), (d,v)\big\}\right) \cup \big\{(c\cup d,\min(u,v))\big\}$ \Comment{Merge clusters}
                \EndIf
            \EndIf
        \EndIf
    \EndFor
    \State \textbf{return} $\mathtt{clusters}$
    \EndProcedure
    \end{algorithmic}
\end{algorithm}

\section{Persistable} \label{Persistable}

Persistable is a pipeline for density-based clustering that integrates the algorithms defined in this paper. 
In another publication \citep{scoccola-rolle-joss}, we described the implementation of Persistable. 
In this section, we describe the design choices of Persistable in detail, 
and explain how these choices are motivated by the theoretical results in this paper.

Compared to existing clustering methods, 
a novel feature of Persistable is that 
the parameter selection process is guided by interactive visualization tools. 
This parameter selection process is based on the stability results proved in this paper; 
the visualization tools are inspired by tools 
from multiparameter persistent homology, 
in particular, the software library RIVET \citeyearpar{rivet}. 


We begin by demonstrating the Persistable pipeline on a simple running example, 
and then we evaluate its performance on real-world benchmark data sets. 
See the Persistable software repository 
(link available in \citealt{scoccola-rolle-joss}) 
for code that replicates all the examples in this section, 
as well as for further evaluations of Persistable on benchmark data sets.

\subsection{The Persistable Pipeline} \label{persistable-workflow}

\begin{figure}[tb]
    \centering
    \includegraphics[width=0.45\textwidth]{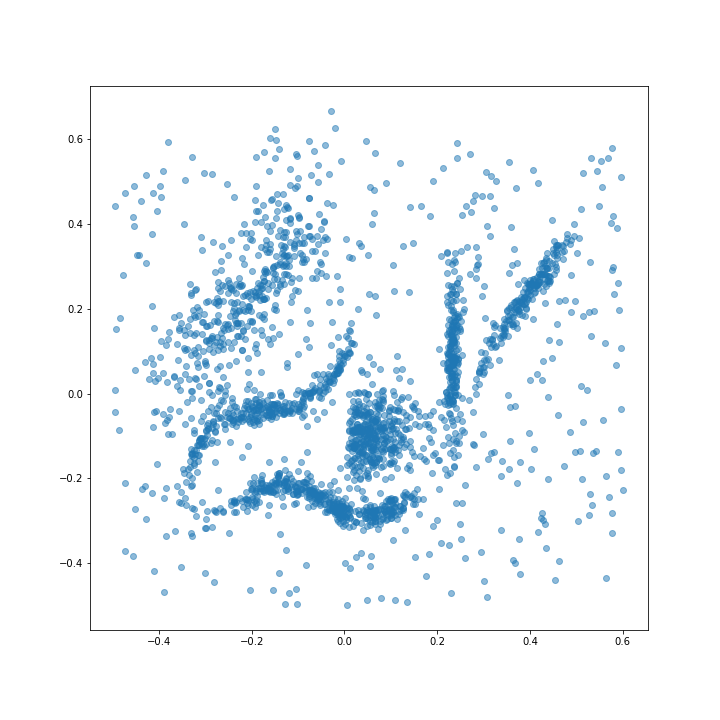}
    \includegraphics[width=0.45\textwidth]{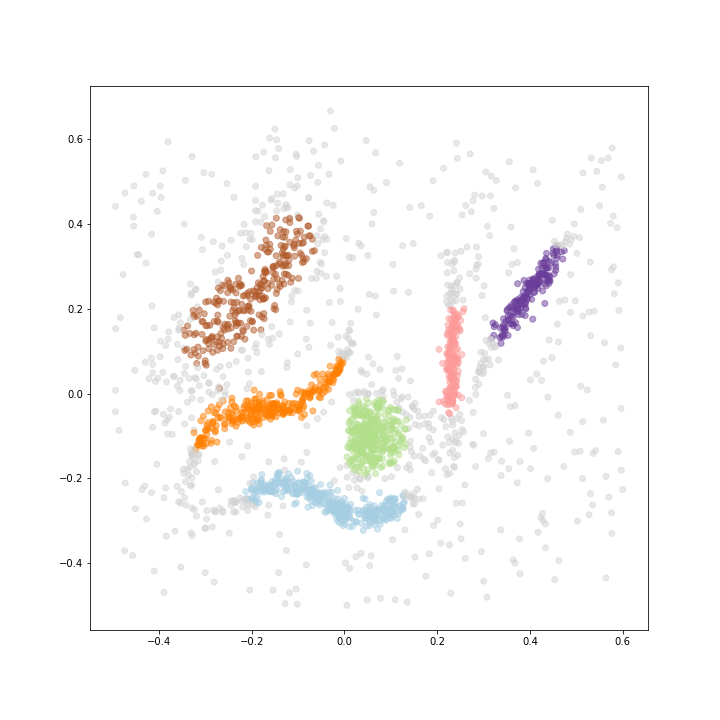}
    \caption{The running example. 
    We cluster the data with Persistable: 
    clusters are indicated by colors, 
    and gray points are classified as noise.}
    \label{hdbscan-data}
\end{figure}

\begin{figure}[p]
    \centering
    \includegraphics[width=\twC\textwidth]{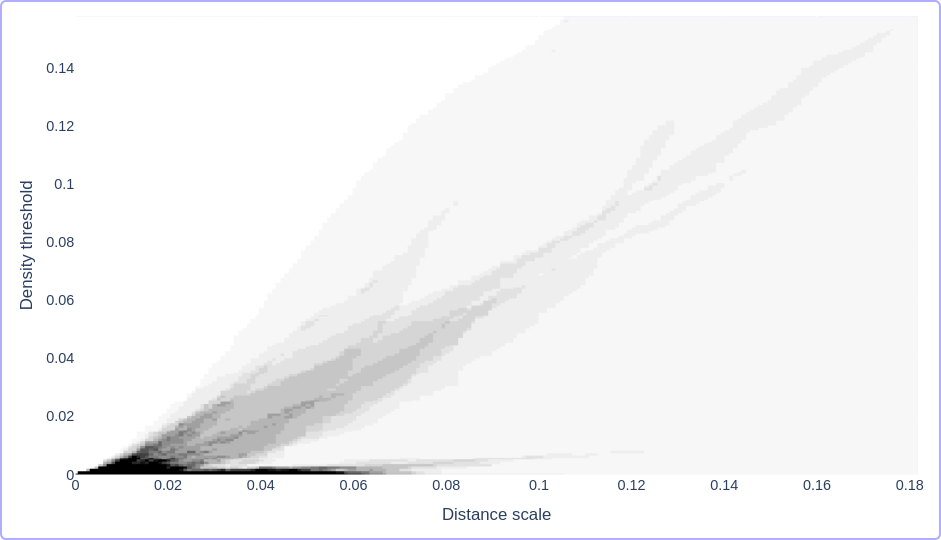}
    \includegraphics[width=\twC\textwidth]{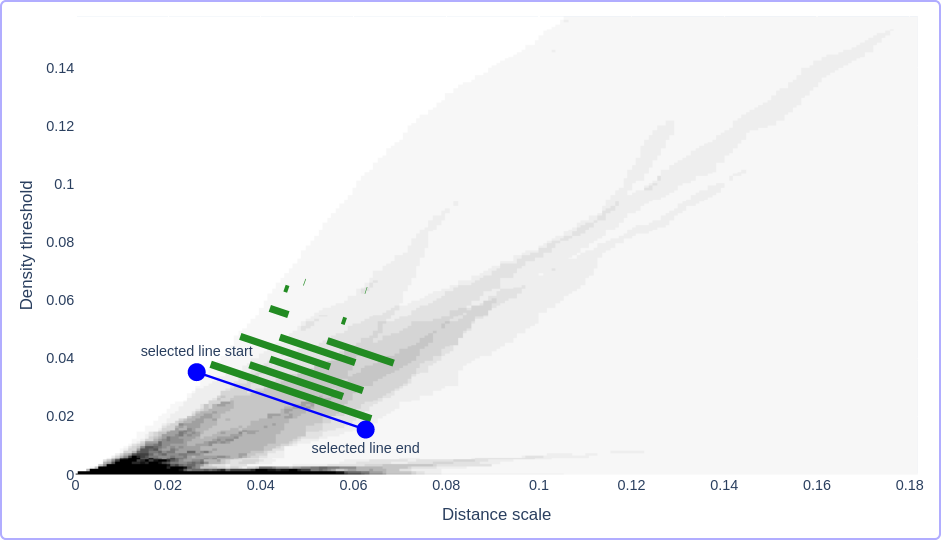}
    \caption{The first Persistable visualization is \textbf{the component counting function}; 
    we show this for the running example. 
    One can see typical behavior of degree-Rips ($\DR$): 
    when the distance scale $s$ is small and the density threshold $k$ is large, 
    no points are clustered; when $s$ is large and $k$ is small, 
    all points are clustered together. 
    In between these two regimes is a band of interesting cluster structure.
    The blue line segment defines a slice of $\DR$, and its barcode is plotted in green. 
    The sixth gap size of this slice is quite large 
    (the sixth-longest bar is much longer than the seventh-longest bar). 
    So, we choose six clusters for the persistence-based flattening. 
    The output is displayed in \cref{hdbscan-data}.}
    \label{component-counting-function}
\end{figure}

\begin{figure}[p]
    \centering
    \includegraphics[width=\twC\textwidth]{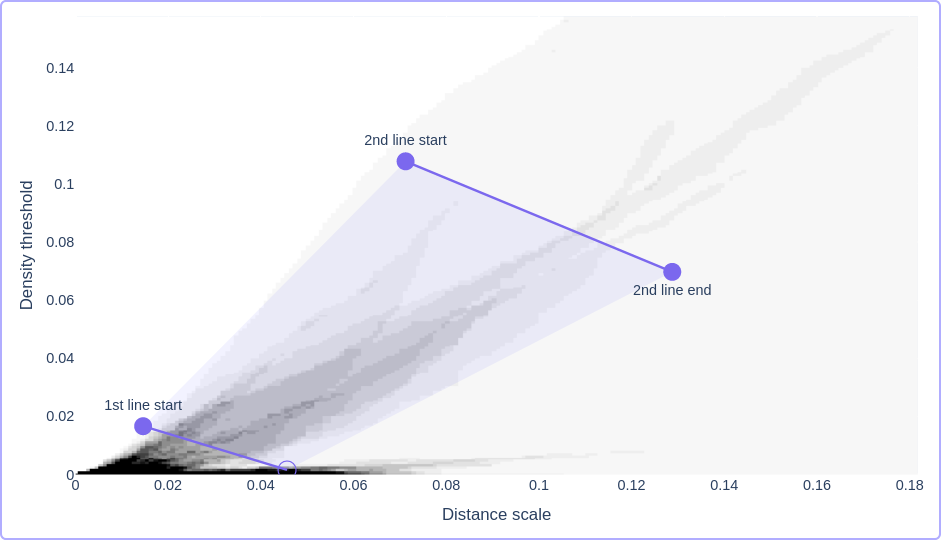}
    \includegraphics[width=\twC\textwidth]{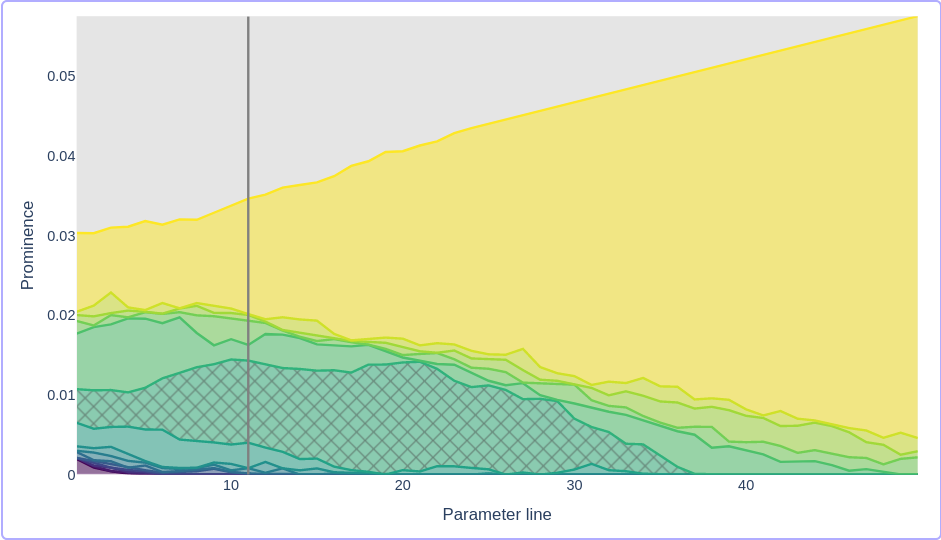}
    \caption{The second Persistable visualization is the \textbf{prominence vineyard}. 
    The user chooses two slices of degree-Rips by choosing the start and end points. 
    This determines a one-parameter family of slices that interpolate 
	between the two chosen slices. 
	For each slice in the family, Persistable computes the barcode, 
	and the prominence vineyard (on the right) displays 
	the prominence (i.e., length) of each bar in the barcode. 
	So, the first (top) curve in the prominence vineyard shows the length 
	of the longest bar in the barcode of each slice, 
	the second curve shows the length of the second-longest bar, etc. 
	We call these curves \emph{vines}. 
	The first gap (between the first and second vines, displayed in yellow) 
	is very large: this just reflects the fact that the longest bar of an HC 
	is typically much longer than any other bar. 
	The sixth gap (marked in the figure) is also large. 
	This means there are many slices in the vineyard 
	such that the sixth gap size of the slice is large. 
	We used this prominence vineyard to choose the slice in \cref{component-counting-function}.}
    \label{prominence-vineyard-figure}
\end{figure}

As input, Persistable takes a finite metric space $\Met$, 
and produces a clustering of $\Met$, in the sense of \cref{clustering-def}. 
As a running example, we use a synthetic data set from 
the \textsf{hdbscan} clustering library \citep{mcinnes-healy-astels-joss}. 
This data set is designed to be challenging for clustering algorithms, 
while being easy to visualize: see \cref{hdbscan-data}.

Conceptually, Persistable begins with the degree-Rips hierarchical clustering $\DR(\Met)$ 
that was described in the introduction 
(see \cref{DR-intro-figure}, and see \cref{def-DR} for the formal definition). 
We can get insight into $\DR(\Met)$ by plotting the 
\define{component counting function}, 
which is the function defined on the first quadrant of the plane 
where at $(s,k)$ we simply count the number of clusters in $\DR(\Met)(s,k)$. 
The first visualization in the Persistable pipeline 
is a heat map of the component counting function. 
See \cref{component-counting-function} 
for this visualization on the running example.

Now, Persistable constructs a clustering of the input $\Met$ in two steps.

\subsubsection{Step One: Reduce a Two-Parameter HC to a One-Parameter HC}
The first step is to reduce from the two-parameter hierarchical clustering $\DR(\Met)$ 
to a one-parameter hierarchical clustering, by taking a slice 
(see \cref{slicing-intro-figure}). 
Using the notation established in \cref{lambda-notation}, 
this means taking $\linkage{\lambda}(\Met)$, 
where $\lambda$ is a choice of line in the $(s,k)$-parameter space. 
For example, see the blue line segment in \cref{component-counting-function}. 
The practitioner can choose such a slice 
by clicking on the component counting visualization 
to choose the start and end points of a line segment. 
This determines an interval $(s_{\mathrm{start}}, s_{\mathrm{end}})$ 
on which the one-parameter hierarchical clustering $\linkage{\lambda}(\Met)$ 
is defined, and for $s \in (s_{\mathrm{start}}, s_{\mathrm{end}})$, 
$\linkage{\lambda}(\Met)(s) = \DR(\Met)(s, \sigma \cdot s + y)$, 
where $y$ is the $y$-intercept of the selected line, 
and $\sigma$ is its slope. 
The second visualization in the Persistable pipeline is an interactive tool for choosing slices. 
We introduce this tool after we discuss the second step of Persistable.

\subsubsection{Step Two: Reduce a One-Parameter HC to a Clustering}
The second step is to reduce from the one-parameter hierarchical clustering $\linkage{\lambda}(\Met)$ 
to a clustering of $\Met$, by applying the persistence-based flattening procedure, 
defined in \cref{Pruning-and-flattening}. 
To apply the persistence-based flattening, 
one chooses the number of clusters in the output, 
guided by the barcode of the HC. 
The barcode is a visual summary of an HC (see \cref{figure:barcode}). 
If one chooses $n$ clusters, 
these will correspond to the $n$ longest bars in the barcode. 
We call the difference between the length of the 
$\gapindex^{\mathrm{th}}$ longest and $(\gapindex+1)^{\mathrm{th}}$ longest bars 
the $\gapindex^{\mathrm{th}}$ \emph{gap size} of the HC. 
As explained in \cref{Pruning-and-flattening}, 
the larger the  $\gapindex^{\mathrm{th}}$ gap size, 
the more stable the output with $\gapindex$ clusters will be. 
So, choosing the number of clusters boils down to looking at the barcode, 
and finding choices of $\gapindex$ such that the $\gapindex^{\mathrm{th}}$ gap size is large.

In the case of the running example, 
there is a drop-off between the sixth and seventh longest bars, 
so choosing six bars (i.e., six clusters) is a reasonable choice 
(see \cref{component-counting-function}). 
The resulting clustering of the data is displayed in \cref{hdbscan-data}.

\subsubsection{Choosing the Slice}
To complete the description of the Persistable pipeline, 
it remains to discuss how the practitioner chooses a slice. 
The answer is motivated by Step 2. 
When one applies the persistence-based flattening, 
one looks for large gap sizes in the barcode of $\linkage{\lambda}(\Met)$. 
So, the second visualization tool in the Persistable pipeline 
helps the user identify slices $\lambda$ that lead to barcodes with large gap sizes.

The practitioner begins with the component counting visualization 
(\cref{component-counting-function}). 
From this, one can find 
the region of the $\DR$ parameter space where interesting cluster structure is captured. 
The practitioner is asked to choose two slices in the $\DR$ parameter space, 
which determine a one-parameter family of slices that interpolate 
between the two chosen slices. 
As $\lambda$ varies in the family, 
the slice $\linkage{\lambda}(\Met)$ changes in a continuous way 
by \cref{stability-lambda-linkage-in-lambda}. 
Thus, the barcode of $\linkage{\lambda}(\Met)$ also changes in a continuous way. 
In the \textbf{prominence vineyard} visualization, 
Persistable plots the \emph{prominence} (i.e., length) 
of each bar in the barcode of $\linkage{\lambda}(\Met)$. 
As $\lambda$ varies, these prominences trace out curves, which we call \emph{vines} 
(this is standard terminology in topological data analysis, beginning with 
\citealp{cohen-steiner-edelsbrunner-morozov}).
See \cref{vineyard-intro-figure}. 
In \cref{prominence-vineyard-figure}, we display a prominence vineyard for the running example. 
There is a large gap between the first and second vines; 
this is typical behavior, as the longest bar in the barcode of an HC 
is just the interval on which the HC is non-empty, which is often much longer than any other bar. 
More interesting structure is captured by the other gaps in the prominence vineyard. 
For example, there is a large gap between the sixth and seventh vines, 
which is marked in \cref{prominence-vineyard-figure}. 
If we choose a slice that includes this gap, 
then the sixth gap size of the barcode of this slice is large. 
Indeed, this is how we picked the slice that we used in Step 2, above.

\subsection{Examples of Persistable on Benchmark Data Sets} \label{Persistable-on-benchmark-data}

We now demonstrate how Persistable can identify meaningful cluster structure in data.

\subsubsection{Rideshare Data}
We consider a data set consisting of approximately 560\,000 rideshare pickup locations 
in the New York City area from April, 2014. 
The data set is the result of a Freedom of Information request 
by the website \cite{rideshare-data}. 
This data set has very complex cluster structure, at many different scales. 
There are informative clusterings at a very coarse level, 
and also at much finer levels. 

To get a feel for the data, we begin by considering the subset of points 
in a square centered at LaGuardia Airport (see \cref{laguardia-data-fig}), 
which consists of approximately 10\,000 points.

\begin{figure}[p]
    \centering
    \begin{overpic}[width=0.49\textwidth]{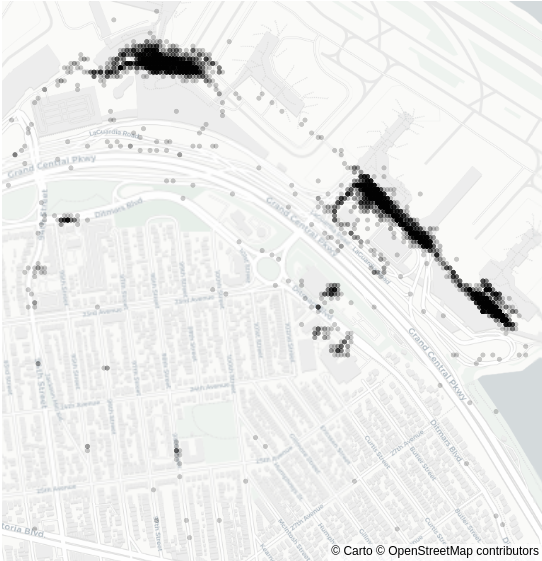}
	\put (52,49) {{\footnotesize M}}
	\put (52,40) {{\footnotesize H}}
	\put (60,32) {{\footnotesize P}}
	\put (33,17) {{\footnotesize S}}
	\put (83,55) {{\footnotesize C}}
	\put (8,45) {{\footnotesize A}}
	\put (11,55) {{\footnotesize N}}
	\put (28,80) {{\footnotesize B}}
	\end{overpic}
    \includegraphics[width=0.49\textwidth]{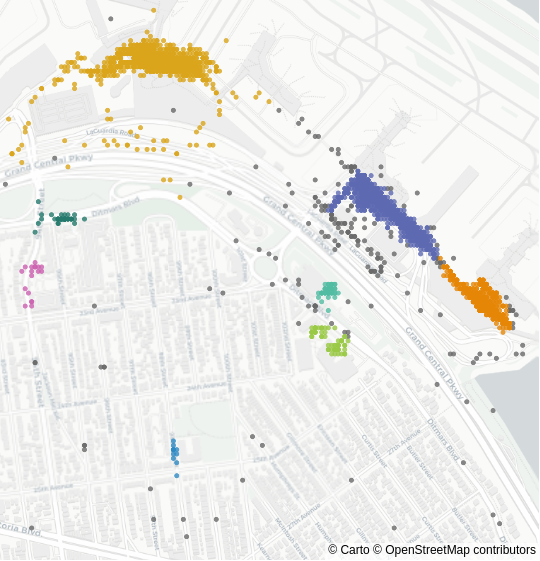}
    \caption{The subset of the Rideshare data in a square centered at LaGuardia Airport. 
    On the left, some relevant landmarks are marked: 
    ACE Rent A Car (A), 
    National Car Rental (N), 
    LaGuardia Airport Terminal B (B), 
    Terminal C (C), 
    Marriott Hotel (M), 
    Hampton Inn Hotel (H), 
    LaGuardia Plaza Hotel (P), 
    P.S. 127 Aerospace Science Magnet School (S). 
    On the right, a clustering produced by Persistable, 
    using the slice displayed in \cref{Persistable-laguardia-visualizations}. 
    Gray points are unclustered.}
    \label{laguardia-data-fig}
\end{figure}

\begin{figure}[p]
    \centering
    \includegraphics[width=\twC\textwidth]{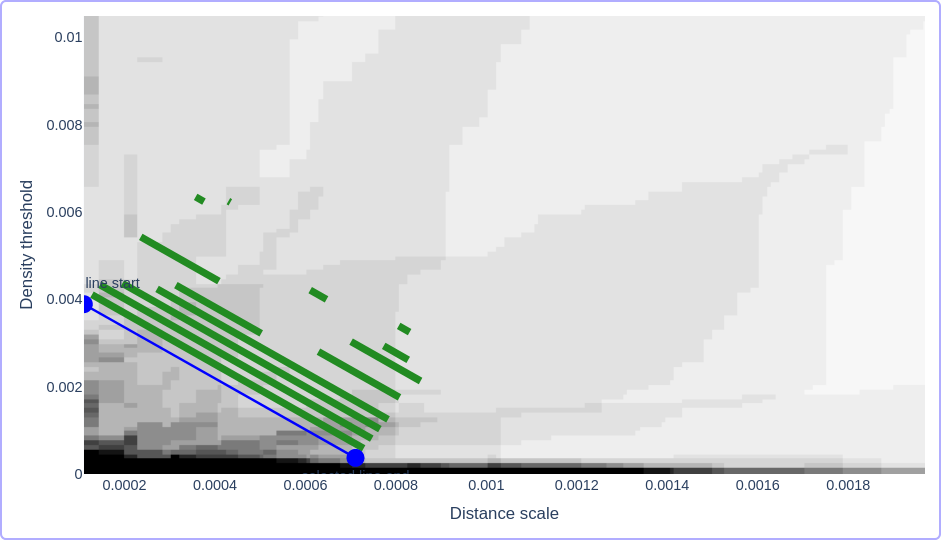}
    \includegraphics[width=\twC\textwidth]{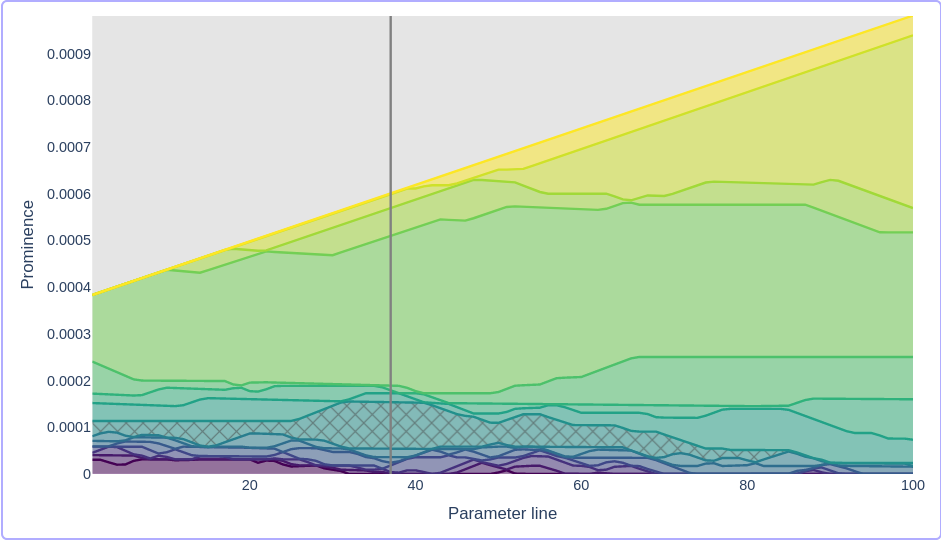}
    \caption{Persistable visualizations for 
    the subset of the Rideshare data near LaGuardia Airport. 
    On the left, the component counting visualization, 
    and on the right, a prominence vineyard visualization. 
    We use the prominence vineyard to choose a slice, 
    which appears as the blue line segment on the component counting visualization; 
    the corresponding slice in the vineyard is marked by a vertical line. 
    The barcode of this slice is displayed in green.
    Several gaps in this vineyard lead to interesting clusterings. 
    If we choose the gap below the eighth vine (marked in this figure), 
    we get the clustering of the data displayed in \cref{laguardia-data-fig}. 
    See \cref{appendix-Persistable-laguardia-visualizations} 
    in \cref{appendix-Persistable} 
    for the result of choosing the gap below the fourth vine.}
    \label{Persistable-laguardia-visualizations}
\end{figure}

\begin{figure}[p]
    \centering
    \includegraphics[width=\twC\textwidth]{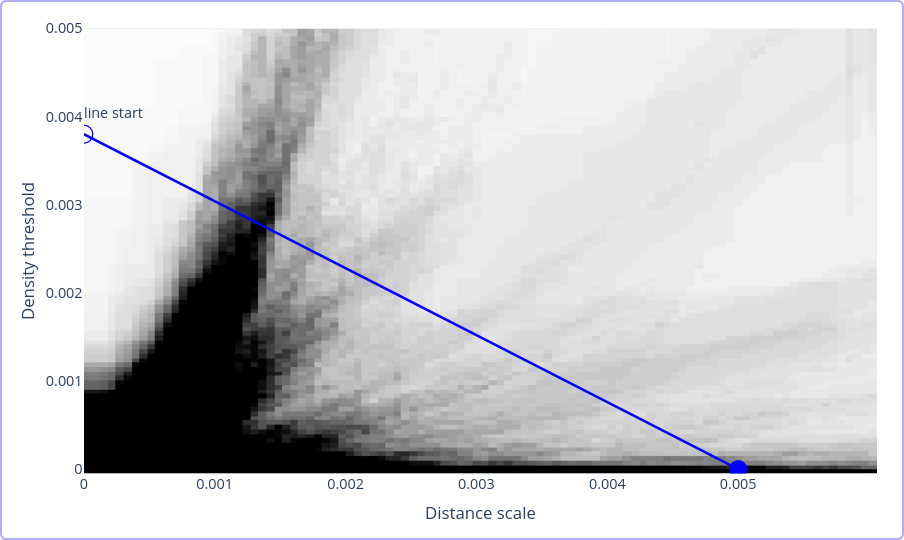}
    \includegraphics[width=\twC\textwidth]{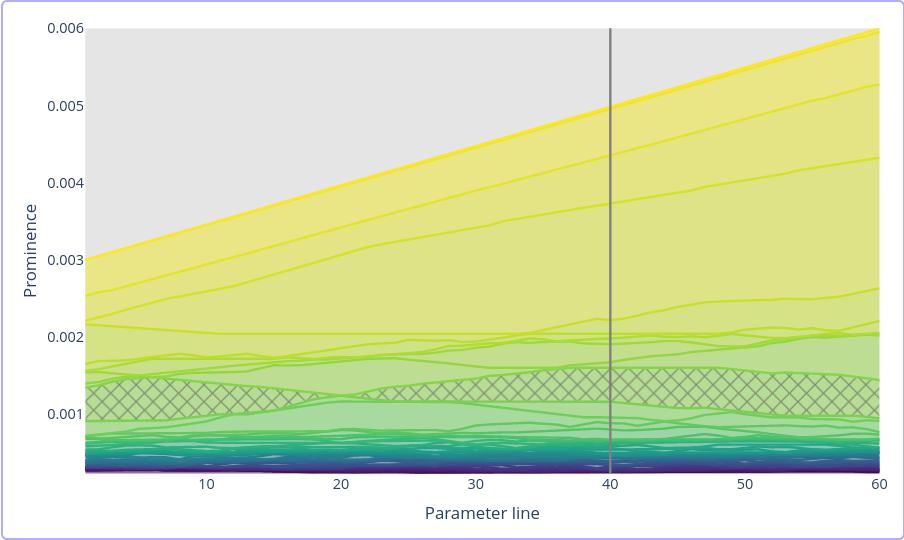}
    \caption{Persistable visualizations for the Rideshare data. 
    On the left, the component counting visualization, 
    and on the right, a prominence vineyard visualization. 
    While the component counting visualization is very complicated, 
    one can easily identify interesting gaps in the prominence vineyard. 
    We use the prominence vineyard to choose a slice, 
    marked on the component counting function by a blue line and on the vineyard by a vertical line. 
    If we choose the marked gap, 
    we get the clustering of the data displayed in \cref{rideshare-data-fig}.}
    \label{Persistable-rideshare-visualizations}
\end{figure}

\begin{figure}[p]
    \centering
    \includegraphics[width=0.49\textwidth]{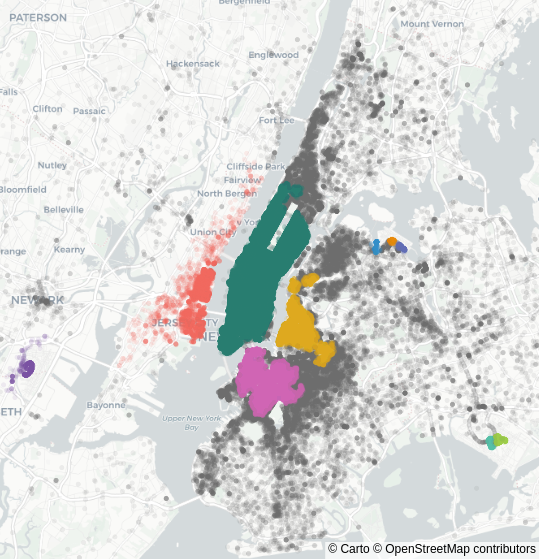}
    \includegraphics[width=0.49\textwidth]{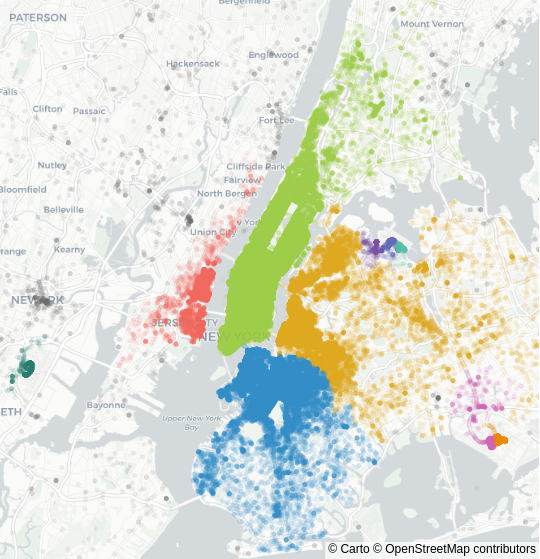}
    \caption{The Rideshare data, clustered using Persistable. 
    Gray points are unclustered. 
    In both cases, the clustering is obtained by choosing the slice 
    indicated in \cref{Persistable-rideshare-visualizations}. 
    On the left, the clustering is obtained using the usual Persistable pipeline, 
    choosing the gap marked in \cref{Persistable-rideshare-visualizations}. 
    On the right, the clustering is obtained using {\sc ExhaustivePF}, rather than $\PF$. 
    We find large clusters in Manhattan, Brooklyn, Queens, and New Jersey, 
    as well as smaller clusters at Newark, LaGuardia, and John F. Kennedy airports.}
    \label{rideshare-data-fig}
\end{figure}

While this data set is easily visualizable, 
it is challenging for many density-based clustering algorithms, 
because its apparent cluster structure takes place at very different levels of density. 
For example, there are approximately 4000 data points clustered near Terminal~B of the airport, 
and meanwhile, there is a cluster of approximately 60 data points 
near the LaGuardia Airport Marriott hotel, 
and an even smaller cluster of 12 data points near the P.S. 127 Aerospace Science Magnet School.

In particular, this data set is challenging for HDBSCAN 
\citep{campello-moulavi-sander}, a popular density-based clustering algorithm 
that is based on the robust single-linkage algorithm of \cref{robust-single-linkage} 
(for a description in these terms, see \citealp{mcinnes-healy}). 
Unlike the related DBSCAN algorithm \citep{dbscan}, 
HDBSCAN can detect multi-scale cluster structure that takes place at a range of distance scales. 
However, like robust single-linkage, it relies on a fixed density threshold, 
and for this reason, cannot find the multi-scale cluster structure in this example 
(i.e., it cannot simultaneously find the large cluster near Terminal~B and the small cluster near the hotel).  
See the jupyter notebook for this data set at 
the Persistable repository \citeyearpar{scoccola-rolle-joss} 
to try clustering the data with HDBSCAN.

On the other hand, Persistable is sensitive to this kind of multi-scale clustering structure, 
because it uses slices of $\DR$ in which both the density threshold and the distance scale vary. 
In \cref{Persistable-laguardia-visualizations}, we show Persistable visualizations 
for the data. 
Guided by these visualizations, one can use Persistable to find 
clusterings of the rideshare data that simultaneously capture the 
large clusters near the airport terminals and the small clusters in the surrounding neighborhood.

Now we consider the complete Rideshare data set. 
The complexity of the data is reflected in the Persistable visualizations; 
see \cref{Persistable-rideshare-visualizations}. 
Using Persistable, one can obtain informative clusterings of this data at coarser or finer scales. 
For example, see \cref{rideshare-data-fig} for coarse but informative results. 
Using smaller gaps in the prominence vineyard, one can obtain finer results, 
with, for example, clusters centered at Penn Station and the Meatpacking district in Manhattan, 
and Williamsburg and Downtown Brooklyn. 
See the jupyter notebook for this data set at 
the Persistable repository \citeyearpar{scoccola-rolle-joss} for details.

We are able to easily compute clusterings of this data set with Persistable 
using the subsampling approximation described in \cref{subsampling-approximation} 
(which is justified by the Gromov--Hausdorff--Prokhorov stability of $\lambda$-linkage). 
Using a subsample of 30\,000 data points, we are able to compute clusterings 
of the complete Rideshare data set in a matter of seconds, using approximately 200 MB of RAM, 
using a laptop with an Intel(R) Core(TM) i5 CPU (4 cores, 1.6GHz) and 8 GB RAM, running GNU/Linux.

Without the subsampling approximation, clustering this data set would be a significant computational challenge. 
For context, we clustered the data using the high-performance implementation of HDBSCAN of 
\cite{mcinnes-healy-astels-joss}. 
This is a natural comparison, because HDBSCAN and Persistable are very similar on an algorithmic level, 
and because the implementation of \cite{mcinnes-healy-astels-joss} is very similar to 
our implementation of Persistable 
(indeed, important components of our implementation come directly from this implementation of HDBSCAN). 
The key performance advantage that Persistable has in this example is the subsampling approximation. 
An analogous approximation scheme is not valid for HDBSCAN, 
as the hierarchical clustering algorithm underlying HDBSCAN (robust single-linkage) 
is not Gromov--Hausdorff--Prokhorov stable (see \cref{instability-related-methods}). 

See \cref{hdbscan-rideshare-performance-table} in \cref{appendix-Persistable} for the results. 
The memory usage of HDBSCAN scales with $n \cdot k$, where $n$ is the number of data points 
and $k$ is the density threshold parameter $\mathsf{min\_samples}$. 
This means that, on the laptop described above, we are only able to run HDBSCAN with very small values of the $\mathsf{min\_samples}$ parameter, producing only very fine clusterings.

\subsubsection{Olive Oil Data}
We consider a data set concerning the fatty acid composition of $572$ samples of olive oil. 
For each sample, the percentages of $8$ fatty acids were measured. 
The samples were taken from nine regions of Italy, 
and these regions are grouped into three larger areas 
(South Italy, Sardinia, and North Italy). 
Each sample is labeled with its region of origin. 

This is a useful test data set for classification and clustering methods, 
and it is used as a benchmark data set for density-based clustering by 
\cite{stuetzle-nugent}. 
The data set is due to \cite{olive-oil-source}, 
and we obtained it from the supplementary materials of \cite{stuetzle-nugent}. 
Using Persistable, one can recover much of the hierarchical clustering structure 
defined by the regions of origin.

We consider the data as points in $\bbR^8$ with the Euclidean metric; 
each feature is independently centered to have mean zero and scaled to unit variance.
One can see many large gaps in the prominence vineyard visualization 
(see \cref{Persistable-oliveoil-visualizations}). 
This is a consequence of the multi-scale clustering structure of the data. 
For example, say we choose the slice indicated in \cref{Persistable-oliveoil-visualizations}. 
If we choose the large gap between the third and fourth vines 
(i.e., we choose three clusters in the persistence-based flattening), 
we get a clustering of the data that fits the large area labels perfectly, 
with $89\%$ of the data clustered 
(see \cref{oliveoil-3-clusters-table} in \cref{appendix-Persistable} for the confusion matrix). 
Meanwhile, if we choose the gap between the eighth and ninth vines 
marked in \cref{Persistable-oliveoil-visualizations}, 
we get a clustering that fits the region labels very accurately: 
the adjusted Rand index is $0.98$, and $49\%$ of the data is clustered 
(see \cref{oliveoil-8-clusters-table} in \cref{appendix-Persistable} for the confusion matrix). 

In order to cluster more data points, 
we also run the exhaustive persistence-based flattening (\cref{tomato-flattening-algorithm}). 
We choose the same slice and gap as before, 
but replace the usual persistence-based flattening with the exhaustive persistence-based flattening. 
The result clusters $95\%$ of the data, 
with an adjusted Rand index of $0.90$ with respect to the region labels 
(see \cref{oliveoil-8-clusters-exhaustivePF-table} in \cref{appendix-Persistable} for the confusion matrix). 

For comparison, \cite{stuetzle-nugent} apply a density-based clustering algorithm to this data, 
and report an adjusted Rand index of $0.61$ with respect to the region labels, 
with all data points clustered.

\begin{figure}[btp]
    \centering
    \includegraphics[width=\twC\textwidth]{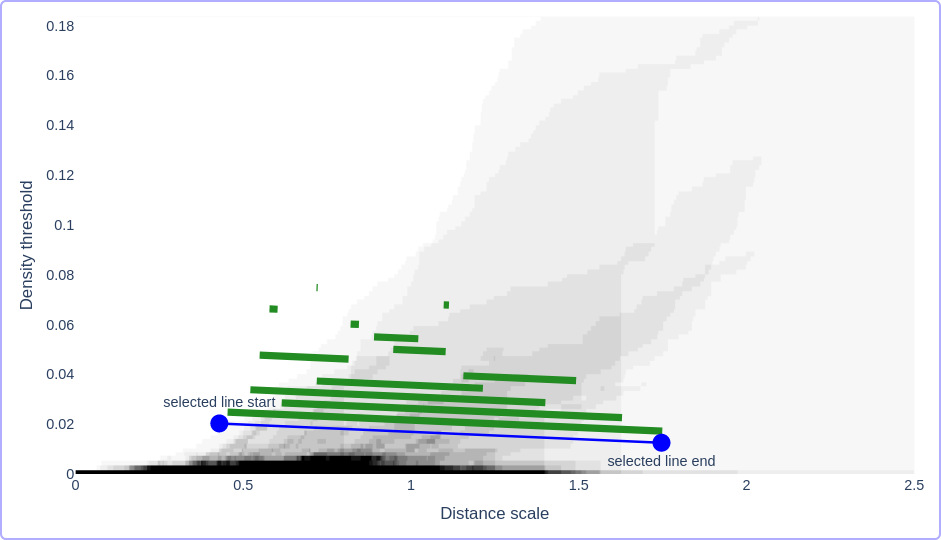}
    \includegraphics[width=\twC\textwidth]{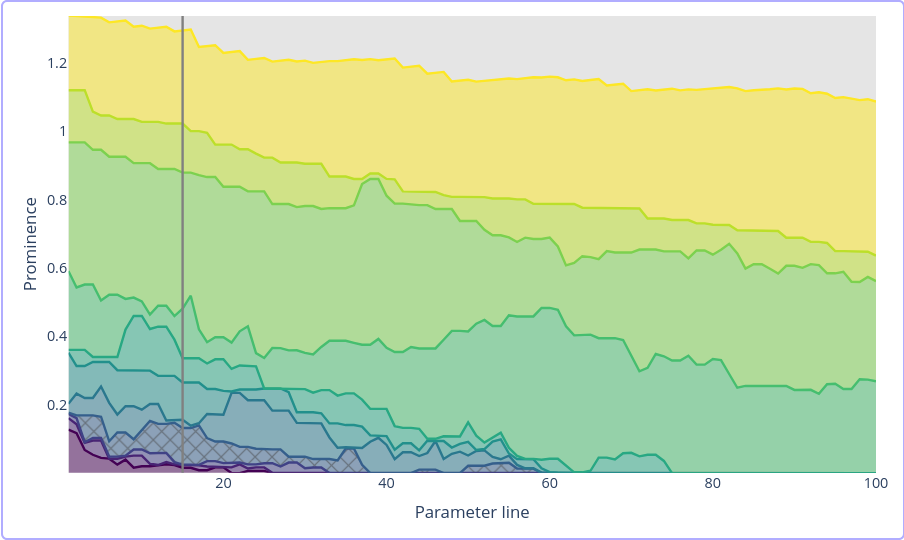}
    \caption{Persistable visualizations for the olive oil data. 
    On the left, the component counting visualization, 
    and on the right, a prominence vineyard visualization. 
    We use the prominence vineyard to choose a slice, 
    which appears as the blue line segment on the component counting visualization; 
    the corresponding slice in the vineyard is marked by a vertical line. 
    The barcode of this slice is displayed in green.
    Several gaps in this vineyard lead to interesting clusterings. 
    If we choose the gap below the eighth vine (marked in this figure), 
    we get a clustering that fits the region labels very accurately. 
    If we choose the gap below the third vine, 
    we get a clustering that fits the large area labels.}
    \label{Persistable-oliveoil-visualizations}
\end{figure}

\section{Conclusions}
\label{conclusions}
We conclude by mentioning some possible directions for future work. 
As we explained in \cref{choice-of-GHP}, \cite{blumberg-lesnick-2param-published} 
prove a stability result for the simplicial degree-Rips bifiltration using the Gromov--Prokhorov distance. 
The authors also provide experimental evaluation of their stability result, 
and suggest the goal of developing a stability theory for degree-Rips that better explains such experimental results 
\citep[Remark A.3]{blumberg-lesnick-2param-published}. 
It may be fruitful to pursue this goal in the setting of the degree-Rips hierarchical clustering.

In this paper we used slices of kernel linkage given by lines $\lambda$ in the parameter space. 
Our stability theorem for kernel linkage implies that appropriately chosen non-linear slices 
are also stable, and our consistency theorem applies also to appropriate non-linear slices. 
An interesting question is whether there are use cases in which non-linear slices lead to more informative clusterings.

For Persistable, we use a two-step process: 
we first reduce from degree-Rips to a one-parameter hierarchical clustering by taking a slice, 
then we reduce to a clustering using the persistence-based flattening. 
We begin by taking a slice of degree-Rips because one-parameter hierarchical clusterings are much simpler 
than multiparameter hierarchical clusterings. 
This distinction between one-parameter hierarchical clusterings and multiparameter hierarchical clusterings is analogous to the distinction between one-parameter persistence modules and multiparameter persistence modules; 
see \cite{bauer-etal-hierarchical-clustering} for a discussion 
of the structural complexity of multiparameter persistence modules 
with a focus on persistence modules arising from hierarchical clusterings. 
For our purposes, it is particularly important that 
the barcode and the persistence-based flattening algorithm are only defined for one-parameter hierarchical clusterings. 
Not much is known about flattening multiparameter hierarchical clusterings directly, 
but see \cite{jardine-stable-components} and \cite{shiebler}. 
An interesting question for future research is how one can stably extract a single clustering from 
degree-Rips or kernel linkage, without taking a one-parameter slice. 


\acks{%
We thank Michael Lesnick for telling us about 
the good properties of one-parameter slices in which all parameters vary, 
and for telling us about his work with Blumberg on the stability of degree-Rips, 
which inspired our work here. 
We thank Leland McInnes for helpful discussion about HDBSCAN.
And, we thank Dan Christensen, Rick Jardine, Michael Kerber, and Matt Piekenbrock
for helpful discussions about this project and related topics.

A.R.~was partially supported by Austrian Science Fund (FWF) 
grant number P 29984-N35, 
and by the German Research Foundation (DFG) Project-ID 195170736 
``TRR109 Discretization in Geometry and Dynamics''. 
L.S.~was partially supported by the National Science Foundation through grant CCF-2006661
and CAREER award DMS-1943758, as well as by EPSRC grant ``New Approaches to Data
Science: Application Driven Topological Data Analysis'', EP/R018472/1.
For the purpose of Open Access, the authors have applied a CC BY public copyright licence to any Author Accepted Manuscript (AAM) version arising from this submission.
}

\appendix

\addtocontents{toc}{\protect\setcounter{tocdepth}{1}}

\section{Missing Details}

\subsection{Details from \cref{Hierarchical-clustering}}
\label{appendix-HC}

%

\begin{proof}
\textbf{(of \cref{dI-and-infty-distance})}
First, say $H$ and $E$ are $\epsilon$-interleaved 
for $\epsilon \geq 0$. 
Let $x,y \in X$; we show 
$|\mergefun{H}(x,y) - \mergefun{E}(x,y)| \leq \epsilon$. 
Without loss of generality, 
$\mergefun{H}(x,y) \leq \mergefun{E}(x,y)$. 
Let $r > \mergefun{H}(x,y)$. 
Then there is $C \in \extension{H}(r)$ with $x,y \in C$. 
By the $\epsilon$-interleaving property, 
there is $D \in \extension{E}(r + \epsilon)$ 
with $C \subseteq D$, 
and thus $\mergefun{E}(x,y) \leq r + \epsilon$.
This shows that 
$|\mergefun{H}(x,y) - \mergefun{E}(x,y)| \leq \epsilon$, 
and it follows that 
$\dI(H, E) \geq d_{\infty}(\mergefun{H}, \mergefun{E})$.

Now, say $\epsilon > d_{\infty}(\mergefun{H}, \mergefun{E})$. 
We show $H$ and $E$ are $\epsilon$-interleaved. 
Let $r \in \bbR$, and let $C \in \extension{H}(r)$. 
Choose $x \in C$. 
As $\mergefun{H}(x,x) \leq r$, 
we have $\mergefun{E}(x,x) < r + \epsilon$, 
so that there is $D \in \extension{E}(r + \epsilon)$ 
with $x \in D$. 
Similarly, for any $y \in C$, 
$\mergefun{H}(x,y) \leq r$, 
so that $\mergefun{E}(x,y) < r + \epsilon$, 
and thus $y \in D$, 
and therefore $C \subseteq D$. 
This shows that $\extension{H}(r) \preceq \extension{E}(r + \epsilon)$. 
A symmetric argument shows that 
$\extension{E}(r) \preceq \extension{H}(r + \epsilon)$, 
and thus $H$ and $E$ are $\epsilon$-interleaved. 
It follows that 
$\dI(H, E) \leq d_{\infty}(\mergefun{H}, \mergefun{E})$, 
and the proposition follows.
\end{proof}

\begin{proof}
\textbf{(of \cref{Hf-stability})}
Let $\epsilon \geq 0$. 
We show that $||f - g||_\infty \leq \epsilon$ 
if and only if $H(f)$ and $H(g)$ are $\epsilon$-interleaved, 
and the proposition follows. 
If $||f - g||_\infty \leq \epsilon$,
then, for every $r \geq \epsilon$, we have $\{f \geq r\}\subseteq \{g \geq r-\epsilon\}$,
and $\{g \geq r\} \subseteq \{f \geq r-\epsilon\}$.
This implies that, after taking connected components,
every connected component of $\{f \geq r\}$ is included in
a connected component of $\{g \geq r-\epsilon\}$,
and that every connected component of $\{g \geq r\}$ is included in
a connected component of $\{f \geq r-\epsilon\}$. 
If $H(f)$ and $H(g)$ are $\epsilon$-interleaved, then, 
for every $x$ in the support of the functions,
if $f(x) \geq \epsilon$, then there exists a cluster
in $\{g \geq f(x) - \epsilon\}$ containing $x$, since $x \in \{f \geq f(x)\}$.
This implies that for any $x$ in the support of the functions we have $g(x) \geq f(x) - \epsilon$.
A symmetric argument shows that
$f(x) \geq g(x) - \epsilon$ for every $x$ in the support, concluding the proof.
\end{proof}

\begin{proof}
\textbf{(of \cref{CI-distance-properties})} 
The only non-trivial case is the triangle inequality, which is proved by
composing correspondences.
If $X,Y,Z$ are sets, and $R \subseteq X \times Y$ 
and $S \subseteq Y \times Z$ are correspondences, 
$S \circ R \subseteq X \times Z$ is the correspondence
$S \circ R = \{ (x, z) : \exists y \in Y \text{ with } (x, y) \in R \text{ and } (y, z) \in S \}.$
If $H$ and $E$ are $\vec{\epsilon}$-interleaved with respect to $R$, 
and $E$ and $F$ are $\vec{\delta}$-interleaved with respect to $S$,
then $H$ and $F$ are 
$(\vec{\epsilon} + \vec{\delta})$-interleaved with respect to $S \circ R$. 
From this it follows that 
$\dWI(H, F) \leq \dWI(H, E) + \dWI(E, F)$.
\end{proof}

\begin{lemma} \label{dis-interleaving-lemma}
Let $H$ and $E$ be ultrametric hierarchical clusterings 
of sets $X$ and $Y$ respectively. 
If $R$ is a correspondence between $X$ and $Y$, then 
the distortion of $R$ is 
$\distortion(R) = 
\inf \{\epsilon \geq 0 : 
H,E \text{ are } \epsilon\text{-interleaved w.r.t. } R\}$.
\end{lemma}

\begin{proof}
First we show that if $H$ and $E$ are 
$\epsilon$-interleaved with respect to $R$, 
then $\distortion(R) \leq \epsilon$. 
Let $(x,y),(x',y') \in R$. 
If $r > \mergefun{H}(x,x')$, 
then there is $C \in \extension{H}(r)$ containing $x,x'$, 
and thus $\pi_X^{-1}(C)$ 
contains $(x,y),(x',y')$. 
By the interleaving property, 
there is a cluster in $\pi_Y^{*}(\extension{E})(r+\epsilon)$ 
containing $(x,y),(x',y')$, 
and thus there is a cluster in $\extension{E}(r+\epsilon)$ 
containing $y,y'$. 
So, $\mergefun{E}(y,y') \leq r + \epsilon$, 
and thus $\mergefun{E}(y,y') 
\leq \mergefun{H}(x,x') + \epsilon$. 
Together with a symmetric argument, we have 
$|\mergefun{H}(x,x')-\mergefun{E}(y,y')| \leq \epsilon$. 
Thus $\distortion(R) \leq \epsilon$. 
Now, we show that $H$ and $E$ are 
$\epsilon$-interleaved with respect to $R$, 
for any $\epsilon > \distortion(R)$, 
which finishes the proof. 
Let $r \in \bbR$ and let $C \in \extension{H}(r)$. 
We need to show there is $D \in \extension{E}(r + \epsilon)$ 
such that $\pi_X^{-1}(C) \subseteq \pi_Y^{-1}(D)$. 
Let $x,x' \in C$ and let $(x,y),(x',y') \in R$. 
We have $r \geq \mergefun{H}(x,x')$, 
therefore $\mergefun{E}(y,y') \leq r + \distortion(R)$, 
and thus there is $D \in \extension{E}(r + \epsilon)$ 
with $y,y' \in D$. 
It follows that $(x,y),(x',y') \in \pi_Y^{-1}(D)$, 
and as $x,x'$ were arbitrary, 
we have $\pi_X^{-1}(C) \subseteq \pi_Y^{-1}(D)$.
\end{proof}

\begin{lemma} \label{formulafordensity}
Let $K$ be a kernel, and let $\MPS$ be a metric probability space.
Let $K^{-1} : \bbRs \to \bbR_{\geq 0}$ be defined as $K^{-1}(t) = \min\{u : K(u) \leq t\}$.
Then $K^{-1}$ is a non-increasing function with compact support,
and we have, for every $x \in \MPS$,
\[
    \conv{\mu_\MPS}{K}{s}(x) = \int_0^{\infty} \mu_\MPS\left(B(x,sK^{-1}(r))\right)\, \dd r.
\]
\end{lemma}

\begin{proof}
Since $K(r) \to 0$ as $r \to \infty$, for every $t > 0$ the set $\{u : K(u) \leq t\}$ is non-empty.
Moreover, $K$ is continuous from the right, so the set has a minimum, and thus
$K^{-1}$ is well-defined.
The fact that $K^{-1}$ is non-increasing is clear, and the fact that it has
compact support follows from the fact that $K$ is bounded.

To prove the statement about $\conv{\mu_\MPS}{K}{s}$, we need the following straightforward
fact about $K^{-1}$: for every $s,t \in \bbR_{\geq 0}$ we have $K^{-1}(t) > s$ if and only if $t < K(s)$.
We finish the proof by computing
\begin{align*}
    \int_{x' \in \MPS} K\left(\frac{d(x,x')}{s}\right) \dd\mu_\MPS &=
    \int_{x' \in \MPS} \int_0^\infty \mathbf{1}_{\left\{r < K\left(\frac{d(x,x')}{s}\right)\right\}} \dd r\,\dd\mu_\MPS \\
    = \int_{x' \in \MPS} \int_0^\infty \mathbf{1}_{\left\{d(x,x') < s K^{-1}(r)\right\}} \dd r\,\dd\mu_\MPS
    &= \int_0^\infty \int_{x' \in \MPS} \mathbf{1}_{\left\{d(x,x') < s K^{-1}(r)\right\}} \dd\mu_\MPS\,\dd r\\
    &= \int_0^\infty \mu_\MPS\left(B(x,s K^{-1}(r)\right) \dd r,
\end{align*}
as required.
\end{proof}

\subsection{Details from \cref{Stability}} \label{appendix-stability}

\begin{lemma} \label{localdensitybounds}
Let $K$ be a kernel and let $r' \in (0,K(0))$.
Let $Z$ be a compact metric space and let $\mu$ and $\nu$ be Borel probability measures on $Z$
such that $\dP(\mu,\nu) < \epsilon$ for $\epsilon > 0$.
Let $x,y \in Z$ such that $d_Z(x,y) < \epsilon'$.
Then, for all $s > 0$, we have
    $\conv{\mu}{K}{s}(x) \leq \conv{\nu}{K}{s+\epsilon_s}(y) + \epsilon_k$,
for $\epsilon_s = \frac{\epsilon+\epsilon'}{K^{-1}(r')}$
and $\epsilon_k = K(0) \left(\frac{K(0)}{r'} - 1\right) + K(0) \epsilon$.
\end{lemma}

\begin{proof}
Using \cref{formulafordensity}, we know that
$\conv{\mu}{K}{s}(x) = \int_0^{K(0)} \mu(B(x,s K^{-1}(r)))\, \dd r$,
since, if $r > K(0)$, then $K^{-1}(r) = 0$.
Note that, for any radius $R \geq 0$, we have
\[
	\mu \left( B(x,R) \right) 
    \leq  \nu \left( B(x,R)^{\epsilon} \right) + \epsilon
    \leq  \nu \left( B(x,R + \epsilon) \right) + \epsilon
	\leq  \nu \left( B(y,R + \epsilon + \epsilon') \right) + \epsilon,
\]
so we can bound the local density estimate of $x$ as follows:
{\small
\[
    \conv{\mu}{K}{s}(x) \leq \int_0^{K(0)} \nu(B(y,s K^{-1}(r) + \epsilon + \epsilon')) + \epsilon \,\dd r = \int_0^{K(0)} \nu(B(y,s K^{-1}(r) + \epsilon + \epsilon')) \, \dd r + K(0) \epsilon.
\]
}
Since $K^{-1}$ is non-increasing, and $r' < K(0)$,
it follows that
$K^{-1}(r r'/K(0)) \geq K^{-1}(r)$ for every $r \geq 0$.
Moreover, for any $0 \leq r \leq K(0)$, we have $K^{-1}(r r'/K(0)) \geq K^{-1}(r')$.
These two considerations imply that, for $0 \leq r \leq K(0)$, we have
\[
    sK^{-1}(r) + \epsilon + \epsilon' \leq \left(s + (\epsilon + \epsilon')/K^{-1}(r')\right) K^{-1}\left(r r'/K(0)\right).
\]
Combining this with the above bound for the local density estimate of $x$ we get
\begin{align*}
    \conv{\mu}{K}{s}(x) &\leq \int_0^{K(0)} \nu\left(B\left(y,\left(s + (\epsilon + \epsilon')/K^{-1}(r')\right) K^{-1}\left(r r'/K(0)\right)\right)\right) \, \dd r + K(0) \epsilon\\
                  &= \frac{K(0)}{r'} \int_0^{r'} \nu\left(B\left(y,\left(s + (\epsilon + \epsilon')/K^{-1}(r')\right) K^{-1}\left(r \right)\right)\right) \, \dd r + K(0) \epsilon\\
                  &\leq \frac{K(0)}{r'}\,\, \conv{\nu}{K}{\left(s + (\epsilon + \epsilon')/K^{-1}(r')\right)}(y) + K(0) \epsilon.
\end{align*}

Finally, note that, for $0\leq a \leq M < \infty$ and $c \geq 1$, we have $ca \leq a + M(c-1)$.
As $\nu$ is a probability measure, any local density estimate is bounded by $K(0)$.
This implies that
\[
    \conv{\mu}{K}{s}(x) \leq \conv{\nu}{K}{\left(s + (\epsilon + \epsilon')/K^{-1}(r')\right)}(y) + K(0) \left(\frac{K(0)}{r'} - 1\right) + K(0) \epsilon,
\]
as required.
\end{proof}


\begin{proof}
\textbf{(of \cref{RSL-is-GHP-discontinuous})}
We will construct a finite metric space $\Met$ 
such that $\RL_{\kappa, \alpha}$ is not continuous at $\Met$. 
For any $\delta > 0$, 
we will show that there is a finite metric space $\Metalt$ 
with $\dGHP(\Met,\Metalt) < \delta$ such that 
$\dWI(\RL_{\kappa, \alpha}(\Met), \RL_{\kappa, \alpha}(\Metalt)) > 1/2$. 

Let $\Met \subset \bbR$ be a subset with $0, 1 \in \Met$, 
and with $\kappa - 2$ points in the interval $(-\frac{1}{10}, 0)$ 
and $\kappa - 2$ points in the interval $(1, \frac{11}{10})$. 
For $\ell \geq 1$, let $\Met_\ell = \Met \cup \{x + \frac{1}{\ell} : x \in \Met\}$. 
We have $\dGHP(\Met, \Met_\ell) \leq \frac{1}{\ell}$ for all $\ell \geq 1$. 
It remains to show that 
$\dWI(\RL_{\kappa, \alpha}(\Met), \RL_{\kappa, \alpha}(\Met_\ell)) > 1/2$ 
for all sufficiently large $\ell$. 
Note that $\RL_{\kappa, \alpha}(\Met)(r) = \emptyset$ for $r < 1$; 
however, for $\ell \geq 10$, 
$0$ is in a cluster of $\RL_{\kappa, \alpha}(\Met_\ell)(r)$ 
for any $r > 1/10$. This finishes the proof.
\end{proof}

As we described in \cref{Stability}, 
one could also formalize robust single-linkage by 
taking the density threshold parameter to be a ratio $k \in (0,1)$, 
and then letting $\RL_{k, \alpha}(\Met) = \mml(\Met)^{\gamma}$ 
for the covariant curve 
$\gamma \colon (0, \infty) \to \bbRs^{\times 3}$ with 
$\gamma(r) = (r, \alpha r, k)$. 
This variant also fails to be continuous 
with respect to the Gromov--Hausdorff--Prokhorov distance:

\begin{proposition} \label{RSL-k-is-GHP-discontinuous}
Let $k \in (0,1)$ be rational, and let $\alpha > 0$. 
With respect to the Gromov--Hausdorff--Prokhorov distance 
and the correspondence-interleaving distance, 
$\RL_{k, \alpha}$ is discontinuous.
\end{proposition}

\begin{proof}
Write $k = p/q$. 
Without loss of generality, we may assume $p \geq 2$. 
We will construct a finite metric space $\Met$ such that 
$\RL_{k, \alpha}$ is not continuous at $\Met$. 
Let $\Met \subset \bbR$ be a subset with $0 \in \Met$, 
with $|\Met \cap (-\frac{1}{10}, 0)| = p-1$ 
and with $1,2, \dots, q-p \in \Met$. 
For $n \geq 1$, let
\[
	\Met_n = \left( \Met \cup \Met + \frac{1}{n^2} \cup \dots 
	\cup \Met + \frac{n-1}{n^2} \right) \setminus \left\{\frac{n-1}{n^2}\right\} \, ,
\]
where $\Met + a = \{x + a : x \in \Met\}$. 
The idea is that we replace each point of $\Met$ 
with $n$ points that are tightly grouped together, 
except $0$, which we replace with only $n-1$ points 
(hence we remove the point $\frac{n-1}{n^2}$). 
We have $\dGHP(\Met, \Met_n) \to 0$ as $n \to \infty$. 
We have $0$ in a cluster of $\RL_{k, \alpha}(\Met)(r)$ 
for any $r > 1/10$; 
however, for $n$ sufficiently large, 
$\RL_{k, \alpha}(\Met_n)(r) = \emptyset$ 
for any $r \leq 1/2$. 
This shows that 
$\dWI(\RL_{k, \alpha}(\Met), \RL_{k, \alpha}(\Met_n)) > 4/10$ 
for sufficiently large $n$, 
finishing the proof.
\end{proof}

\begin{proof}
\textbf{(of \cref{RSL-is-not-Lipschitz-in-kappa})}
For simplicity, we assume $\kappa' = \kappa + 1$, 
but the construction can easily be extended to the general case.  
Let $\Met \subset \bbR$ consist of $\kappa$ points in the interval $(-1,0)$, 
as well as the point $D + 2$.  
Let $x \in \Met \cap (-1, 0)$. 
Then $x$ is in a cluster of $\RL_{\kappa, \alpha}(\Met)(r)$ for any $r > 1$, 
but $\RL_{\kappa', \alpha}(\Met)(r) = \emptyset$ for all $r \leq D + 2$.
\end{proof}

We also have the analogue of \cref{RSL-is-not-Lipschitz-in-kappa} 
for the variant of robust single-linkage 
that takes a density threshold $k \in (0,1)$ instead of $\kappa$:

\begin{proposition} \label{RSL-is-not-Lipschitz-in-k}
Let $k, k' \in (0,1)$ be rational with $k \neq k'$, 
and let $\alpha > 0$. 
For any $D > 0$, there is a finite metric space $\Met$ such that 
$\dWI(\RL_{k, \alpha}(\Met), \RL_{k', \alpha}(\Met)) > D$. 
\end{proposition}

\begin{proof}
The construction of $\Met$ is similar to the proof of \cref{RSL-k-is-GHP-discontinuous}. 
Write $k = p/q$. 
Without loss of generality, we may assume $p \geq 2$, 
and $k < k'$. 
Let $\Met \subset \bbR$ be a subset with $0 \in \Met$, 
with $|\Met \cap (-\frac{1}{2}, 0)| = p-1$, 
and with $D+1, 2(D+1), \dots, q-p(D+1) \in \Met$. 
If $s \geq 1/2$, then $0$ is in a cluster of $\RL_{k, \alpha}(\Met)$. 
However, if $s < D + 1$, then $\RL_{k', \alpha}(\Met) = \emptyset$, 
and the proposition follows.
\end{proof}

\begin{proof}
\textbf{(of \cref{RG-is-GHP-discontinuous})}
Let $\Met = \{0, t\} \subset \bbR$. 
Say $\RG$ is defined using the kernel $K$. 
Because we always use isotropic kernels, 
we have $\conv{\mu_\Met}{K}{s}(0) = \conv{\mu_\Met}{K}{s}(t) =: M$. 
So, $\RG_{s,t}(\Met)(r) = \emptyset$ if $r > M$ 
and $\RG_{s,t}(\Met)(r) = \{\Met\}$ if $r \leq M$. 
Now, for any $\epsilon > 0$, let $\Met_{\epsilon} = \{0, t+\epsilon\}$. 
For any correspondence $R$ between $\Met$ and $\Met_{\epsilon}$, 
we have $\pi_\Met^*(\RG_{s,t}(\Met))(r) = \{R\}$ for $r \leq M$, 
but $\pi_{\Met_{\epsilon}}^*(\RG_{s,t}(\Met_{\epsilon}))(r) \neq \{R\}$ 
for any $r > 0$. So, for any $0 \leq \delta < M$, 
$\RG_{s,t}(\Met)$ and $\RG_{s,t}(\Met_{\epsilon})$ are not 
$\delta$-interleaved with respect to $R$, 
and thus $\dWI(\RG_{s,t}(\Met), \RG_{s,t}(\Met_{\epsilon})) > \delta$. 
But, as $\epsilon \to 0$, we have $\dGHP(\Met, \Met_{\epsilon}) \to 0$.
\end{proof}

\begin{proof}
\textbf{(of \cref{RG-is-discontinuous-in-t})} 
Say $\RG$ is defined using the kernel $K$. 
Let $s > 0$ be arbitrary, let 
$t = \min_{x \neq x'} d_\Met(x,x')$, 
and let $x_0, x_1 \in \Met$ be such that $d_\Met(x_0, x_1) = t$. 
If $t' < t$, then $\RG_{s,t'}(\Met)(r)$ consists of singletons for all $r > 0$, 
while $\RG_{s,t}(\Met)(r)$ has a cluster containing $x_0$ and $x_1$ 
for all $r \leq \min(\conv{\mu_\Met}{K}{s}(x_0), \conv{\mu_\Met}{K}{s}(x_1))$. 
So, $\dWI(\RG_{s,t}(\Met), \RG_{s,t'}(\Met))$ does not go to zero 
as $t' \to t$ from below.
\end{proof}

%
%

\subsection{Details from \cref{Consistency}}
\label{appendix}

In order to prove that CI-consistency implies Hartigan consistency, we need a lemma, 
which is similar to \citet[Lemma 14, Appendix: Consistency]{chaudhuri-dasgupta-10}, 
except that we do not require super-level sets to have finitely many connected components.

\begin{lemma}\label{tamenessofHf}
Let $f : \bbR^d \to \bbR$ be a continuous, compactly supported probability density function, 
let $r > 0$ and $A \neq A' \in H(f)(r)$.
There exists $\epsilon > 0$ and $B, B' \in H(f)(r - \epsilon)$ with $B \neq B'$ such that $A \subseteq B$ and $A' \subseteq B'$.
\end{lemma}

\begin{proof}
Say, towards a contradiction, that for all $n$ with $1/n < r$, there is $B_n \in H(f)(r - 1/n)$ 
with $A \subset B_n$ and $A' \subset B_n$. 
It is a standard fact that if $K_1 \supseteq K_2 \supseteq \cdots$ 
is a nested sequence of non-empty, compact, connected sets in Euclidean space, 
then the intersection $\cap_{i=1}^{\infty} K_i$ is connected \cite[10.1.23]{csaszar}. 
So, the intersection $B = \cap B_n$ is connected. 
For all $b \in B$, we have $f(b) \geq r - 1/n$ for all $n$ large enough, 
so we must have $f(b) \geq r$. 
As $B$ is contained in $\{ f \geq r \}$ and $B$ is connected, 
$B$ must intersect only one connected component of $\{ f \geq r \}$, 
but we have $A \subset B$ and $A' \subset B$, a contradiction.
\end{proof}

\begin{proof}
\textbf{(of \cref{CI-implies-Hartigan})}
    Given $r > 0$ and distinct elements $A$ and $A'$ of $H(f)(r)$,     
    we show that the probability of $A_n \cap A'_n = \emptyset$
    goes to $1$ as $n \to \infty$,
    where $A_n$ is the smallest cluster in $\mathbb{A}^{\theta_n}(X_n)$ that contains
    $A \cap X_n$ and likewise for $A'$, 
    and the $\theta_n$ are the parameters whose existence is given by CI-consistency of $\mathbb{A}$.

    From \cref{tamenessofHf} it follows that
    there exists $\epsilon > 0$ and distinct elements $B,B' \in H(f)(r-\epsilon)$
    such that $A \subseteq B$ and $A' \subseteq B'$.    
    Let $\delta \in (0,1)$. By assumption, there exists $N$ such that, if $n\geq N$,
    then the probability that $\mathbb{A}^{\theta_n}(X_n)$ and $H(f)$
    are $\epsilon/2$-interleaved with respect to the closest point correspondence $R_c \subseteq X_n \times \support(f)$
    is greater than $1-\delta$.
    As $R_c$ contains the pairs $(x,x)$ for $x \in X_n$, 
	if $\mathbb{A}^{\theta_n}(X_n)$ and $H(f)$
    are $\epsilon/2$-interleaved with respect to $R_c$, then 
    $\mathbb{A}^{\theta_n}(X_n)$ and $i^*(H(f))$ are $\epsilon/2$-interleaved as hierarchical clusterings of $X_n$, 
    where $i : X_n \to \support(f)$ is the inclusion. 
    It is therefore enough to show that if 
    $\mathbb{A}^{\theta_n}(X_n)$ and $i^*(H(f))$ are $\epsilon/2$-interleaved, then $A_n \cap A'_n =\emptyset$. 
    Now, if $\mathbb{A}^{\theta_n}(X_n)$ and $i^*(H(f)))$ are $\epsilon/2$-interleaved,
    then there exist $C,C' \in \mathbb{A}^{\theta_n}(X_n)(r-\epsilon/2)$
    such that $A\cap X_n \subseteq C \subseteq B$ and $A'\cap X_n \subseteq C' \subseteq B'$.
    As $A_n \subseteq C$ and $A'_n \subseteq C'$, and $B \cap B' = \emptyset$, 
    we have $A_n \cap A'_n = \emptyset$.
\end{proof}

We now prove \cref{consistency-lambda-linkage}, 
the CI-consistency of $\linkage{\overline{\lambda}}$. 
Because the argument works in greater generality, 
we actually prove this for any ``admissible family'' of curves 
(this is \cref{consistency-gamma-linkage}). 
The curves $\{\lambda\con^{x,y}\}_{x,y > 0}$ from \cref{consistency-lambda-linkage} 
will be an example of an admissible family. 
We will use the following curves in order to define slices of kernel linkage.

\begin{definition}
A \define{slicing curve} consists of an interval $I_{\gamma} = (0, \max_{\gamma})$ 
with $\max_{\gamma} \in (0, \infty]$, 
and an order-preserving function 
$\gamma = (\gamma_s,\gamma_t,\gamma_k) : I_{\gamma}^{\op} \to \bbRs \times \bbRs \times \bbRs^{\op}$, 
which we assume is continuous when viewed as a map between subspaces of Euclidean space. 
For any slicing curve $\gamma$ and any metric probability space $\MPS$, 
we write $\linkage{\gamma}(\MPS) = \mml(\MPS)^{\gamma}$.
\end{definition}

As in \cref{phi-and-rescaling-lambda}, we have to re-parameterize 
slices of kernel linkage in order to interleave them with the density-contour hierarchical clustering. 
We do this as follows.


\begin{definition} \label{phi-and-rescaling-gamma}
Let $K$ be a kernel, 
and let $\gamma$ be a slicing curve. 
For $s > 0$, we write $v_s = \int_{\bbR^d} K(||x||/s) \dd x$. 
Define an order-preserving function $\varphi : I_\gamma \to \bbRs$ by 
$\varphi(r) = \gamma_k(r) / v_{\gamma_s(r)}$. 
We say that $\gamma$ is \define{covering} if
$\gamma_s$ and $\gamma_k$ are injective, $\gamma_s(r) \to 0$ as $r \to \max_\gamma$, 
and $\gamma_k(r) \to 0$ as $r \to 0$. 
If $\gamma$ is covering, then $\varphi$ is a bijection. 
In that case, we write 
$\overline{\gamma} = \gamma \circ \varphi^{-1} : \bbRs^{\op} \to \bbRs \times \bbRs \times \bbRs^{\op}$. 
If $\gamma$ is a slicing curve that is covering, 
then $\overline{\gamma}$ is also a slicing curve. 
\end{definition}

From now on, fix a continuous, compactly-supported density function $f : \bbR^d \to \bbR$ 
with support $\support(f)$. 

\begin{notation}
For $s > 0$, define $f_s : \support(f) \to \bbR$ by
\[
    f_s(x) = \int_{\bbR^d} K\left(\frac{||x-y||}{s}\right) f(y) \; \dd y.
\]
\end{notation}

Note that $f_s(x) = \conv{f}{K}{s}(x) = \conv{\mu_f}{K}{s}(x)$.
As $\support(f)$ is compact, the continuous function $f$ is uniformly continuous.
An elementary consequence (see e.g. \citealp[Theorem~8.14]{folland}) 
is that $f_s / v_s$ approximates $f$ for small enough $s$:

\begin{lemma} \label{uniform-convergence-to-f}
    Given $\epsilon > 0$ there exists $\delta > 0$
    such that if $s < \delta$
    then $||f_s/ v_s - f||_\infty < \epsilon$.
\end{lemma}

\begin{notation}
Let $I \subseteq \bbRs^{\op}$ be an interval, 
and let $H : I \to \C(X)$ be a contravariant hierarchical clustering of a set $X$. 
We write $\height_H : X \to [0,\infty]$ for the function defined by 
$\height_H(x) = \sup\{r \in I \; : \; \exists C \in H(r), x \in C\}$.
\end{notation}

\begin{notation}
We write $\mml(f)$ for the kernel linkage of 
the metric probability space $(\support(f), \mu_f)$. 
For any slicing curve $\gamma$, 
we write $h^{\gamma} = \height_{\mml(f)^{\gamma}}$. 
If $\gamma$ is covering, 
we write $h^{\overline{\gamma}} = \height_{\mml(f)^{\overline{\gamma}}}$.
\end{notation}

\begin{lemma} \label{consistent-curve-consistent-height}
    Let $\gamma$ be a slicing curve that is covering.
    Then, for any $x \in \support(f)$, the quantity
    $h^\gamma(x)$ satisfies $f_{\gamma_s(h^\gamma(x))}(x) = \gamma_k(h^\gamma(x))$.
    And, there exists $r_1 \in I_\gamma$ such that $h^\gamma(x) < r_1$ for every $x \in \support(f)$. 
\end{lemma}

\begin{proof}
	Note first that, for any $x \in \support(f)$, 
	$h^{\gamma}(x) = \sup \, \{ r \in I_{\gamma} : f_{\gamma_s(r)}(x) \geq \gamma_k (r) \}$. 
    We begin the proof by showing the following: 
    there is $r_1 \in I_\gamma$ such that 
    $\{x \in \support(f) : f_{\gamma_s(r_1)}(x) \geq \gamma_k(r_1)\} = \emptyset$,
    and there is $r_0 \in I_\gamma$ such that 
    $\{x \in \support(f) : f_{\gamma_s(r_0)}(x) \geq \gamma_k(r_0)\} = \support(f)$.
    For the existence of $r_1$, note that $f_{\gamma_s(r)}(x) \leq v_{\gamma_s(r)} \cdot \max(f)$ 
	for all $x \in \support(f)$ and $r \in I_{\gamma}$. 
    So $r_1$ exists, since, as $r \to \max_\gamma$, we have 
    $v_{\gamma_s(r)} \cdot \max(f) \to 0$ while $\gamma_k(r)$ is increasing.
    For the existence of $r_0$ note that, for every $r \in I_\gamma$, the function
    $f_{\gamma_s(r)}(x)$ is continuous in $x$ and strictly positive for every $x \in \support(f)$, 
    so we have $\min(f_{\gamma_s(r)}(x)) > 0$ for any $r \in I_\gamma$. 
    So $r_0$ exists since, as $r \to 0$, we have 
    $\gamma_k(r) \to 0$ while $\min(f_{\gamma_s(r)}(x))$ is increasing.

    Now, the function $f_{\gamma_s(r)}(x)$ is decreasing and continuous in $r$,
    and $\gamma_k(r)$ is continuous and strictly increasing in $r$. Since
    $f_{\gamma_s(r_0)}(x) \geq \gamma_k(r_0)$ and $f_{\gamma_s(r_1)}(x) < \gamma_k(r_1)$,
    we have that $h^\gamma(x)$ is the unique number $r \in [r_0,r_1]$
    such that $f_{\gamma_s(r)}(x) = \gamma_k(r)$, as required.
\end{proof}

\begin{lemma} \label{height-approximates-density}
Let $\epsilon > 0$, and let $\gamma$ be a slicing curve that is covering.
There is $\delta > 0$
such that, if $\gamma_s(r) < \delta$ for every $r \in I_\gamma$, 
then $||h^{\overline{\gamma}} - f||_\infty < \epsilon$.
\end{lemma}

\begin{proof}
    Using \cref{uniform-convergence-to-f},
    let $\delta$ be such that if $s < \delta$, then, for all $x \in \support(f)$, we have
    $|f_s(x)/v_s - f(x)| < \epsilon$.
    By definition of $\overline{\gamma}$,
    we have $h^{\overline{\gamma}}(x) = \varphi(h^{\gamma}(x))$.
    Using \cref{consistent-curve-consistent-height}, this implies that, 
    for all $x \in \support(f)$,
    \[
        h^{\overline{\gamma}}(x) = \varphi(h^{\gamma}(x)) =
        \frac{\gamma_k(h^\gamma(x))}{v_{\gamma_s(h^\gamma(x))}} = 
        \frac{f_{\gamma_s(h^\gamma(x))}(x)}{v_{\gamma_s(h^\gamma(x))}},
    \]
    So, if $\gamma_s(r) < \delta$ for every $r \in I_\gamma$, 
    then $|h^{\overline{\gamma}}(x) - f(x)| < \epsilon$
    as $\gamma_s(h^\gamma(x)) < \delta$.
\end{proof}

Now, let $T$ be a topological space, 
and let $\mathcal{U} = \{U_i\}_{i=1}^n$ be an open cover of $T$, 
with $U_i \neq \emptyset$ for all $i$. 
Consider the graph $G_{\mathcal{U}}$ 
with vertex set $\{1, \dots, n\}$, and with an edge $(i, j)$ if 
$U_i \cap U_j \neq \emptyset$. 

\begin{lemma} \label{cech-is-connected}
If $T$ is a connected topological space, 
and $\mathcal{U} = \{U_i\}_{i=1}^n$ is a finite open cover of $T$ 
with $U_i \neq \emptyset$ for all $i$, 
then the graph $G_{\mathcal{U}}$ is connected.
\end{lemma}

\begin{proof}
We use induction on $n$. We assume the statement for $n-1$, and prove it for $n$. 
If $U_1 \cap U_i = \emptyset$ for all $1 < i \leq n$, 
then we can write $T = U_1 \sqcup (\cup_{i=2}^n U_i)$, 
contradicting the assumption that $T$ is connected. 
So, we can choose $1 < j \leq n$ such that $U_1 \cap U_j \neq \emptyset$. 
Let $\mathcal{U}' = \{U_1 \cup U_j, U_2, \dots, \hat{U_j}, \dots, U_n\}$, 
where $\hat{U_j}$ indicates that we remove $U_j$. 
Then $\mathcal{U}'$ is an open cover of $T$ with $n-1$ elements, 
so by induction, $G_{\mathcal{U}'}$ is connected. 
Now, $G_{\mathcal{U}'}$ is obtained by contracting the edge 
$\{1,j\}$ of $G_{\mathcal{U}}$. 
Thus, $G_{\mathcal{U}}$ is connected.
\end{proof}

\begin{lemma} \label{DR-f-approximates-H-f}
Let $\epsilon > 0$ and let $\gamma$ be a slicing curve that is covering.
There is $\delta > 0$
such that, if $\gamma_s(r),\gamma_t(r) < \delta$ for every $r \in I_\gamma$,
then $\mml(f)^{\overline{\gamma}}$ and $H(f)$ are $\epsilon$-interleaved.
\end{lemma}

\begin{proof} 
Using the fact that $f$ is uniformly continuous, 
\cref{height-approximates-density}, and \cref{uniform-convergence-to-f}, 
choose $\delta > 0$ such that, 
for all $x,y \in \bbR^d$, if $||x-y|| < \delta$, then $|f(x) - f(y)| < \epsilon / 2$,
and such that, if $\gamma_s(r) < \delta$ for all $r \in I_{\gamma}$, 
then $||h^{\overline{\gamma}} - f||_\infty < \epsilon / 2$, 
and such that, if $s < \delta$, then $||f_s / v_s - f||_\infty < \epsilon / 2$. 
Let $\gamma$ be a slicing curve that is covering, 
and such that $\gamma_s(r),\gamma_t(r) < \delta$ for every $r \in I_\gamma$. 
We show that we have 
$\mml(f)^{\overline{\gamma}}(r) \preceq H(f)(r - \epsilon)$ and 
$H(f)(r) \preceq \mml(f)^{\overline{\gamma}}(r - \epsilon)$ 
in $\C(\support(f))$ for all $r > \epsilon$. 

By \cref{height-approximates-density}, for any $x \in \support(f)$, 
if $x$ is contained in a cluster of $\mml(f)^{\overline{\gamma}}(r)$, then 
$x$ is contained in a cluster of $H(f)(r - \epsilon)$;  
and if $x$ is contained in a cluster of $H(f)(r)$, then 
$x$ is contained in a cluster of 
$\mml(f)^{\overline{\gamma}}(r - \epsilon)$. 
Next, say $x,y \in C \in \mml(f)^{\overline{\gamma}}(r)$ for $r > \epsilon$. 
We show that $x$ and $y$ belong to the same cluster of 
$H(f)(r - \epsilon)$.
Let $r_0 = \varphi^{-1}(r)$, $s_0 = \gamma_s(r_0)$, $t_0 = \gamma_t(r_0)$, and $k_0 = \gamma_k(r_0)$. 
So, by the definition of $\varphi$, we have
$r = k_0 / v_{s_0}$.
Note we have $t_0,s_0 < \delta$.
As $x,y \in C$, 
there is a chain $x_0, \dots, x_n \in \support(f)$ 
with $x = x_0, y = x_n$, such that 
$f_{s_0}(x_i) \geq k_0$ and 
$||x_i - x_{i+1}|| \leq t_0$ for all $i$. 
Dividing by $v_{s_0}$, we have 
$f_{s_0}(x_i) / v_{s_0} \geq k_0 / v_{s_0} = r$. 
Let $0 \leq i \leq n-1$, and let 
$\alpha_i : [0,1] \to \bbR^d$ 
parameterize the straight-line path from $x_i$ to $x_{i+1}$. 
Let $q \in [0,1]$. Because 
$||x_i - \alpha_i(q)|| \leq t_0 < \delta$, we have 
$|f(x_i) - f(\alpha_i(q))| < \epsilon / 2$. 
As $s_0 < \delta$, we have $| f_{s_0}(x_i)/v_{s_0} - f(x_i) | < \epsilon / 2$. 
So, we have $f(\alpha_i(q)) > r - \epsilon$. 
The concatenation of the $\alpha_i$ is therefore a path 
in $\support(f)$ from $x$ to $y$ such that $f(p) > r - \epsilon$ 
for all points $p$ on the path.
So, $x$ and $y$ belong to the same cluster of $H(f)(r - \epsilon)$. 

Finally, let $x,y \in C \in H(f)(r)$. 
We show that $x$ and $y$ belong to the same cluster of 
$\mml(f)^{\overline{\gamma}}(r - \epsilon)$ for $r > \epsilon$.
Write $t_{\epsilon} = \gamma_t(\varphi^{-1}(r-\epsilon)) > 0$; 
we will show that there is a $t_{\epsilon}$-chain $(x = x_0, \dots, x_n = y) \in C$. 
Let $\{P_i\}_{i \in I}$ be the set of path components of $C$. 
For each $i \in I$, let 
$P_i^{t_{\epsilon}} = \cup_{a \in P_i} B_C(a, t_{\epsilon})$.
Then, $\{ P_i^{t_{\epsilon}} \}_{i \in I}$ is an open cover of $C$. 
Since $C$ is compact, 
there is a finite $J \subseteq I$ such that 
$\mathcal{U} = \{ P_i^{t_{\epsilon}} \}_{i \in J}$ is an open cover of $C$. 

Now, say $i,j \in J$, and $P_i^{t_{\epsilon}} \cap P_j^{t_{\epsilon}} \neq \emptyset$. 
We show that for any $a \in P_i$ and any $b \in P_j$, 
there is a $t_{\epsilon}$-chain in $C$ connecting $a$ and $b$. 
Choose $w \in P_i^{t_{\epsilon}} \cap P_j^{t_{\epsilon}}$; 
by definition, there is $w_i \in P_i$ and $w_j \in P_j$ such that 
$||w - w_i|| < t_{\epsilon}$, and $||w - w_j|| < t_{\epsilon}$. 
Then, there is a $t_{\epsilon}$-chain in $P_i$ connecting $a$ to $w_i$, 
and a $t_{\epsilon}$-chain in $P_j$ connecting $b$ to $w_j$, 
which together give a $t_{\epsilon}$-chain in $C$ connecting $a$ and $b$. 
By \cref{cech-is-connected}, the graph $G_{\mathcal{U}}$ is connected. 
So, there is a $t_{\epsilon}$-chain in $C$ connecting $x$ and $y$.
\end{proof}

\begin{definition}
A slicing curve $\gamma = (\gamma_s,\gamma_t,\gamma_k) : I_\gamma \to \bbRs^{3}$ 
is \define{non-singular in each component} 
if it is continuously differentiable and 
the derivatives $\gamma_s', \gamma_t'$, and $\gamma_k'$ never vanish.
\end{definition}

\begin{definition}\label{admissible-family-def}
We say that a family $\{ \gamma^{\theta} \}_{\theta \in \Theta}$ 
of slicing curves is an \define{admissible family} 
if each $\gamma^{\theta}$ is covering and non-singular in each component, 
and if, for every $b > 0$, there is $\theta \in \Theta$ 
such that for all $r \in I_{\gamma^{\theta}}$, 
we have $\gamma^{\theta}_{s} (r), \gamma^{\theta}_{t} (r) < b$.
\end{definition}

\begin{lemma} \label{continuity-wrt-input}
Let $\gamma$ be a slicing curve that is 
non-singular in each component. 
Let $H$ and $E$ be $\bbRs \times \bbRs \times \bbRs^{\op}$-hierarchical clusterings of 
sets $X$ and $Y$ respectively, 
and let $R \subseteq X \times Y$ be a correspondence. 
Assume there exists $r \in I_\gamma$ with $H^\gamma(r) = \emptyset$. 
For every $\epsilon > 0$ there is $\delta > 0$ such that, 
if $H$ and $E$ are $(\delta,\delta,\delta)$-interleaved with respect to $R$, 
then $H^\gamma$ and $E^\gamma$ are $\epsilon$-interleaved with respect to $R$.
\end{lemma}

\begin{proof}
    Choose $r_0 \in I_\gamma$ such that
    $H^\gamma(r_0) = \emptyset$ and let $\epsilon > 0$. 
    Without loss of generality, suppose that $\epsilon < r_0$.
    Choose $r_1,r_2 \in I_\gamma$ such that $r_0 < r_1 < r_2 < \max_\gamma$. 
    Let $c = \min_{r \in [\epsilon/3,r_2]} 
    \left\{\min\left(|\gamma'_s(r)|, |\gamma'_t(r)|, \gamma'_k(r)\right)\right\} > 0$. 
	Let $\delta = \min(c(r_1-r_0), c\epsilon/3)$. 
    We show that if $H$ and $E$ are $(\delta,\delta,\delta)$-interleaved w.r.t. $R$, 
	then $H^\gamma$ and $E^\gamma$ are $\epsilon$-interleaved w.r.t. $R$. 
	
	Let $H^{\gamma}_{\epsilon/3}$ be the $I_{\gamma}$-hierarchical clustering of $X$ 
	with $H^{\gamma}_{\epsilon/3}(r) = H^{\gamma}(r)$ for $r \in I_\gamma$ with $r > \epsilon/3$, 
	and $H^{\gamma}_{\epsilon/3}(r) = \{X\}$ else. 
	Define $E^{\gamma}_{\epsilon/3}$ in the same way. 
	Since $H^{\gamma}_{\epsilon/3}$ and $H^{\gamma}$ are $\epsilon/3$-interleaved, 
	and similarly for $E^{\gamma}$, it suffices to show that 
	$H^{\gamma}_{\epsilon/3}$ and $E^{\gamma}_{\epsilon/3}$ are $\epsilon/3$-interleaved w.r.t. $R$. 
	We first show that 	$E^\gamma(r_1) = \emptyset$. 
    It follows that $E^\gamma(r) = H^\gamma(r) = \emptyset$ for all $r \geq r_1$. 
    By the definition of $c$, we have 
	$\gamma_s(r_1) + c(r_1-r_0) \leq \gamma_s(r_0)$, 
	$\gamma_t(r_1) + c(r_1-r_0) \leq \gamma_t(r_0)$, and  
    $\gamma_k(r_1) - c(r_1-r_0) \geq \gamma_k(r_0)$. 
    Using these equations, the $(\delta,\delta,\delta)$-interleaving between $H$ and $E$, 
    and the assumption that $\delta \leq c(r_1-r_0)$, 
    we have $\pi_Y^*(E^{\gamma})(r_1) \preceq \pi_X^*(H^{\gamma})(r_0) = \emptyset$, 
    and thus $E^\gamma(r_1) = \emptyset$. 
    Now, to show that $H^{\gamma}_{\epsilon/3}$ and $E^{\gamma}_{\epsilon/3}$ 
    are $\epsilon/3$-interleaved w.r.t. $R$, 
    it suffices to show that, for $r \in (2\epsilon/3, r_1)$, 
    we have $\pi_X^*(H^{\gamma}_{\epsilon/3})(r) \preceq \pi_Y^*(E^{\gamma}_{\epsilon/3})(r - \epsilon/3)$, and 
    $\pi_Y^*(E^{\gamma}_{\epsilon/3})(r) \preceq \pi_X^*(H^{\gamma}_{\epsilon/3})(r - \epsilon/3)$. 
    Again by the definition of $c$, we have 
	$\gamma_s(r) + c(\epsilon/3) \leq \gamma_s(r - \epsilon/3)$, 
	$\gamma_t(r) + c(\epsilon/3) \leq \gamma_t(r - \epsilon/3)$, and  
    $\gamma_k(r) - c(\epsilon/3) \geq \gamma_k(r - \epsilon/3)$ 
    for any $r \in (2\epsilon/3, r_2)$. 
    Using these equations, the $(\delta,\delta,\delta)$-interleaving between $H$ and $E$, 
    and the assumption that $\delta \leq c\epsilon/3$, 
    we obtain the desired relations.
\end{proof}

\begin{lemma} \label{main-consistency}
Let $\{\gamma^{\theta}\}_{\theta \in \Theta}$ 
be an admissible family of slicing curves, 
and let $X_n$ be a sample of $f$. For every $\epsilon > 0$ there exist $\theta \in \Theta$ and $\delta > 0$
such that, if $\dP(\mu_n,\mu_f),\dH(X_n,\support(f)) < \delta$, then
$\mml(X_n)^{\overline{\gamma^{\theta}}}$ and $H(f)$ are
$\epsilon$-interleaved with respect to the closest point correspondence $R_c \subseteq X_n \times \support(f)$.
\end{lemma}

\begin{proof}
    By \cref{DR-f-approximates-H-f}, and the fact that the family $\{\gamma^{\theta}\}_{\theta \in \Theta}$
    is admissible, we can fix the parameter $\theta$ so that
    $H(f)$ and $\mml(f)^{\overline{\gamma^{\theta}}}$ are $\epsilon/2$-interleaved.
    It is then enough to show that we can choose $\delta > 0$
    such that, if $\dP(\mu_n,\mu_f) < \delta$ and $\dH(X_n,\support(f)) < \delta$, then
    $\mml(f)^{\overline{\gamma^{\theta}}}$ and $\mml(X_n)^{\overline{\gamma^{\theta}}}$ 
    are $\epsilon/2$-interleaved with respect to $R_c$. 
    To see that this can be done, note that the operation
    $\mml(-)^{\overline{\gamma^{\theta}}}$ is the composite of $\mml(-)$
    and slicing by $\overline{\gamma^{\theta}}$, 
	and apply \cref{continuity-of-mml} 
	(note that the interleaving constructed in the proof 
	is with respect to $R_c$) 
	and \cref{continuity-wrt-input}, 
	where the last result applies by 
	\cref{consistent-curve-consistent-height}, since $\overline{\gamma^{\theta}}$ is covering.
\end{proof}

\begin{lemma}\label{convergenceofsample}
    Let $(\MPS,d,\mu)$ be a compact metric probability space with full support and let $X_n$ be an i.i.d.~$n$-sample of $\MPS$,
    seen as a subspace of $\MPS$. Let $\epsilon > 0$.
    Then, the probability that $\max(\dP(\mu_n,\mu),\dH^\MPS(X_n,\MPS)) > \epsilon$ goes to $0$ as $n \to \infty$.
    Here $\mu_n$ is the empirical measure given by the sample $X_n$.
\end{lemma}

\begin{proof}
    We show that $P(\dH^\MPS(X_n,\MPS) > \epsilon) \to 0$ as $n \to \infty$.
    Since $\MPS$ is compact, for any $\epsilon > 0$
    we can cover $\MPS$ with finitely many $\epsilon$-balls, all of which have positive measure,
    by assumption. This implies that the probability that there is a sample point inside of each
    of these goes to $1$ as $n$ goes to $\infty$. 
    We conclude by showing that $P(\dP(\mu_n,\mu) > \epsilon) \to 0$ as $n \to \infty$. 
    This follows from the facts that 
    i.i.d.~samples of a separable metric space with
    the empirical measure converge weakly to the sampled space, in probability
    \cite[Chapter~II, Theorem~7.1]{parthasarathy}, and that 
    weak convergence implies convergence in the Prokhorov distance 
    \cite[Section~6]{billingsley}.
\end{proof}

\begin{theorem} \label{consistency-gamma-linkage}
Let $\{ \gamma^{\theta} \}_{\theta \in \Theta}$ be an admissible family of slicing curves. 
The hierarchical clustering algorithm $\theoremlinkage{\overline{\gamma}}$ 
with parameter space $\Theta$, defined using any kernel $K$, 
is CI-consistent with respect to 
any continuous, compactly supported probability density function $f : \bbR^d \to \bbR$.
\end{theorem}

\begin{proof}
    The theorem follows from \cref{main-consistency} and the fact that samples converge
    to the space being sampled, \cref{convergenceofsample}.
\end{proof}

\subsection{Details from \cref{Structure-one-parameter}} \label{appendix-structure-one-parameter}

\begin{lemma}
    \label{lemma:splits-from-merges}
    Let $I \subseteq \bbR$ be an interval, let $H$ be an $I$-hierarchical clustering, and let $\bfC < \bfD \in \PC(H)$.
    Then, there exists $\bfE \in \PC(H)$ such that $\bfE < \bfD$ and $\bfC$ and $\bfE$ are incomparable.
\end{lemma}
\begin{proof}
    Let $r < s \in I$ and $C \in H(r)$ and $D \in H(s)$ such that $[C] = \bfC$ and $[D] = \bfD$.
    Since $\bfC \neq \bfD$, there must exist $t \in [r,s]$ and $E \in H(t)$ such that $E \subseteq D$ but $C \not\subseteq E$.
    Then, the persistent cluster $\bfE \in \PC(H)$ satisfies the required conditions.
\end{proof}

\begin{lemma}
    \label{lemma:underlying-clusters-are-disjoint}
    Let $I \subseteq \bbR$ be an interval, $H$ be an $I$-hierarchical clustering, and let $\bfC, \bfD \in \PC(H)$.
    If $U(\bfC) \cap U(\bfD) \neq \emptyset$, then $U(\bfC) \subseteq U(\bfD)$ or $U(\bfD) \subseteq U(\bfC)$.
    Moreover, if $U(\bfC) = U(\bfD)$, then $\bfC = \bfD \in \PC(H)$.
\end{lemma}

\begin{proof}
Say $U(\bfC) \cap U(\bfD) \neq \emptyset$, and let $x \in U(\bfC) \cap U(\bfD)$. 
Choose $r \in \life(\bfC)$ with $x \in \bfC(r)$ and $r' \in \life(\bfD)$ with $x \in \bfD(r')$. 
Without loss of generality, $r \leq r'$, and thus $\bfC(r) \subseteq \bfD(r')$. 
We show $U(\bfC) \subseteq U(\bfD)$. 
Say $y \in U(\bfC)$. Choose $i \in \life(\bfC)$ with $y \in \bfC(i)$. 
We may assume $i \geq r$. 
If $r' \in \life(\bfC)$, then $\bfC(r') = \bfD(r')$, and thus $\bfC = \bfD$. 
So, assume $r' \not\in \life(\bfC)$. 
As $\life(\bfC)$ is an interval, $i < r'$, and thus $\bfC(i) \subseteq \bfD(r')$. 
So, $y \in \bfD(r') \subseteq U(\bfD)$.

Now, assume $U(\bfC) = U(\bfD)$. 
As in the first step above, we have without loss of generality 
$\bfC(r) \subseteq \bfD(r')$ for some $r \leq r'$. 
We show $\bfC(r) \sim \bfD(r')$. 
Let $r'' \in [r,r']$, and let $A,A' \in H(r'')$ with $A,A' \subseteq \bfD(r')$. 
We have $A,A' \subseteq U(\bfD) = U(\bfC)$. 
Let $B \in H(r'')$ be the cluster such that $\bfC(r) \subseteq B$. 
It is straightforward to show $A=B$, and similarly $A'=B$. 
Thus $A=A'$, finishing the proof.
\end{proof}

\begin{lemma}
    \label{lemma:upset-linearly-ordered}
    Let $I \subseteq \bbR$ be an interval, let $H$ be an $I$-hierarchical clustering, and let $\bfC < \bfD \in \PC(H)$.
    Then, the set $\{ \bfD \in \PC(H) : \bfD \geq \bfC\}$ is linearly ordered in $\PC(H)$.
\end{lemma}
\begin{proof}
    Let $r,r' \in I$, $D \in H(r)$, $D'\in H(r')$ such that $[D] = \bfD$ and $[D'] = \bfD'$.
    Assume, without loss of generality, that $r \leq r'$.
    Since $\bfC \leq \bfD, \bfD'$, we have $U(\bfC) \subseteq D$ and $U(\bfC) \subseteq D'$.
    If follows that $U(\bfD) \cap U(\bfD') \neq \emptyset$ and thus $\bfD \leq \bfD'$ or $\bfD' \leq \bfD$, by \cref{lemma:underlying-clusters-are-disjoint}.
\end{proof}

\begin{definition}
Let $H$ be a one-parameter hierarchical clustering and let $\bfC \in \PC(H)$.
The \define{successor} of $\bfC$ is, if it exists, the unique minimal element of $\{ \bfD \in \PC(H) : \bfD > \bfC\}$, where uniqueness follows from the fact that $\{ \bfD \in \PC(H) : \bfD > \bfC\}$ is linearly ordered (\cref{lemma:upset-linearly-ordered}).
\end{definition}

We now recall the terminology of persistence modules; 
we refer the interested reader to \cite{chazal-silva-glisse-oudot,bauer-lesnick-matchings} for details. 
Fix a field $\fieldk$. 
Let $I \subseteq \bbR$ be an interval, and let $H$ be an $I$-hierarchical clustering of a set $X$. 
For each $r \in I$, consider the vector space $\fieldk H(r)$ that is freely generated by the clusters of $H(r)$. 
For $r \leq r'$, there is a linear map $\fieldk H(r) \to \fieldk H(r')$ 
defined on the basis given above by $H(r \leq r')$. 
This data gives a functor $\fieldk H : I \to \kvec$, 
where $\kvec$ is the category of vector spaces over $\fieldk$. 
Analogously to \cref{P-HC-def}, we call such a functor an \define{$I$-persistence module}. 

Let $I \subseteq \bbR$ be an interval and let $M : I \to \kvec$ be an $I$-persistence module.
Given $r \leq r' \in I$, define the \define{rank invariant} \citep{carlsson-zomorodian} of $M$ at $r \leq r'$ as $\rk(M)(r \leq r') = \rk(M(r) \to M(r'))$.
Note that $\rk(H) = \rk(\fieldk H)$.

For our purposes, given a set $A$, a \define{multiset} of elements of $A$ consists of an indexing set $J$ and a function $a : J \to A$.
We usually denote such a multiset by $\{a_j\}_{j \in J}$.
We say that two multisets $\{a_j\}_{j \in J}$ and $\{b_k\}_{k \in K}$ of elements of $A$ are \define{equal} if there exists a bijection $\beta : J \to K$ such that $a_j = b_{\beta(j)}$ for all $j \in J$.

Let $B \subseteq I \subseteq \bbR$ be inclusions of intervals.
The $I$-persistence module $\fieldk_{B} : I \to \kvec$ is the functor taking the value $\fieldk$ on every $r \in B$ and the value $0$ elsewhere, with structure morphism being the identity $\fieldk \to \fieldk$ whenever that is possible.
The following theorem is due to Crawley-Boevey.

\begin{theorem}[{\citealt[Theorem~1.1]{crawley}}]
    \label{theorem-decomposition}
    Let $\fieldk$ be a field, let $I \subseteq \bbR$ be an interval, and let $M : I \to \kvec$ be an $I$-per\-sis\-tence module.
    If $M$ is pointwise finite-dimensional, then there exists a unique multiset of intervals $\{\barA_j \subseteq I\}_{j \in J}$ such that $M$ is isomorphic to $\bigoplus_{j \in J} \fieldk_{\barA_j}$.
\end{theorem}

\begin{lemma}
    \label{lemma:characterization-barcode-as-decomposition}
    Let $I \subseteq \bbR$ be an interval.
    Let $H$ be a pointwise finite $I$-hierarchical clustering and let $\fieldk$ be any field.
    Then $\barc(H)$ is the unique multiset $\{B_j \subseteq I\}_{j \in J}$ such that $\fieldk H \cong \bigoplus_{j \in J} \fieldk_{B_j}$.
\end{lemma}
\begin{proof}
    By definition, the $I$-persistence module $\fieldk H : I \to \kvec$ is pointwise finite dimensional, in the sense of \cite{crawley}.
    Thus, by \cref{theorem-decomposition}, there exist a unique multiset of intervals $\{\barA_j \subseteq I\}_{j \in J}$ such that $\fieldk H \cong \bigoplus_{j \in J} \fieldk_{\barA_j}$.
    By unfolding definitions, we get 
    \[
        \rk\left(\bigoplus_{j \in J} \fieldk_{\barA_j}\right)(r \leq r') = 
        \sum_{j \in J} \rk\left(\fieldk_{\barA_j}\right)(r \leq r') =
        |\{j \in J : r,r' \in \barA_j \}|,
    \]
    so $\{B_j \subseteq I\}_{j \in J}$ is a barcode for $H$.

    The fact that the barcode is unique is a particular case of \cite[Proposition~2.8]{botnan-oppermann-oudot}, where the poset $P$ is taken to be $I$ and the collections of intervals $\widehat{\mathcal{I}}$ and $\mathcal{I}$ are all intervals included in $I$.
\end{proof}

\begin{proof}
\textbf{(of \cref{theorem:barcode-existence-uniqueness})} 
This follows at once from \cref{theorem-decomposition} and \cref{lemma:characterization-barcode-as-decomposition}.
%
\end{proof}

\begin{lemma}
    \label{lemma:number-of-leaves-independent-clusters}
Let $H$ be a finite $I$-hierarchical clustering.
Then
\[
    |\leaves(H)| = \max \{k \in \bbN : \exists \{r_j\in I\}_{1 \leq j \leq k}, \{D_j \in H(r_j)\}_{1\leq j \leq k} \text{ s.t. } D_i \subseteq D_j \Rightarrow i = j \},
\]
and any set $\{D_j \in H(r_j)\}_{1\leq j \leq k}$ attaining the maximum must be such that $\{[D_j] \in \PC(H)\} = \leaves(H)$.
\end{lemma}
\begin{proof}
    The inequality $(\leq)$ follows immediately by taking the set of clusters $\{D_j\}$ to be a set of representatives of the leaves of $H$.
    The inequality $(\geq)$ follows from the fact that, given a set $\{D_j \in H(r_j)\}_{1\leq j \leq k}$ as in the statement, the set $\{[D_j] \in \PC(H)\}_{1\leq j \leq k}$ forms an antichain of $\PC(H)$ and thus its cardinality is bounded above by the cardinality of the set of minimal elements of $\PC(H)$, which is, by definition, the set of leaves.
    To prove the last claim, note that any antichain of $\PC(H)$ of cardinality $|\leaves(H)|$ must necessarily be the set $\leaves(H)$ itself, since $\PC(H)$ forms a forest.
\end{proof}

\begin{lemma}
    \label{lemma:removing-leaf}
    Let $H$ be a finite $I$-hierarchical clustering of a set $\SET$.
    Assume that $H$ is not constantly empty and let $\bfC \in \leaves(H)$.
    Then
    \begin{enumerate}
        \item $H \setminus \bfC$ is a finite $I$-hierarchical clustering;
        \item $|\leaves(H\setminus \bfC)| = |\leaves(H)| - 1$;
        \item If $|\leaves(H)| \geq 2$, then $\min_{\bfD \in \leaves(H)} \length(\bfD) \leq \min_{\bfD' \in \leaves(H \setminus \bfC)} \length(\bfD')$.
    \end{enumerate}
\end{lemma}

\begin{proof}
Note first that if $D_1 \in H(r_1)$ and $D_2 \in H(r_2)$ are such that 
$[D_1] = [D_2]$ in $\PC(H)$ and $[D_1] \neq \bfC$, then 
$[D_1] = [D_2]$ in $\PC(H \setminus \bfC)$. 
Using this fact, we can define a function 
$\varphi : \PC(H) \setminus \bfC \to \PC(H \setminus \bfC)$ as follows. 
For $\bfD \in \PC(H) \setminus \bfC$, pick any $D \in \bfD$, 
and let $\varphi(\bfD) = [D] \in \PC(H \setminus \bfC)$. 
We now prove that $\varphi$ is surjective, which implies the first statement of the lemma. 
Let $\bfD \in \PC(H \setminus \bfC)$, and let $D \in \bfD$. 
Then $[D] \neq \bfC$ in $\PC(H)$ by definition of $H \setminus \bfC$. 
So, $\varphi([D]) = \bfD$.

Next, we show that, when $\varphi$ is restricted to $\leaves(H) \setminus \bfC$, 
$\varphi$ is a bijection between $\leaves(H) \setminus \bfC$ and $\leaves(H \setminus \bfC)$; 
this proves the second statement of the lemma. 
For this, we will use the following fact, which is straightforward to check. 
If $G$ is an $I$-hierarchical clustering, and $\mathbf{A} \in \PC(G)$, 
then $\mathbf{A}$ is a leaf if and only if for any $r \in \life(\mathbf{A})$ 
and for any $r' \in I$ with $r' < r$, there is at most one cluster $B \in H(r')$ with 
$B \subseteq \mathbf{A}(r)$. 
Now, let $\bfD \in \leaves(H) \setminus \bfC$. We need to show $\varphi(\bfD)$ is a leaf. 
Let $D = \bfD(r)$ for some $r$, so that $\varphi(\bfD) = [D]$. 
If there is $r' < r$ and $B,B' \in (H \setminus \bfC)(r')$ with $B,B' \subseteq D$, 
then as $\bfD$ is a leaf of $H$, we have $B = B'$. 
Let $s \in \life(\varphi(\bfD))$, and let $s' \in I$ with $s' < s$. 
If $s' < r$, then we have shown that there is at most one cluster in 
$(H \setminus \bfC)(s')$ contained in $\varphi(\bfD)(s)$. 
If $r \leq s'$, then as $\varphi(\bfD)(r) \sim \varphi(\bfD)(s)$, 
there is exactly one cluster in $(H \setminus \bfC)(s')$ contained in $\varphi(\bfD)(s)$. 
So, $\varphi(\bfD)$ is a leaf. 

We show that $\varphi|_{\leaves(H) \setminus \bfC}$ is injective. 
Say $\mathbf{D}, \mathbf{E} \in \leaves(H) \setminus \bfC$ and $\mathbf{D} \neq \mathbf{E}$. 
Let $D = \mathbf{D}(r)$ and $E = \mathbf{E}(r')$ for some $r,r'$. 
As $\mathbf{D}$ and $\mathbf{E}$ are distinct leaves, 
\cref{lemma:underlying-clusters-are-disjoint} implies $D \cap E = \emptyset$. 
As $D = \varphi(\mathbf{D})(r)$ and $E = \varphi(\mathbf{E})(r')$, 
we have $\varphi(\mathbf{D}) \neq \varphi(\mathbf{E})$.

Next we show that the image of $\varphi|_{\leaves(H) \setminus \bfC}$ 
is $\leaves(H \setminus \bfC)$. Say $\bfD \in \leaves(H \setminus \bfC)$. 
Let $D = \bfD(r)$ for some $r$. 
If $[D] \in \leaves(H)$, then we are done, since $\varphi([D]) = \bfD$. 
So, say $[D]$ is not a leaf of $H$. 
Then there is $r' \in I$ with $r' < r$ and $B,B' \in H(r')$ with $B \neq B'$ 
and $B,B' \subseteq D$. Without loss of generality, $B \not\in (H \setminus \bfC)(r')$, 
so $[B] = \bfC$ in $\PC(H)$. 
Thus, $[B'] \neq \bfC$, so that $B' \in (H \setminus \bfC)(r')$. 
As $\bfD$ is a leaf of $H \setminus \bfC$, 
we have $[B'] = \bfD$ in $\PC(H \setminus \bfC)$, 
so that $\varphi([B']) = \bfD$. 
We show that $[B'] \in \leaves(H) \setminus \bfC$. 
Say there is $r'' < r'$ and $A,A' \in H(r'')$ with $A,A' \subseteq B'$. 
Then $A \cap B = \emptyset$ and $A' \cap B = \emptyset$, 
so that $[A] \neq \bfC$ and $[A'] \neq \bfC$ in $\PC(H)$. 
Thus, $A,A' \in (H \setminus \bfC)(r'')$, 
and since $\bfD$ is a leaf of $H \setminus \bfC$, $A = A'$. 

We have shown that $\varphi|_{\leaves(H) \setminus \bfC}$ 
is a bijection between $\leaves(H) \setminus \bfC$ and $\leaves(H \setminus \bfC)$, 
proving the second statement of the lemma. 
We will also use this fact to prove the third statement of the lemma. 
Note first that if $\bfD \in \PC(H) \setminus \bfC$, 
then $\life(\bfD) \subseteq \life(\varphi(\bfD))$, 
and so $\length(\bfD) \leq \length(\varphi(\bfD))$. 
Now, say $|\leaves(H)| \geq 2$, so that $|\leaves(H \setminus \bfC)| \geq 1$. 
Choose $\mathbf{E'} \in \leaves(H \setminus \bfC)$ 
such that $\length(\mathbf{E'}) = \min_{\bfD' \in \leaves(H \setminus \bfC)} \length(\bfD')$. 
Let $\mathbf{E} \in \leaves(H) \setminus \bfC$ be such that $\varphi(\mathbf{E}) = \mathbf{E'}$. 
Then
$\min_{\bfD \in \leaves(H)} \length(\bfD) \leq 
\length(\mathbf{E}) \leq 
\length(\mathbf{E'}) = 
\min_{\bfD' \in \leaves(H \setminus \bfC)} \length(\bfD').$
\end{proof}

\begin{lemma}
    \label{lemma:inductive-step-elder-rule}
    Let $H$ be a finite $I$-hierarchical clustering that is not constantly empty, and let $\bfC$ be a minimal leaf of $H$.
    Then $\fieldk H \cong \fieldk(H \setminus \bfC) \oplus \fieldk_{\life(\bfC)}$.
\end{lemma}
\begin{proof}
    Let $H$ be a finite $I$-hierarchical clustering of a set $\SET$.
    For convenience, denote $H' = H \setminus \bfC$.
    We start by defining a morphism of $I$-persistence modules
    $\fieldk H' \oplus \fieldk_{\life(\bfC)} \to \fieldk H$ as a sum of two morphisms $\phi : \fieldk H' \to \fieldk H$ and $\psi : \fieldk_{\life(\bfC)} \to \fieldk H$.
    Let $\phi : \fieldk H' \to \fieldk H$ be given by mapping the basis element of $\fieldk H'(r)$ corresponding to $D \in H'(r)$ to the basis element of $\fieldk H(r)$ corresponding to the same cluster $D \in H(r)$.

    To define $\psi$, we consider two cases.
    If $\bfC$ has no successor in $\PC(H)$, then the morphism $\psi : \fieldk_{\life(\bfC)} \to \fieldk H$ defined by mapping the basis element $1 \in \fieldk = \fieldk_{\life(\bfC)}(r)$ to the basis element of $\fieldk H(r)$ corresponding to $\bfC(r) \in H(r)$ is well-defined.

    If $\bfC$ does have a successor $\bfD$, let $\bfA$ denote any leaf of $H$ smaller than $\bfD$ and different from $\bfC$, which must exist by \cref{lemma:splits-from-merges}.
    Since $\bfC$ is a minimal leaf, for each $r \in \life(\bfC)$, there exists $r' \in \life(\bfA)$ such that $r' \leq r$.
    Let $\bfA' : \life(\bfC) \to \C(\SET)$ denote the persistent cluster defined as $\bfA'(r) = H(r' \leq r)(\bfA(r')) \in H(r)$.
    Let $\psi : \fieldk_{\life(\bfC)} \to \fieldk H$ be defined by mapping the basis element $1 \in \fieldk = \fieldk_{\life(\bfC)}(r)$ to the subtraction of basis elements of $\fieldk H(r)$ given by $\bfC(r) - \bfA'(r)$.
    In order to see that this is well-defined, note that, if $t \in \life(\bfD)$, then $H(r \leq t)(\bfC(r)) = H(r \leq t)(\bfA'(r))$, since both $\bfC$ and $\bfA$ are smaller than $\bfD$ in $\PC(H)$, and thus $\bfC(r) - \bfA'(r)$ maps to $0$ in $\fieldk H(t)$ as soon as $t$ is larger than all elements in $\life(\bfC)$.

    To conclude the proof, it is enough to prove that the morphism $\phi + \psi : \fieldk H' \oplus \fieldk_{\life(\bfC)} \to \fieldk H$ is an epimorphism, since, in that case, it must also by a monomorphism, by dimension-counting.
    Let $r \in I$; it is clear that all basis elements of $\fieldk H(r)$ corresponding to clusters $B \in H(r)$ such that $[B] \neq \bfC$ are in the image of $\phi + \psi$, since they are already in the image of $\phi$, by construction.
    It is then enough to prove that basis elements of $\fieldk H(r)$ corresponding to clusters $C \in H(r)$ such that $[C] = \bfC$ are in the image of $\phi + \psi$, and for this we necessarily have $r \in \life(\bfC)$.
    To conclude, note that the element of $\fieldk H(r)$ corresponding to $\bfA'(r)$ is in the image of $\phi$, since it is a basis element of $\fieldk H'(r)$, and the element corresponding to $\bfC(r) - \bfA'(r)$ is in the image of $\psi$, by construction.
\end{proof}

\begin{proof}
\textbf{(of \cref{proposition:elder-rule})} 
    Note that the process is well-defined and terminates in $|\leaves(H)|$ by \cref{lemma:removing-leaf}.
    By induction and \cref{lemma:inductive-step-elder-rule}, we have $\fieldk H = \bigoplus_{1 \leq i \leq k} \fieldk_{\life(\bfC_i)}$, so the claim follows from \cref{lemma:characterization-barcode-as-decomposition}.
\end{proof}

\begin{corollary}
    \label{corollary:leaf-bar-length}
    Let $H$ be a finite hierarchical clustering.
    Then, 
    \[
        \min_{\bfC \in \leaves(H)} \length(\bfC) = \min_{B \in \barc(H)} \length(B).
    \]
\end{corollary}
\begin{proof}
    This follows from the first step in \cref{proposition:elder-rule} and the third claim of \cref{lemma:removing-leaf}.
\end{proof}

\begin{lemma}
\label{lemma-examples-tame}
    For any $\lambda \in \{\lambda\con^{x,y}\}_{x,y > 0}$ and any compact metric probability space $\MPS$, the hierarchical clustering $\linkage{\lambda}(\MPS)$ is essentially finite.
    If $f : \bbR^d \to \bbR$ is continuous and compactly supported, then $H(f)$ is essentially finite.
\end{lemma}
\begin{proof}
    Let $H$ be an $\bbR$-hierarchical clustering.
    For the purposes of this proof, and in analogy with the notion of q-tameness introduced in \cite{chazal-etal-interleavings}, we say that $H$ is \define{q-tame} if, for every $r < r' \in \bbR$, the cardinality of the image of the function $H(r \leq r') : H(r) \to H(r')$ is finite, and, we say that $H$ is \define{bounded} if there exist $s\leq t \in \bbR$ such that $H$ is constant on $(-\infty, s)$ and on $(t, \infty)$.
    We start by proving that the hierarchical clusterings of interest are q-tame and bounded.

    Let $\MPS$ be a compact metric probability space.
    Then, the extension of $\linkage{\lambda}(\MPS)$ to an $\bbR$-hierarchical clustering is q-tame since $\linkage{\lambda}(\MPS)$ is pointwise finite, as its values are a single-linkage clustering of a totally bounded metric space.
    The extension is also bounded, since it is an extension of a hierarchical clustering defined over a finite interval.

    Let $f : \bbR^d \to \bbR$ be continuous and compactly supported.
    The hierarchical clustering $H(f)$ is q-tame by \citet[Theorem~2]{cagliari-landi} (or \citealt[Theorem~3.33]{chazal-silva-glisse-oudot}), and it is bounded since $f$ takes values in a compact set. 
    Then, the result follows from the following claim.
    
    \smallskip
    \noindent \textit{Claim.} A bounded $\bbR$-hierarchical clustering is q-tame if and only if it is essentially finite.

    \smallskip
    \noindent \textit{Proof of claim.}
    Let $H$ be a bounded $\bbR$-hierarchical clustering.

    Assume that $H$ is essentially finite.
    If $r < r' \in \bbR$, there exists $\tau > 0$ such that $r\leq r' - \tau$; thus $H(r\leq r') = H(r' - \tau \leq r') \circ H(r \leq r' - \tau)$.
    The image of the function $H(r' - \tau \leq r')$ is finite, since $H_{\geq \tau}$ is a finite hierarchical clustering.
    It follows that the image of $H(r \leq r')$ is finite, and thus that $H$ is q-tame.

    Now assume that $H$ is q-tame, and, towards a contradiction, let $\tau > 0$, assume $\PC(H_{\geq \tau})$ is infinite, and let $\{\bfC_n\}_{n \geq 0}$ be a countably infinite family of elements of $\PC(H_{\geq\tau})$.
    Let $s \leq t \in \bbR$ be such that $H_{\geq \tau}$ is constant on $(-\infty,s)$ and on $(t,\infty)$; such $s$ and $t$ must exist since $H$ is bounded.
    Only finitely many elements of $\{\bfC_n\}$ can have support intersecting $(-\infty, s) \cup (t,\infty)$, since $H_{\geq \tau}$ is constant on $(-\infty,s)$ and on $(t,\infty)$.
    Thus, after taking a subsequence of $\{\bfC_n\}$ we may assume that all intervals $\life(\bfC_n)$ are contained in $[s,t]$.
    Since $[s,t]$ is compact, after taking a subsequence of $\{\bfC_n\}$, we may assume that, as $n \to \infty$, we have
    $b_n := \birth(\bfC_n) \to b$ and $d_n := \death(\bfC_n) \to d$, with $b \leq d$.
    Let $s_n = (b_n + d_n)/2$ and let $s = (b + d)/2$. 
    For each $n$, there is $C_n \in H_{\geq \tau}(s_n)$ with $[C_n] = \bfC_n$, 
    and therefore there is $C_n' \in H(s_n - \tau)$ with $H(s_n - \tau \leq s_n)(C_n') = C_n$. 
    As $n \to \infty$, we have $s_n \to s$, so there exists $n_0 \in \bbN$ such that $ s_n - \tau < s - \frac{2\tau}{3} < s - \tau/3 < s_n$ for all $n \geq n_0$.
    Thus, for $n \geq n_0$, the elements $H(s - \tau \leq s - \frac{\tau}{3})(C_n')$ form an infinite subset of $H(s - \frac{2\tau}{3} \leq s - \frac{\tau}{3})$, a contradiction.
    This concludes the proof of the claim.
%
%
\end{proof}

\begin{proof}
\textbf{(of \cref{bottleneck-correspondence-interleaving})} 
Note that, by \cref{lemma:characterization-barcode-as-decomposition}, the barcode of a pointwise finite HC $H$ is equal to the barcode of $\fieldk H$ for any field $\fieldk$. 
If $H$ and $E$ are $\epsilon$-interleaved with respect to some correspondence, then $\fieldk H$ and $\fieldk E$ are $\epsilon$-interleaved in the sense of \citet{bauer-lesnick-matchings}, so it
follows from \citet[Theorem~6.4]{bauer-lesnick-matchings} that there exists an $\epsilon$-matching between the barcodes of $H$ and $E$.
\end{proof}

\begin{proof}
\textbf{(of \cref{barcode-algorithm-independent-ordering})} 
As every permutation can be written as a composition of adjacent transpositions, 
it suffices to consider modifying an ordering $[\sigma_1, \dots, \sigma_p]$ 
by transposing $\sigma_i$ and $\sigma_{i+1}$. 
If this transposition results in an ordering satisfying the assumptions, 
then necessarily $f(\sigma_i) = f(\sigma_{i+1})$. There are three cases. 

In Case 1, $\sigma_i$ and $\sigma_{i+1}$ are both vertices. 
In this case it is clear that the transposition does not effect the output.
In Case 2, one of the simplices is a vertex $x$ and the other an edge $e$. 
In this case, since it is admissible to order $e$ before $x$, 
$x$ is not an endpoint of $e$. 
Thus, processing $e$ does not effect the connected component nor the bar introduced 
when processing $x$, so the transposition does not effect the output.
In Case 3, $\sigma_i = \{x,y\}$ and $\sigma_{i+1} = \{a,b\}$ are both edges. 
Let $(c_x, u_x)$ be the connected component containing $x$ after processing $\sigma_{i-1}$, 
and similarly for $y,a,b$. 
If $c_x = c_y$ or $c_a = c_b$, then it is clear that the transposition has no effect on the output. 
So, we assume $c_x \neq c_y$ and $c_a \neq c_b$. 
If $\{c_x,c_y,c_a,c_b\}$ has $2$ or $4$ elements, then it is straightforward to check 
that the transposition has no effect on the output. 
Say $\{c_x,c_y,c_a,c_b\}$ has $3$ elements. 
Without loss of generality, we assume $c_y = c_b$. 
Let $r = f(\{x,y\}) = f(\{a,b\})$ and let $u_1,u_2,u_3$ 
be $u_x,u_y,u_a$ ordered from smallest to largest. 
Then, after processing $\sigma_i$ and $\sigma_{i+1}$ in either order, 
we have
\begin{align*}
	\mathtt{conn\_comp} &\gets 
	\left(\mathtt{conn\_comp} \setminus \big\{(c_x,u_x), (c_y,u_y), (c_a,u_a)\big\}\right) 
	\cup \big\{(c_x \cup c_y \cup c_a, \min(u_x,u_y,u_a))\big\} \\
	\mathtt{barcode} &\gets 
	\left(\mathtt{barcode} \setminus \big\{[u_x,\infty), [u_y,\infty), [u_a,\infty)\big\}\right) 
	\cup \big\{[u_1,\infty), [u_2,r), [u_3,r)\big\}
\end{align*}
So, the transposition does not effect the output.
\end{proof} 

\begin{proof}
\textbf{(of \cref{barcode-algorithm-correctness})} 
Let $\SET$ be the set of vertices of $G$, and let $H = H(G,f)$. 
By \cref{example:hierarchical-clustering-finite-set}, $H$ is a finite hierarchical clustering. 
We begin by noting some useful facts about \cref{barcode-algorithm}. 
Note first that if $G = G_1 \cup \dots \cup G_q$ is the decomposition of $G$ into connected components, 
then the output of the algorithm on $(G,f)$ is equal to the union of 
the output of the algorithm on $(G_i, f|_{G_i})$, 
and similarly the barcode of $H$ is equal to the union of the barcodes of $H(G_i, f|_{G_i})$. 
So, if $G$ is non-empty, we may assume it is connected.  

Second, note that if $T$ is a minimum spanning tree of $(G,f)$, 
then we have $H(G,f) = H(T,f|_{T})$. 
Furthermore, the output of the algorithm on $(G,f)$ is equal to the output on $(T,f|_{T})$. 
To see this, order the simplices of $G$ such that, among all simplices $\sigma$ with $f(\sigma) = r$, 
we first take all the vertices, then all the edges in $T$, then all the edges not in $T$. 
Now, if we run the algorithm on $(G,f)$ and process an edge $e = \{x,y\}$ not in $T$, 
then the algorithm does nothing at this step, 
since $T$ contains a path between $x$ and $y$ such that every edge on this path has $f$ value less than or equal to $f(e)$. 
So, we may assume $G$ is a tree.

Third, note that if $G$ is a tree with an edge $e = \{x,y\}$ such that $f(e) = f(x)$, 
then the output of the algorithm on $(G,f)$ is equal to the output on $(\tilde{G}, \tilde{f})$, 
where $\tilde{G} = G / e$ arises from contracting the edge $e$, 
and $\tilde{f}(\sigma) = f(\sigma)$ for $\sigma \not\in \{x,y,e\}$, 
and $\tilde{f}(v_e) = f(y)$, where $v_e$ is the vertex onto which $e$ contracts. 
And, the barcodes of $H(G,f)$ and $H(\tilde{G}, \tilde{f})$ are equal; 
to see this, note that $H(G,f)$ and $H(\tilde{G}, \tilde{f})$ are $0$-interleaved 
with respect to the correspondence that identifies $x,y$ with $v_e$, 
so that there is a $0$-matching between the barcodes of $H(G,f)$ and $H(\tilde{G}, \tilde{f})$ 
by \cref{bottleneck-correspondence-interleaving}. 
So, we may assume $G$ contains no edges $e$ such that $f(e) = f(v)$ for $v \in e$.

Fourth, say that $G$ contains no edges $e$ such that $f(e) = f(v)$ for $v \in e$. 
Then for $x \in X$, there is a leaf $\bfC$ of $H$ such that $\bfC(\birth(\bfC)) = \{x\}$. 
This defines a bijection $X \iso \leaves(H)$.

We prove the correctness of the algorithm by induction on the number of leaves of $H$. 
Consider the case $|\leaves(H)| = 0$. 
Then $G$ is empty and the output of the algorithm is correct. 
Consider the case $|\leaves(H)| = 1$.  
We may assume $G$ contains no edges $e$ such that $f(e) = f(v)$ for $v \in e$, and thus  $|X|=1$. 
Then it's straightforward to check that the output of the algorithm is correct. 

Consider the case $|\leaves(H)| \geq 2$. 
Let $\bfC$ be a minimal leaf of $H$. 
We may assume $G$ is connected, so we have $\death(\bfC) < \infty$. 
We may assume $G$ contains no edges $e$ such that $f(e) = f(v)$ for $v \in e$, 
and thus $U(\bfC) = \{x\}$ for some $x \in X$. 
Let $f' : G \to \bbR$ coincide with $f$ on $G \setminus \{x\}$ 
and have $f'(x) = \death(\bfC)$. 
In this case, $H(G,f') = H \setminus \bfC$. 
By \cref{lemma:removing-leaf} and the inductive hypothesis, 
the algorithm is correct on the input $(G,f')$. 
So, by \cref{proposition:elder-rule}, it is enough to prove that the output of the algorithm on $(G,f)$ 
is equal to the union of the output on $(G,f')$ and the interval $\life(\bfC)$. 
Say we run the algorithm on $(G,f)$ using the ordering $[\sigma_1, \dots, \sigma_p]$. 
Let $\sigma_i = \{x,y\}$ be the first edge adjacent to $x$ in this ordering. 
We run the algorithm on $(G,f')$ with the ordering obtained from $[\sigma_1, \dots, \sigma_p]$ 
by moving $x$ to the first position such that the ordering satisfies the assumptions of the algorithm. 
Let $(d,v)$ be the connected component with $y \in d$ after processing $\sigma_{i-1}$ 
(this is the same whether running the algorithm on $(G,f)$ or $(G,f')$). 
We have $f(\sigma_i) = \death(\bfC)$, 
and as $\bfC$ is a minimal leaf, we have $v \leq f(x)$, else $H$ would contain a leaf shorter than $\bfC$. 
After processing $\sigma_i$ on the input $(G,f)$, $\mathtt{conn\_comp}$ and $\mathtt{barcode}$ are updated as
\begin{align*}
	\mathtt{conn\_comp} &\gets 
	\left(\mathtt{conn\_comp} \setminus \big\{(\{x\},f(x)), (d,v)\big\}\right) 
	\cup \big\{(\{x\}\cup d,v)\big\}\\
	\mathtt{barcode} &\gets 
	\left(\mathtt{barcode} \setminus \big\{[f(x),\infty)\big\}\right) \cup \big\{[f(x),\death(\bfC))\big\}
\end{align*}
After processing $\sigma_i$ on the input $(G,f')$, the variable $\mathtt{conn\_comp}$ is updated in the same way, 
and $\mathtt{barcode}$ is updated as
\[
	\mathtt{barcode} \gets 
	\left(\mathtt{barcode} \setminus \big\{[\death(\bfC),\infty)\big\}\right) \cup \big\{[\death(\bfC),\death(\bfC))\big\}.
\]
It follows that the output of the algorithm on $(G,f)$ 
is equal to the union of the output on $(G,f')$ and the interval $\life(\bfC)$, as desired.
\end{proof}

\begin{lemma}
    \label{lemma:best-matching-is-identity}
    Let $P,Q : \bbN \to [0,\infty]$ be prominence diagrams and let $\sigma : \bbN \to \bbN$ be a bijection.
    Then, 
    \[
        \sup_{i \in \bbN} |P(i) - Q(i)| \leq \sup_{i \in \bbN} |P(i) - Q(\sigma(i))|.
    \]
\end{lemma}
\begin{proof}
    Consider the bijection $\sigma_1 : \bbN \to \bbN$ given by
    \[
        \sigma_1(i) =
        \begin{cases}
        0 & \text{ if $i = 0$}\\
        \sigma(0) & \text{ if $i = \sigma^{-1}(0)$}\\
        \sigma(i) & \text{ otherwise}.
        \end{cases}
    \]
    In words, the bijection $\sigma_1$ coincides with $\sigma$, except that it maps $0$ to $0$ and $\sigma^{-1}(0)$ to $\sigma(0)$.
    We then proceed inductively by considering a bijection $\sigma_{j+1}$ which coincides with $\sigma_j$ except that it maps $j$ to $j$ and $\sigma_{j}^{-1}(j)$ to $\sigma_{j}(j)$.

    As $j \to \infty$, the bijection $\sigma_j$ converges pointwise to the identity function $\bbN \to \bbN$.
    Since both $P(i)$ and $Q(i)$ converge to $0$ as $i \to \infty$, it follows that $\sup_{i \in \bbN} |P(i) - Q(\sigma_{j}(i))|$ converges to $\sup_{i \in \bbN} |P(i) - Q(i)|$ as $j \to \infty$.
    Then, the result follows from the following claim.

    \medskip
    \noindent \textit{Claim.} 
    For all $j \geq 1 \in \bbN$, we have
    \[
        \sup_{i \in \bbN} |P(i) - Q(\sigma_{j+1}(i))| \;\leq \; \sup_{i \in \bbN} |P(i) - Q(\sigma_j(i))|.
    \]

    \smallskip
    \noindent \textit{Proof of claim.}
    Since $\sigma_j$ and $\sigma_{j+1}$ coincide except on $j$ and $\sigma^{-1}_j(j)$, it is sufficient to prove that 
    \begin{align*}
        \max\left( | P(j) - Q(\sigma_{j+1}(j)) | , |P(\sigma^{-1}_j(j)) - Q(\sigma_{j+1}(\sigma^{-1}_j(j)))| \right)\\
        \leq \max\left( | P(j) - Q(\sigma_j(j)) | , |P(\sigma^{-1}_j(j)) - Q(j)| \right)
    \end{align*}
    By definition of $\sigma_{j+1}$, the left-hand side of the inequality is equal to 
    \[
        \max\left( | P(j) - Q(j) | , |P(\sigma^{-1}_j(j)) - Q(\sigma_j(j))| \right).
    \]
    Recall that, if $a \leq a' \in \bbR$ and $b \leq b' \in \bbR$ we have $\max(|a - b|, |a' - b'|) \leq \max(|a- b'| , |a' - b|)$.
    It is thus enough to prove that $P(j) \geq P(\sigma^{-1}_j(j))$ and $Q(j) \geq Q(\sigma_j(j))$.
    This follows from the fact that we have $\sigma_j(i) = i$ for all $i < j$ and thus $j \leq \sigma_j^{-1}(j)$ and $j \leq \sigma_j(j)$.
\end{proof}

\begin{proof}
\textbf{(of \cref{lemma:stability-prominence})} 
    Let $\barc(H) = \{A_0, \dots, A_k\}$ and $\barc(E) = \{B_0, \dots, B_m\}$, such that $\prom(H)(i) = \length(A_i)$ for $0 \leq i \leq k$ and $\prom(E) = \length(B_i)$ for $0 \leq i \leq m$.
    By \cref{bottleneck-correspondence-interleaving}, there exists an $\epsilon$-matching $f : \{0, \dots, k\} \to \{0, \dots, m\}$ between $\barc(H)$ and $\barc(E)$.
    We can thus extend the matching $f$ to a bijection $f : \bbN \to \bbN$ and any such extension has the property that $|\prom(H)(i) - \prom(E)(f(i))| \leq 2\epsilon$.
    This is because, if $i \in \{0, \dots, m\}$ is in the domain of $f$, then $|\length(A_i) - \length(B_{f(i)})| \leq 2\epsilon$, if $i \in \{0, \dots, m\}$ is not in the domain of $f$, then $|\length(A_i)| \leq 2\epsilon$, and if $j \in \{0, \dots, m\}$ is not in the codomain of $f$, then $|\length(B_i)| \leq 2\epsilon$.
    Now, the result follows from \cref{lemma:best-matching-is-identity}.
\end{proof}

\begin{proof}
\textbf{(of \cref{lemma:stability-prominence-essentially-finite})} 
By \cref{lemma:stability-prominence}, 
$d_\infty(\prom(H_{\geq \tau}), \prom(E_{\geq \tau})) \leq 2 \, \dWI(H_{\geq \tau},E_{\geq \tau})$ 
for any $\tau > 0$. 
By \cref{stability-persistence-pruning}, 
$\dWI(H_{\geq \tau},E_{\geq \tau}) \leq \dWI(H,E)$. 
So, $d_\infty(\prom(H_{\geq \tau}), \prom(E_{\geq \tau})) \leq 2 \, \dWI(H,E)$. 
By definition, as $\tau \to 0$, 
we have $d_\infty(\prom(H_{\geq \tau}), \prom(H)) \to 0$. 
So, letting $\tau \to 0$ and using the triangle inequality, 
we have $d_\infty(\prom(H), \prom(E)) \leq 2 \, \dWI(H,E)$.
\end{proof}

\subsection{Details from \cref{Pruning-and-flattening}} \label{appendix-simplification}


\begin{proposition}\label{stability-persistence-pruning}
    Let $H$ and $E$ be $\bbR$-hierarchical clusterings of sets $X$ and $Y$ respectively, and let $\tau \geq 0$.
    The hierarchical clusterings $H_{\geq \tau}$ and $H$ are $\tau$-interleaved
    and, if $H$ and $E$ are $\epsilon$-interleaved with respect to a correspondence $R \subseteq X \times Y$,
    then $H_{\geq \tau}$ and $E_{\geq \tau}$ are $\epsilon$-interleaved with respect to $R$.
\end{proposition}

\begin{proof}
The fact that $H_{\geq \tau}$ and $H$ are $\tau$-interleaved follows immediately from the definitions. 
We show that $H_{\geq \tau}$ and $E_{\geq \tau}$ are $\epsilon$-interleaved with respect to $R$. 
Let $r \in \bbR$ and let $C \in H_{\geq \tau}(r)$. 
As $H$ and $E$ are $\epsilon$-interleaved with respect to $R$, 
there is $D \in E(r + \epsilon)$ such that $\pi_X^{-1}(C) \subseteq \pi_Y^{-1}(D)$. 
As $C \in H_{\geq \tau}(r)$, there is $C' \in H(r-\tau)$ with $C' \subseteq C$. 
Again by the interleaving, 
there is $D' \in E(r - \tau + \epsilon)$ such that $\pi_X^{-1}(C') \subseteq \pi_Y^{-1}(D')$. 
It follows that $D' \subseteq D$, and thus $D \in E_{\geq \tau}(r+\epsilon)$, as desired.
\end{proof}

\begin{notation} \label{notation-R_X}
Let $H$ and $E$ be $\bbR$-hierarchical clusterings of sets $X$ and $Y$ respectively. 
Let $\epsilon \geq 0$, and let $R \subseteq X \times Y$ be a correspondence such that 
$H$ and $E$ are $\epsilon$-interleaved with respect to $R$. 
For $r \in \bbR$, we write $R_X : H(r) \to E(r + \epsilon)$ 
for the function such that $\pi_X^{-1}(A) \subseteq \pi_Y^{-1}(R_X (A))$ 
for all $A \in H(r)$.
\end{notation}

An \define{interval} of a poset $P$ (\cref{poset-def}) consists of a subset 
$I \subseteq P$ such that whenever we have $x \preceq y \preceq z \in P$
with $x,z \in I$, we also have $y \in I$. 
A subset $T$ of a poset $P$ is \define{totally-ordered} if, 
for all $x, y \in T$, we have $x \preceq y$ or $y \preceq x$. 
Let $\inter(P)$ denote the set of totally-ordered intervals of $P$. 
Let $H$ and $E$ be $\bbR$-hierarchical clusterings of sets $X$ and $Y$ respectively, 
and assume there is $\epsilon \geq 0$ such that $H$ and $E$ are $\epsilon$-interleaved with respect to
a correspondence $R \subseteq X \times Y$. 
Define a function $i_X : \PC(H) \to \inter(\PC(E))$ by mapping
a persistent cluster $\bfC$ to the totally-ordered interval 
$\{[R_X(\bfC(r))]\}_{r \in \life(\bfC)}$. 
If $H$ and $E$ are finite, 
then we get an order-preserving function  
$m_X : \PC(H) \to \PC(E)$
by mapping $\bfC$ to $\min(i_X(\bfC))$. 
Note that this depends on $\epsilon$. 

\begin{lemma} \label{m_X-properties}
Let $H$ and $E$ be finite $\bbR$-hierarchical clusterings of sets $X$ and $Y$ respectively, 
and assume there is $\epsilon \geq 0$ such that $H$ and $E$ are $\epsilon$-interleaved with respect to
a correspondence $R \subseteq X \times Y$. 
Let $\bfC \in \PC(H)$.
\begin{enumerate}
    \item We have $\birth(m_X(\bfC)) \leq \birth(\bfC) + \epsilon$. \label{part-a}
    \item Assume all the leaves of $H$ and $E$ have length strictly greater than $2\epsilon$.
        If $\bfC \in \leaves(H)$ and $\bfD \in \leaves(E)$,
        then $\bfD \leq m_X(\bfC)$ if and only if $\bfC \leq m_Y(\bfD)$. \label{part-d}
\end{enumerate}
\end{lemma}
\begin{proof}
For part $(1)$, we can choose $r > \birth(\bfC)$ that is arbitrarily close to $\birth(\bfC)$ 
such that $m_X(\bfC) = [R_X(\bfC(r))]$, and thus $r + \epsilon \in \life(m_X(\bfC))$.

For part $(2)$, say we have $\bfC \in \leaves(H)$ and $\bfD \in \leaves(E)$ 
with $\bfD \leq m_X(\bfC)$. We show that $\bfC \leq m_Y(\bfD)$. 
Choose $r_0 \in \life(\bfC)$ such that $r_0 + 2\epsilon \in \life(\bfC)$ 
and such that $m_X(\bfC) = [R_X(\bfC(r_0))]$. 
As $\bfD \leq m_X(\bfC)$, we can choose $r_1 \leq r_0 + \epsilon$ 
with $r_1 \in \life(\bfD)$ and such that $m_Y(\bfD) = [R_Y(\bfD(r_1))]$. 
We have $\bfD(r_1) \subseteq R_X(\bfC(r_0))$; choose $y \in \bfD(r_1)$. 
If $(x,y) \in R$, then as $y \in R_X(\bfC(r_0))$, 
we have $x \in R_Y(R_X(\bfC(r_0)))$. 
Meanwhile, as $r_0 + 2\epsilon \in \life(\bfC)$, 
we have $\bfC = [R_Y(R_X(\bfC(r_0)))]$, 
and so $x \in U(\bfC)$. 
Now, as $y \in \bfD(r_1)$, 
$x \in R_Y(\bfD(r_1))$, and so $x \in U(m_Y(\bfD))$. 
We have therefore shown that 
$U(\bfC) \cap U(m_Y(\bfD)) \neq \emptyset$, 
and thus $\bfC$ and $m_Y(\bfD)$ are comparable 
in the poset $\PC(H)$, by \cref{lemma:underlying-clusters-are-disjoint}. 
As $\bfC$ is a leaf, we must have $\bfC \leq m_Y(\bfD)$. 
By a symmetric argument, 
$\bfD \leq m_X(\bfC)$ if and only if $\bfC \leq m_Y(\bfD)$.
\end{proof}

\begin{lemma}\label{persistence-flattenings-coincide}
Let $H$ and $E$ be finite $\bbR$-hierarchical clusterings of sets $X$ and $Y$ respectively, 
and assume there is $\epsilon \geq 0$ such that $H$ and $E$ are $\epsilon$-interleaved with respect to
a correspondence $R \subseteq X \times Y$. 
If the leaves of $H$ and $E$ all have length strictly greater than $2\epsilon$, 
then $m_X$ restricts to a bijection $\leaves(H) \to \leaves(E)$ such that 
$\bfC$ and $m_X(\bfC)$ are $\epsilon$-interleaved with respect to $R$ for every $\bfC \in \leaves(H)$.
\end{lemma}

\begin{proof}
Let $\bfC \in \leaves(H)$. 
We start by proving that $m_X(\bfC)$ is a leaf; 
a symmetric argument shows that $m_Y$ sends leaves to leaves. 
Let $\bfD \in \leaves(E)$ with $\bfD \leq m_X(\bfC)$; 
we have $\bfC \leq m_Y(\bfD)$, by \cref{m_X-properties} (2). 
Towards a contradiction, say $\bfD \neq m_X(\bfC)$. 
As $\length(\bfD) > 2\epsilon$, we have 
$\birth(\bfD) + 2\epsilon < \birth(m_X(\bfC)) 
\leq \birth(\bfC) + \epsilon 
\leq \birth(m_Y(\bfD)) + \epsilon$. 
So, we have 
$\birth(\bfD) + \epsilon < \birth(m_Y(\bfD))$, 
which contradicts \cref{m_X-properties} (1). 
So, $\bfD = m_X(\bfC)$, and thus $m_X(\bfC)$ is a leaf. 
It follows that we have $m_Y(m_X(\bfC)) = \bfC$ for any leaf $\bfC$, 
since $\bfC \leq m_Y(m_X(\bfC))$ and $\bfC$ and $m_Y(m_X(\bfC))$ are both leaves.    
Together with a symmetric argument, 
this shows that $m_X$ and $m_Y$ restrict to inverse bijections on leaves. 
The proof that $\bfC$ and $m_X(\bfC)$ are $\epsilon$-interleaved is straightforward.
\end{proof}

Next, we need a lemma that describes the barcode of $H_{\geq \tau}$ in terms of the barcode of $H$. 
Let $H$ be a pointwise finite $I$-hierarchical clustering with $I \subseteq \bbR$ an interval, 
and let $\barc(H) = \{\barA_j\}_{j \in J}$. 
For $\tau > 0$, let $\barc(H)_\tau = \{\tilde{\barA}_j\}_{j \in \tilde{J}}$, 
where for any $j \in J$, $\tilde{\barA}_j \subset \barA_j$ 
is the sub-interval $\tilde{\barA}_j = \{ x \in \barA_j : x - \tau \in \barA_j \}$, 
and $\tilde{J} \subseteq J$ consists of $j$ such that $\tilde{\barA}_j \neq \emptyset$.

\begin{lemma} \label{barcode-of-pruning}
Let $H$ be a pointwise finite $I$-hierarchical clustering with $I \subseteq \bbR$ an interval. 
For any $\tau > 0$, $\barc(H_{\geq \tau}) = \barc(H)_\tau$.
\end{lemma}

\begin{proof}
We need to check that, for all $r \leq r'$, we have
$\rk(H_{\geq \tau}) = | \{ j \in \tilde{J} : r,r' \in \tilde{\barA}_j \} |$. 
As $H_{\geq \tau}(r) = \Im \, H(r-\tau \leq r)$, 
we have $\Im \, H_{\geq \tau} (r \leq r') = \Im \, H (r-\tau \leq r')$. 
And by definition of the barcode, we have 
$|\Im \, H (r-\tau \leq r')| = | \{ j \in J : r-\tau,r' \in \barA_j \} |$. 
Now, for any $j \in J$, we have $r-\tau, r' \in \barA_j$ if and only if we have $r,r' \in \tilde{\barA}_j$. 
So, $| \{ j \in J : r-\tau,r' \in \barA_j \} | = | \{ j \in \tilde{J} : r,r' \in \tilde{\barA}_j \} |$. 
This finishes the proof.
\end{proof}

\begin{lemma} \label{length-leaf-inequality}
Let $H$ be an essentially finite $I$-hierarchical clustering with $I \subseteq \bbR$ an interval. 
Let $\gapindex \geq 1$ and assume $\gap_\gapindex(H)$ is non-empty. 
If $\tau \in \gap_\gapindex(H)$, 
then for every $\bfC \in \leaves(H_{\geq \tau})$, 
we have $\length(\bfC) \geq \prom(H)(\gapindex - 1) - \tau$.
\end{lemma}

\begin{proof}
First, say $H$ is finite. Then 
$\min_{\bfC \in \leaves(H_{\geq \tau})} \length(\bfC) = \min_{B \in \barc(H_{\geq \tau})} \length(B)$ 
by \cref{corollary:leaf-bar-length}. 
As $H$ is finite, it is pointwise finite, so we can apply \cref{barcode-of-pruning}. 
Let $(\ell_0, \dots, \ell_k)$ be the lengths of the intervals in $\barc(H) = \{\barA_j\}_{j \in J}$ 
ordered from largest to smallest, as in \cref{definition:prom-finite}. 
We have $\barc(H_{\geq \tau}) = \{\tilde{\barA}_j\}_{j \in \tilde{J}}$, 
and thus the lengths of the intervals in $\barc(H_{\geq \tau})$, ordered from largest to smallest, 
are $(\tilde{\ell}_0, \dots, \tilde{\ell}_{\gapindex-1})$, 
where $\tilde{\ell}_p = \ell_p - \tau$. Thus, 
$\min_{B \in \barc(H_{\geq \tau})} \length(B) = \tilde{\ell}_{\gapindex-1} = \ell_{\gapindex-1} - \tau 
    = \prom(H)(\gapindex - 1) - \tau$. 
Now we consider the case where $H$ is essentially finite. 
Let $0 < \sigma < \tau$. As $\sigma \to 0$, we have 
$\prom(H_{\geq \sigma})(\gapindex - 1) \to \prom(H)(\gapindex - 1)$ and 
$\prom(H_{\geq \sigma})(\gapindex) \to \prom(H)(\gapindex)$. 
So, if $\sigma$ is small enough, $\gap_\gapindex(H_{\geq \sigma})$ is non-empty and 
$\tau \in \gap_\gapindex(H_{\geq \sigma})$. 
We have $H_{\geq \tau} = (H_{\geq \sigma})_{\geq \tau - \sigma}$. 
Applying the finite case, we have, for every $\bfC \in \leaves(H_{\geq \tau})$, 
$\length(\bfC) \geq \prom(H_{\geq \sigma})(\gapindex - 1) - (\tau - \sigma)$. 
As $\sigma \to 0$, we have 
$\prom(H_{\geq \sigma})(\gapindex - 1) - (\tau - \sigma) \to \prom(H)(\gapindex - 1) - \tau$, 
which finishes the proof.
\end{proof}

\begin{proof}
\textbf{(of \cref{stability-PF-in-tau})}
Without loss of generality, $\tau < \tau'$. 
Let $\bfC \in \leaves(H_{\geq \tau})$. 
By \cref{length-leaf-inequality}, $\length(\bfC) \geq \prom(H)(\gapindex-1) - \tau$, 
and thus $\length(\bfC) > \tau' - \tau$. 
So, there is $r \in \bbR$ such that, with $C = \bfC(r)$, 
$C' := H(r < r + (\tau' - \tau))(C) \in \bfC$. 
By construction, $C' \in H_{\geq \tau'}(r + (\tau' - \tau))$. 
It is easy to check that $[C'] \in \leaves(H_{\geq \tau'})$. 
Define $m$ by setting $m(\bfC) = [C']$. 
We show $m$ is injective. Let $\bfC, \bfD \in \leaves(H_{\geq \tau})$, 
and say $[C'] = [D']$, using the notation from above. 
Without loss of generality, $C' \subseteq D'$. 
Then $U(\bfC) \cap U(\bfD) \neq \emptyset$, 
so by \cref{lemma:underlying-clusters-are-disjoint}, 
$\bfC \leq \bfD$ or $\bfD \leq \bfC$. 
As $\bfC$ and $\bfD$ are leaves, it follows that $\bfC = \bfD$.

We show $m$ is a bijection by showing $|\leaves(H_{\geq \tau})| = |\leaves(H_{\geq \tau'})|$. 
By \cref{proposition:elder-rule}, for any finite $\bbR$-hierarchical clustering $E$, 
$|\leaves(E)| = |\barc(E)|$. 
By an argument like the last step of the proof of \cref{length-leaf-inequality}, 
we have $|\barc(H_{\geq \tau})| = |\barc(H_{\geq \tau'})|$. 
Thus, $|\leaves(H_{\geq \tau})| = |\leaves(H_{\geq \tau'})|$ 
and therefore $m$ is a bijection. 
It is easy to check that for any $\bfC \in \leaves(H_{\geq \tau})$, 
$U(\bfC) = U(m(\bfC))$.
\end{proof}


\begin{proof}
\textbf{(of \cref{stability-tomato-flattening})}
By definition, $\PF(H, \gapindex) = \leaves(H_{\geq \tau_H})$, 
where $\tau_H = (\prom(H)(\gapindex-1) + \prom(H)(\gapindex))/2$, and 
$\PF(E, \gapindex) = \leaves(E_{\geq \tau_E})$, 
where $\tau_E = (\prom(E)(\gapindex-1) + \prom(E)(\gapindex))/2$. 
By \cref{stability-persistence-pruning}, 
$H_{\geq \tau_H}$ and $E_{\geq \tau_H}$ are $\epsilon$-interleaved with respect to $R$, 
and $E_{\geq \tau_H}$ and $E_{\geq \tau_E}$ are $|\tau_H - \tau_E|$-interleaved. 
So, $H_{\geq \tau_H}$ and $E_{\geq \tau_E}$ are 
$(\epsilon + |\tau_H - \tau_E|)$-interleaved with respect to $R$. 
By \cref{lemma:stability-prominence-essentially-finite}, 
$d_\infty(\prom(H), \prom(E)) \leq 2 \, \dWI(H,E) \leq 2\epsilon$. 
So, $|\tau_H - \tau_E| \leq 2\epsilon$, and thus $H_{\geq \tau_H}$ and $E_{\geq \tau_E}$ are 
$3\epsilon$-interleaved with respect to $R$.

For any $\bfC \in \leaves(H_{\geq \tau_H})$, \cref{length-leaf-inequality} implies that 
$\length(\bfC) \geq \prom(H)(\gapindex - 1) - \tau_H = (\prom(H)(\gapindex - 1) - \prom(H)(\gapindex))/2 
> 8\epsilon$. 
Similarly, for any $\bfD \in \leaves(E_{\geq \tau_E})$, we have 
$\length(\bfD) \geq \prom(E)(\gapindex - 1) - \tau_E = (\prom(E)(\gapindex - 1) - \prom(E)(\gapindex))/2 
\geq ((\prom(H)(\gapindex - 1) - \prom(H)(\gapindex))/2) - 2\epsilon > 6\epsilon$. 
As $H_{\geq \tau_H}$ and $E_{\geq \tau_E}$ are finite hierarchical clusterings that are 
$3\epsilon$-interleaved with respect to $R$, 
and the leaves of $H_{\geq \tau_H}$ and $E_{\geq \tau_E}$ all have length strictly greater than $6\epsilon$, 
\cref{persistence-flattenings-coincide} implies that there is a bijection 
$m : \PF(H, \gapindex) \to \PF(E, \gapindex)$ 
such that for all $\bfC \in \PF(H, \gapindex)$, 
$\bfC$ and $m(\bfC)$ are $3\epsilon$-interleaved with respect to $R$.
\end{proof}

\begin{remark} \label{exhaustive-PF-ToMATo}
We describe the relationship between {\sc ExhaustivePF} (\cref{tomato-flattening-algorithm}) 
and ToMATo \citep[Algorithm 1]{chazal-guibas-oudot-skraba}. 
In line 10 of {\sc ExhaustivePF} we use $\leq$ rather than $<$ in order to follow the behavior of the persistence-based flattening. 
In this remark, we consider the version of {\sc ExhaustivePF} that uses $<$ in line 10. 
Say given the input $(G,\tilde{f},\tau)$ to ToMATo. 
Extend the filtering function to edges, by setting 
$\tilde{f}(\{x,y\}) = \min\{\tilde{f}(x),\tilde{f}(y)\}$. 
Reverse the filtration, by setting 
$f(\sigma) = -\tilde{f}(\sigma)$ for $\sigma \in G$.
Order the simplices of $G$ as follows. 
Begin with any ordering of the vertices $x_1, \dots, x_q$ such that $f(x_i) \leq f(x_{i+1})$. 
For each $i$, insert directly after $x_i$ 
those edges $\{x_j, x_i\}$ with $j < i$, ordered by $j$. 
Now, the output of {\sc ExhaustivePF} on this input 
agrees with the output of ToMATo on $(G,\tilde{f},\tau)$, 
with one exception: 
ToMATo excludes clusters $C$ such that $\max_{x \in C} \tilde{f}(x) < \tau$.
\end{remark}

\subsection{Details from \cref{Persistable}} \label{appendix-Persistable}

\begin{table*}[b!]\centering
\begin{tabular}{|c|c|c|c|}
\hline
$k$ & Memory (GB) & Time (s) & Number of clusters \\
\hline
$20$ & 0.46 & 65 & 6524\\\hline
$40$ & 0.82 & 58 & 2618\\\hline
$80$ & 1.54 & 64 & 1057\\\hline
$160$ & 2.99 & 68 & 473\\\hline
$320$ & 5.88 & 79 & 230\\\hline
\end{tabular}
\caption{Evaluation of the HDBSCAN implementation of \cite{mcinnes-healy-astels-joss} 
on the Rideshare data set. 
The columns are the algorithm parameter $k = \mathsf{min\_samples}$, 
the peak memory usage in GB, the runtime in seconds, and the number of clusters. 
The algorithm parameter 
$\mathsf{min\_cluster\_size}$ is set equal to $\mathsf{min\_samples}$ 
(the default). 
The implementation includes several algorithms for computing HDBSCAN. 
We use the Dual-Tree Bor\r{u}vka algorithm with kd-trees 
(the default choice for low-dimensional Euclidean data, as in this case), 
which is the only choice that gives reasonable performance on this data set. 
Memory usage scales linearly with $k$ 
because the algorithm stores the $k$ nearest neighbors of each data point.
}
\vspace*{2cm}
\label{hdbscan-rideshare-performance-table}
\end{table*}

%

\begin{table*}\centering
\begin{tabular}{|c||c|c|c||c|}
\hline
 & 1 & 2 & 3 & noise \\
\hline
South Italy & 293 & & & 30\\\hline
Sardinia & & 97 & & 1\\\hline
North Italy & & & 118 & 33\\\hline
\end{tabular}
\caption{A confusion matrix comparing the Persistable clustering 
(clusters 1--3 and points labeled as noise), 
and the large area labels of the Olive oil data. 
The adjusted Rand index is $1.0$, and $89\%$ of the data is clustered.}
\label{oliveoil-3-clusters-table}
\end{table*}

\begin{table*}\centering
\begin{tabular}{|c|c||c|c|c|c|c|c|c|c||c|}
\hline
 & & 1 & 2 & 3 & 4 & 5 & 6 & 7 & 8 & noise \\
\hline
\multirow{4}{*}{South Italy} & North Apulia & 12 & & & & & & & & 13\\\cline{2-11}
& Calabria & & 7 & 1 & & & & & & 48 \\\cline{2-11}
& South Apulia & & & 100 & & & & & & 106 \\\cline{2-11}
& Sicily & 3 & & & & & & & & 33 \\
\hline
\multirow{2}{*}{Sardinia} & Inland Sardinia & & & & 51 & & & & & 14 \\\cline{2-11}
& Coast Sardinia & & & & & 19 & & & & 14 \\
\hline
\multirow{3}{*}{North Italy} & East Liguria & & & & & & 14 & 1 & & 35 \\\cline{2-11}
& West Liguria & & & & & & & 29 & & 21 \\\cline{2-11}
& Umbria & & & & & & & & 42 & 9 \\
\hline
\end{tabular}
\caption{A confusion matrix comparing the Persistable clustering 
(clusters 1--8 and points labeled as noise), 
and the region labels of the Olive oil data. 
The adjusted Rand index is $0.98$, and $49\%$ of the data is clustered.}
\label{oliveoil-8-clusters-table}
\end{table*}

\newpage

\begin{table*}\centering
\begin{tabular}{|c|c||c|c|c|c|c|c|c|c||c|}
\hline
 & & 1 & 2 & 3 & 4 & 5 & 6 & 7 & 8 & noise \\
\hline
\multirow{4}{*}{South Italy} & North Apulia & 21 & 2 & & & & & & & 2\\\cline{2-11}
& Calabria & & 50 & 3 & & & 1 & & & 2\\\cline{2-11}
& South Apulia & & 1 & 196 & & & & & & 9\\\cline{2-11}
& Sicily & 6 & 20 & 7 & & & & & & 3\\
\hline
\multirow{2}{*}{Sardinia} & Inland Sardinia & & & & 65 & & & & & \\\cline{2-11}
& Coast Sardinia & & & & 2 & 31 & & & & \\
\hline
\multirow{3}{*}{North Italy} & East Liguria & & & & & & 41 & 2 & & 7\\\cline{2-11}
& West Liguria & & & & & & & 47 & & 3 \\\cline{2-11}
& Umbria & & & & & & 2 & & 48 & 1 \\
\hline
\end{tabular}
\caption{A confusion matrix comparing the Persistable clustering, 
using the exhaustive persistence-based flattening 
(clusters 1--8 and points labeled as noise), 
and the region labels of the Olive oil data. 
The adjusted Rand index is $0.90$, and $95\%$ of the data is clustered.}
\label{oliveoil-8-clusters-exhaustivePF-table}
\end{table*}

\begin{figure}[p]
    \centering
    \includegraphics[width=0.5\textwidth]{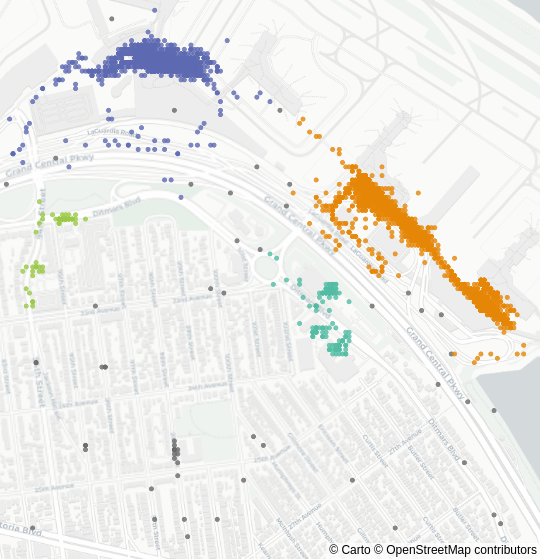}
    \caption{The Persistable clustering of the Rideshare data 
    using the slice in \cref{Persistable-laguardia-visualizations}, 
    choosing the gap below the fourth vine. 
    Gray points are unclustered.}
    \label{appendix-Persistable-laguardia-visualizations}
\end{figure}

\clearpage
\bibliography{ref}


\end{document}